\documentclass[reqno,12pt]{amsart}

\usepackage{latexsym,amsfonts,amssymb,amsmath,euscript}
\usepackage{graphicx,mathrsfs}
\usepackage{hyperref}
\usepackage{color,fancybox}

\newcommand{\C}{\mathbb{C}}
\newcommand{\R}{\mathbb{R}}
\newcommand{\N}{\mathbb{N}}

\newcommand{\eps}{\varepsilon}
\newcommand{\delim}[3]{\left#1 #3 \right#2}

\newcommand{\norma}[1]{\delim{\|}{\|}{#1}}
\newcommand{\set}[1]{\delim{\{}{\}}{#1}}

\begin{document}


\title[Parabolic problems in oscillatory thin domains]{Parabolic problems in highly oscillating thin domains}

\author[M. C. Pereira]{Marcone C. Pereira$^\dagger$}

\thanks{$^\dagger$Partially
supported by CNPq 302847/2011-1, CAPES/DGU 267/2008 and FAPESP 2008/53094-4, Brazil}

\address[M. C. Pereira]{Escola de Artes, Ci\^encias e Humanidades,
Universidade de S\~ao Paulo, Rua Arlindo B\'ettio, 03828-000 S\~ao Paulo SP, Brazil}

\email{marcone@usp.br}



\maketitle
\numberwithin{equation}{section}
\newtheorem{theorem}{Theorem}[section]
\newtheorem{lemma}[theorem]{Lemma}
\newtheorem{corollary}[theorem]{Corollary}
\newtheorem{proposition}[theorem]{Proposition}
\newtheorem{definition}[theorem]{Definition}
\newtheorem{remark}[theorem]{Remark}
\allowdisplaybreaks

\begin{abstract} 

In this work we consider the asymptotic behavior of the nonlinear semigroup defined by a semilinear parabolic problem with homogeneous Neumann boundary conditions posed in a region of $\R^2$ that degenerates into a line segment when a positive parameter $\epsilon$ goes to zero (a \emph{thin domain}).
Here we also allow that its boundary presents highly oscillatory behavior with different orders and variable profile.
We take thin domains possessing the same order $\epsilon$ to the thickness and amplitude of the oscillations but assuming different order to the period of oscillations on the top and the bottom of the boundary. 
We combine methods from linear homogenization theory and the theory on nonlinear dynamics of dissipative systems  to  obtain the limit problem establishing convergence properties for the solutions. At the end we show the upper semicontinuity of the attractors and stationary states.

\end{abstract}

\section{Introduction}

In this paper we are interested in analyzing the asymptotic behavior of the solutions of a semilinear parabolic problem with homogeneous Neumann boundary condition in a thin domain $R^\epsilon$ with a highly oscillatory behavior in its boundary as illustrated in Figure \ref{ThinDomain-1}.


\begin{figure}[h]\label{ThinDomain-1}
\centering \scalebox{1.1}{\includegraphics{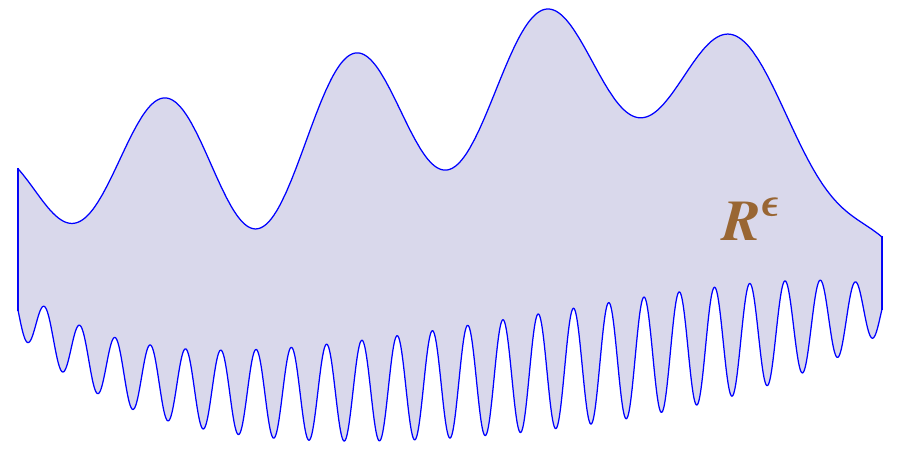}}
\caption{Thin domain with a highly oscillatory boundary.}
\end{figure}

Let $G_\epsilon$, $H_\epsilon: (0,1) \mapsto (0,\infty)$ be two positive smooth functions satisfying  
$0< G_0 \le G_\epsilon(x) \le G_1$ and $0 < H_0 \le H_\epsilon(x) \le H_1$ for all $x \in (0,1)$ and $\epsilon > 0$, where $G_0$, $G_1$, $H_0$ and $H_1$ are constants independent of $\epsilon$, and consider the bounded open region $R^\epsilon$ given by
\begin{equation} \label{TDG}
R^\epsilon = \{ (x,y) \in \R^2 \; | \;  x\in (0,1) \textrm{ and } - \epsilon \, G_\epsilon(x) < y< \epsilon \, H_\epsilon(x)  \}.
\end{equation}
Note that functions $G_\epsilon$ and $H_\epsilon$ define the lower and upper boundary of the 2-dimensional thin domain $R^\epsilon$ with order of thickness $\epsilon$. 
Here we allow $G_\epsilon$ and $H_\epsilon$ to present different orders and profiles of oscillations. The upper boundary established by $\epsilon \, H_\epsilon$ present same order of amplitude, period and thickness, but, the lower boundary given by $\epsilon \, G_\epsilon$ possess oscillation order larger than the compression order $\epsilon$ of the thin domain. 
We express this assuming that
$$
\begin{gathered}
G_\epsilon(x) = G(x,x/\epsilon^\alpha), \quad \alpha > 1, \\
\textrm{ and } \quad 
H_\epsilon(x) = H(x,x/\epsilon), 
\end{gathered}
$$
where the functions $G$, and $H:[0,1] \mapsto (0,\infty)$ are smooth functions with $y \to G(x,y)$ and $y \to H(x,y)$ periodic in variable $y$ with constant period $l_g$ and $l_h$ respectively. 


In the thin domain $R^{\epsilon}$ we look at the semilinear parabolic evolution equation 
\begin{equation} \label{BPO}
\left\{
\begin{split}
& w^\epsilon_t - \Delta w^\epsilon + w^\epsilon = f(w^\epsilon),
\quad \textrm{ in } R^\epsilon, \\
&\frac{\partial w^\epsilon}{\partial \nu^\epsilon} = 0
\quad \textrm{ on } \partial R^\epsilon, \end{split}
\right. t>0,
\end{equation}
where $\nu^\epsilon$ 
is the unit outward normal to $\partial R^\epsilon$,  $\frac{\partial }{\partial \nu^\epsilon}$ is the outwards normal derivative  
 and the function $f:\R \mapsto \R$ is a $\mathcal{C}^2$-function with bounded derivatives. 
Since we are interested in the behavior of solutions as $t\to \infty$ and its dependence with respect to the small parameter $\eps$, we require
that the solutions of \eqref{BPO} are bounded for large values of time.  A natural assumption to obtain this boundedness is given by the following dissipative
condition
\begin{equation} \label{HF}
\limsup_{|s| \to \infty} \frac{f(s)}{s} < 0.
\end{equation}

From the point of view of investigating the asymptotic dynamics assuming $f$ with bounded derivatives does not imply any restriction since we are interested in dissipative nonlinearities. Indeed, it follows from \cite{ACL1, ACB} that under the usual growth assumptions, the attractors are uniformly bounded in $L^\infty(\Omega^\epsilon)$ with respect to $\epsilon$ and we may cut the nonlinearities in a suitable way making them bounded with bounded derivatives.
Recall that an attractor is a compact invariant set which attracts all bounded sets of the phase space. It contains all the asymptotic dynamics of the system and all global bounded solutions lie in the attractor.

In order to analyze problem \eqref{BPO} and its related linear elliptic and parabolic problem we first perform a simple  change of variables which consists in stretching in the
$y$-direction by a factor of $1/\epsilon$. As in \cite{HR,R,PR01}, we use $x_1=x, x_2=y/\epsilon$ to transform $R^\epsilon$ into the domain 
\begin{equation} \label{domain}
\Omega^\epsilon = \{ (x_1,x_2) \in \R^2 \; | \;  x_1 \in (0,1) \textrm{ and } - G_\epsilon(x_1) < x_2 < H_\epsilon(x_1) \}.
\end{equation}
By doing so, we obtain a domain which is not thin anymore although it presents very highly oscillatory behavior given by the fact that the upper and lower boundary
are the graph of the oscillating functions $G_\epsilon$ and $H_\epsilon$.  
Under this change, equation \eqref{BPO} is transformed into 
\begin{equation}\label{RP} 
\left\{
\begin{split}
& u^\epsilon_{t} - \frac{\partial^2 u^\epsilon}{{\partial x_1}^2} - 
\frac{1}{\epsilon^2} \frac{\partial^2 u^\epsilon}{{\partial x_2}^2} + u^\epsilon = f(u^\epsilon)
\quad \textrm{ in } \Omega^\epsilon \\
& \frac{\partial u^\epsilon}{\partial x_1} N_1^\epsilon + \frac{1}{\epsilon^2} \frac{\partial u^\epsilon}{\partial x_2}N_2^\epsilon = 0
\quad \textrm{ on } \partial \Omega^\epsilon 
\end{split}
\right. t>0,
\end{equation}
where $N^\eps=(N^\eps_1,N^\eps_2)$ is the outward normal to the boundary of $\Omega^\epsilon$. 

Observe the factor $1/\eps^2$ in front of the derivative in the $x_2$ direction which means a very fast diffusion in the vertical direction.  In some
sense, we have substituted the thin domain $R^\eps$ with a non thin domain $\Omega^\eps$ but with a very strong diffusion mechanism in the $x_2$-direction.   
Because of the presence of this very strong diffusion mechanism it is expected that solutions of \eqref{RP}  to become homogeneous in the $x_2$-direction so that the limiting solution will not have a dependence in this direction and therefore the limiting problem will be one dimensional. This fact is in agreement with the intuitive idea that an equation in a thin domain should approach an equation in a line segment. 

We get the following limit problem to \eqref{RP} as $\epsilon$ goes to zero:
\begin{equation}\label{LRP} 
\left\{
\begin{split}
& u^\epsilon_{t} - \frac{1}{p(x)} \left( q(x) \, u_x \right)_x + u = f(u),
\quad x \in (0,1), \\
& u_x(0) = u_x(1) = 0, 
\end{split}
\right. \quad t>0,
\end{equation}
where the smooth positive functions $p$ and $q:(0,1) \mapsto (0,\infty)$ are given by 
$$
\begin{gathered}
q(x) =  \frac{1}{l_h} \int_{Y^*(x)} \left\{ 1 - \frac{\partial X(x)}{\partial y_1}(y_1,y_2) \right\} dy_1 dy_2, \\ 
p(x) = \frac{|Y^*(x)|}{l_h} + \frac{1}{l_g} \int_0^{l_g} G(x,y) \, dy - G_0(x), \\
G_0(x) = \min_{y \in \R} G(x,y),
\end{gathered}
$$
and $X(x)$ is the unique solution of the problem 
$$
\left\{
\begin{array}{l}
- \Delta X(x)  =  0  \textrm{ in } Y^*(x)  \\
\frac{\partial X(x)}{\partial N}  =  0  \textrm{ on } B_2(x)  \\
\frac{\partial X(x)}{\partial N}  =  N_1 \textrm{ on } B_1(x)  \\
X(x) \textrm{ $l_h$-periodic on } B_0(x) \\
\int_{Y^*(x)} X(x) \; dy_1 dy_2  =  0  
\end{array}
\right.
$$
in the representative cell $Y^*(x)$ given by
$$
Y^*(x) = \{ (y_1,y_2) \in \R^2 \; | \; 0< y_1 < l_h, \quad -G_0(x) < y_2 < H(x,y_1) \}, 
$$
where $B_0(x)$, $B_1(x)$ and $B_2(x)$ are lateral, upper and lower
boundary of $\partial Y^*(x)$ for $x \in (0,1)$.

If the nonlinearity $f$ satisfies the dissipative conditions \eqref{HF}, then both equations \eqref{RP} and \eqref{LRP} define nonlinear semigroups that possess global attractors $\mathscr{A}_\epsilon \subset H^1(\Omega^\epsilon)$ and $\mathscr{A}_0\subset H^1(0,1)$ respectively. Here in this work we get the continuity of the nonlinear semigroup, as well as, the upper semicontinuity of the family of the attractors $\mathscr{A}_\epsilon$ and the equilibria set at $\epsilon = 0$ obtaining convergence properties for the dynamics set up by problems \eqref{RP} and \eqref{LRP}.

There are several works in the literature dealing with partial differential equations in thin domains presenting oscillating boundaries. We mention \cite{MP2,MP} who studied the asymptotic approximations of solutions to parabolic and elliptic problems in thin perforated domain with rapidly varying thickness, and  \cite{BGG,BGM,Apl3} who consider nonlinear monotone problems in a multidomain with a highly oscillating boundary. 
In addiction, we also cite \cite{AB,BZ, BFF}, in which the asymptotic description of nonlinearly elastic thin films with fast-oscillating profile was successfully obtained in a context of \emph{$\Gamma$-convergence}  \cite{Dal}.

Recently we have considered in \cite{ACPS,AP,AP2,AP3,PS} many classes of oscillating thin domains discussing limit problems and convergence properties. 
We also mention \cite{AVP} who deal with a linear elliptic problem in a thin domain presenting doubly oscillatory behavior which is related to the present one studied here but with constant profile, that means, assuming $G_\epsilon(x)=g(x/\epsilon)$ and $H_\epsilon(x)=h(x/\epsilon)$ for some periodic functions $g$ and $h$. This situation is some times called as \emph{purely periodic case}.
Our goal here is to consider a semilinear parabolic problem in $R^\epsilon$ also presenting doubly oscillatory behavior but now with variable profile generally called \emph{locally periodic case}.
We allow much more complicated shapes combining oscillating orders establishing the limit problem, as well as, its dependence with respect to the thin domain geometry. 
Indeed, we get an explicit relationship among the limit equation, the oscillation, the profile and thickness of the thin domain.

It is worth observing that is not an easy task. In order to do so, we first need to combine different techniques introduced  in \cite{AP2, AP3} and \cite{AVP} to investigate the linear elliptic problem. We use \emph{extension operators} and \emph{oscillating test functions} from \emph{homogenization theory} with \emph{boundary perturbation} results to obtain the limit problem for the elliptic equation. 
Next we apply the \emph{theory of dissipative systems and attractors} to be able to obtain  the continuity of the nonlinear semigroup and the upper semicontinuity of the attractors and stationary states of the parabolic problem here proposed.

We refer to \cite{BLP,CD,CP,SP,Tt} and \cite{ACP, Hale, Henry, PP} for a general introduction to the homogenization theory and the theory of dissipative systems and attractors respectively.    
There are not many results on the behavior of global attractors of dissipative  systems under a perturbation related to homogenization. We would like to cite \cite{BCP,BCK, Fiedler-Vishik-2001, Fiedler-Vishik-2003}.

Finally, we point out that thin structures with rough contours (thin rods,  plates or shells) or fluids filling out thin domains (lubrication) or even chemical diffusion process in the presence of grainy narrow strips (catalytic process) are very common in engineering and applied science.  The analysis of the properties of these structures and the processes taking place on them and understanding how the micro geometry of the thin structure affects the macro properties of the material is a very relevant issue in engineering  and material design.  
Thus,  being able to obtain the limiting equation of a prototype equation in different structures where the micro geometry is not necessarily smooth and being able to analyze how the different micro scales affects the limiting problem goes in this direction and will allow the study and understanding in more complicated situations. 
See \cite{Apl3,Apl1,Apl2,Apl4} for some concrete applied problems.

This paper is organized as follows.
In Section \ref{SBF} we set up the notation and state some technical results which will be used later in the proofs. 
In Section \ref{PWPC} we investigate the linear elliptic problem on thin domains assuming also that $G_\epsilon$ and $H_\epsilon$ are \emph{piecewise periodic functions} obtaining Lemma \ref{PPCT}.
Next, in Section \ref{GenPro}, we use Lemma \ref{PPCT} and the continuous dependence result on the domain given by Proposition \ref{BPT} in order to provide a proof of the main result with respect to the linear elliptic problem associated to \eqref{RP}, namely Theorem \ref{ET}. 
In Section \ref{SCLSg} we obtain the continuity of the linear semigroup defined by \eqref{RP} from Theorem \ref{ET}, and in Section \ref{S-USC} we proof the main result of the paper related to the parabolic problem \eqref{RP} getting the upper semicontinuity of the family of attractors and stationary state by Theorem \ref{USCont}.

We also note that although we deal with Neumann boundary conditions,  
we may also consider different conditions in the lateral boundaries of the thin
domain $R^\epsilon$ since we preserve the Neumann type boundary condition in the upper and lower boundary.  
Dirichlet or even Robin homogeneous can be set in the lateral
boundaries of the problem \eqref{RP}. The limit problem will preserve this boundary condition as a point condition.

\section{Basic facts and notations} \label{SBF}

Let us consider two families of positive functions $G_\eps$, $H_\epsilon: (0,1) \to (0,\infty)$, with $\eps\in (0, \eps_0)$ for some $\eps_0>0$ satisfying the following hypothesis

\textbf{(H)} There exist nonnegative constants $G_0$, $G_1$, $H_0$ and $H_1$ such that 
$$
\begin{gathered}
0< G_0 \le G_\epsilon(x) \le G_1 \quad \textrm{ and }  \quad
0< H_0 \le H_\epsilon(x) \le H_1, 
\end{gathered}
$$
for all $x\in (0,1)$ and  $\eps\in (0,\eps_0)$.
Moreover, the functions $G_\eps$ and $H_\epsilon$ are of the type 
\begin{equation}\label{def-G}
\begin{gathered}
G_\eps(x) = G(x,x/\epsilon^\alpha), \quad \textrm{ for some } \alpha>1, \textrm{ and } \quad  H_\eps(x)=H(x,x/\eps), 
\end{gathered}
\end{equation}
where the functions
$H$, $G:[0,1]\times \R \mapsto (0,+\infty)$ are periodic in the second variable, that is, there exist  positive constants $l_g$ and $l_h$ such that $G(x,y+l_g)=G(x,y)$ and $H(x,y+l_h)=H(x,y)$ for all $(x,y) \in [0,1] \times \R$. 
We also suppose $G$ and $H$ are piecewise $C^1$ with respect to the first variable, it means, there exists a finite number of points $0=\xi_0<\xi_1<\cdots<\xi_{N-1}<\xi_N=1$ such that the functions $G$ and $H$ restricted to the set $(\xi_i,\xi_{i+1}) \times \R$ are $C^1$ with $G$, $H$, $G_x$, $H_x$, $G_y$ and $H_y$ uniformly bounded in $(\xi_i,\xi_{i+1}) \times \R$ having limits when we approach $\xi_i$ and $\xi_{i+1}$.

In this work we consider the highly oscillating thin domain $R^\epsilon$ which is defined in \eqref{TDG} as the open set bounded by the graphs of the functions $\epsilon G_\epsilon$ and $\epsilon H_\eps$.
Since we are taking $\alpha > 1$ to define $G_\epsilon$ in \eqref{def-G}, we are allowing the lower boundary of the thin domain $R^\epsilon$ to present a very high oscillatory behavior.
In fact, as $\epsilon \to 0$ we have that the period of the oscillations is much smaller (order $\sim\epsilon^\alpha$) than the amplitude (order $\sim\epsilon$), the height of the thin domain (order $\sim \epsilon$), and period of the oscillations of the upper boundary (order $\sim\epsilon$) given by function $H_\epsilon$.

A function satisfying the above conditions is  
$F(x,y)=a(x)+\sum_{r=1}^N b_r(x)g_r(y)$
where $a$, $b_1$,..,$b_N$ are piecewise $C^1$ with
$g_1$,..,$g_N$ also $C^1$ and $l$-periodic for some $l>0$.

In order to study the dynamics defined by \eqref{BPO} in $R^\epsilon$, we first study the solutions of the linear elliptic equation associated to the equivalent problem introduced by \eqref{RP}. We consider the following elliptic problem with homogeneous Neumann boundary condition
\begin{equation} \label{EP}
\left\{
\begin{gathered}
 - \frac{\partial^2 u^\epsilon}{{\partial x_1}^2} - 
\frac{1}{\epsilon^2} \frac{\partial^2 u^\epsilon}{{\partial x_2}^2} + u^\epsilon = f^\epsilon
\quad \textrm{ in } \Omega^\epsilon \\
\frac{\partial u^\epsilon}{\partial x_1} N_1^\epsilon + \frac{1}{\epsilon^2} \frac{\partial u^\epsilon}{\partial x_2}N_2^\epsilon = 0
\quad \textrm{ on } \partial \Omega^\epsilon 
\end{gathered}
\right.
\end{equation}
where $N^\eps=(N_1^\eps,N_2^\eps)$ is the outward unit normal to $\partial\Omega^\eps$, and $\Omega^\epsilon$ is the oscillating domain \eqref{domain}.
Moreover, we are taking $f^\eps\in L^2(\Omega^\eps)$ satisfying the uniform condition
\begin{equation} \label{FC}
\|f^\eps\|_{L^2(\Omega^\eps)}\leq C, \quad \forall \epsilon > 0,
\end{equation} 
for some $C>0$ independent of $\epsilon$. 
From Lax-Milgran Theorem, we have that  problem (\ref{EP}) has unique solution for each $\epsilon > 0$. 
We first analyze the behavior of these solutions as $\epsilon \to 0$, that is, as the domain gets thinner and thinner although with a high oscillating boundary.

Recall that the equivalence between the problems \eqref{BPO} and \eqref{RP} is established by changing the scale of the domain $R^\epsilon$ through the map $(x,y) \to (x, \epsilon y)$, see \cite{HR} for more details.
Also, the domain $\Omega^\epsilon$ is not thin anymore but presents very wild oscillations at the top and bottom boundary, although the presence of a high diffusion coefficient in front of the derivative with respect the second variable balance the effect of the wild oscillations.


It is known that the variational formulation of (\ref{EP}) is find $u^\epsilon \in H^1(\Omega^\epsilon)$ such that 
\begin{equation} \label{VFP}
\int_{\Omega^\epsilon} \Big\{ \frac{\partial u^\epsilon}{\partial x_1} \frac{\partial \varphi}{\partial x_1} 
+ \frac{1}{\epsilon^2} \frac{\partial u^\epsilon}{\partial x_2} \frac{\partial \varphi}{\partial x_2}
+ u^\epsilon \varphi \Big\} dx_1 dx_2 = \int_{\Omega^\epsilon} f^\epsilon \varphi dx_1 dx_2, 
\, \forall \varphi \in H^1(\Omega^\epsilon).
\end{equation}
Thus we get that the solutions $u^\epsilon$ satisfy an uniform a priori estimate on $\epsilon$.
Indeed, taking $\varphi = u^\epsilon$ in  expression (\ref{VFP}),  we obtain 
\begin{equation} \label{priori}
\begin{gathered}
\Big\| \frac{\partial u^\epsilon}{\partial x_1} \Big\|_{L^2(\Omega^\epsilon)}^2
+ \frac{1}{\epsilon^2}\Big\| \frac{\partial u^\epsilon}{\partial x_2} \Big\|_{L^2(\Omega^\epsilon)}^2
+ \| u^\epsilon \|_{L^2(\Omega^\epsilon)}^2
\le \| f^\epsilon \|_{L^2(\Omega^\epsilon)} \| u^\epsilon \|_{L^2(\Omega^\epsilon)}. 
\end{gathered}
\end{equation}
Consequently, it follows from (\ref{FC}) that 
\begin{equation} \label{EST0}
\begin{gathered}
\| u^\epsilon \|_{L^2(\Omega^\epsilon)}, \Big\| \frac{\partial u^\epsilon}{\partial x_1} \Big\|_{L^2(\Omega^\epsilon)}
\textrm{ and } \frac{1}{\epsilon} \Big\| \frac{\partial u^\epsilon}{\partial x_2} \Big\|_{L^2(\Omega^\epsilon)} 
\le C, \quad \forall \epsilon>0.
\end{gathered}
\end{equation}


Provided that we have to compare functions defined in $\Omega^\epsilon$ for $\epsilon > 0$, we need to introduce some extension operators $P_\epsilon$ in a convenient way. We note that this approach is very common in homogenization theory.   
For the current analysis we extend the functions only over the upper boundary of the domain $\Omega^\epsilon$, namely, into the open set $\widetilde \Omega^\epsilon$ defined by 
\begin{equation} \label{OmegaT}
\begin{array}{l}
\widetilde \Omega^\epsilon = \{ (x_1,x_2) \in \R^2 \; | \;  x_1 \in (0,1), \; - G_\epsilon(x_1) < x_2 < H_1 \} \setminus \\
\displaystyle \qquad\qquad \cup_{i=1}^N  \{ (\xi_i,x_2) \, | \,   \min \{ H_{0,i-1}, H_{0,i} \} < x_2 < H_1\},
\end{array} 
\end{equation}
where $ H_{0,i} = \min_{y \in \R} H(\xi_i, y)$, and the points $0=\xi_0<\xi_1<\ldots<\xi_{N-1}<\xi_N=1$ and the positive constant $H_1$ are given by hypothesis {\bf (H)}.

\begin{lemma} \label{EOT}
Under conditions described above, there exists an extension  operator 
$$
P_{\epsilon} \in \mathcal{L}(L^p(\Omega^\epsilon),L^p(\widetilde \Omega^\epsilon)) 
  \cap \mathcal{L}(W^{1,p}(\Omega^\epsilon),W^{1,p}(\widetilde \Omega^\epsilon))
$$
and a constant $K$ independent of $\epsilon$ and $p$ such that
\begin{equation} 
\begin{gathered} \label{EQOP}
\| P_{\epsilon} \varphi \|_{L^p(\widetilde \Omega^\epsilon)} \le K \, \| \varphi \|_{L^p(\Omega^\epsilon)}  \\
\Big\| \frac{\partial P_{\epsilon} \varphi}{\partial x_1} \Big\|_{L^p(\widetilde \Omega^\epsilon)} 
\le K \, \Big\{ \Big\| \frac{\partial \varphi}{\partial x_1} \Big\|_{L^p(\Omega^\epsilon)} 
+  \eta(\epsilon) \, \Big\| \frac{\partial \varphi}{\partial x_2} \Big\|_{L^p(\Omega^\epsilon)} \Big\} \\
\Big\| \frac{\partial P_{\epsilon} \varphi}{\partial x_2} \Big\|_{L^p(\widetilde \Omega^\epsilon)} 
\le K \, \Big\| \frac{\partial \varphi}{\partial x_2} \Big\|_{L^p(\Omega^\epsilon)} 
\end{gathered}
\end{equation}
for all $\varphi \in W^{1,p}(\Omega^\epsilon)$ where $1 \le p \le \infty$ and 
$
\eta(\epsilon) = \sup_{x \in (0,1)} \{ | H'_\epsilon(x) | \}, \quad \epsilon > 0.
$
\end{lemma}
\begin{proof}
This result can be obtained using a reflection procedure over the upper oscillating boundary of $\Omega^\epsilon$. See \cite{ACPS, AP2} for details.

\end{proof}

\begin{remark} \label{rem:stoperator}
\begin{itemize}
\item[(i)] Note that operator $P_\epsilon$ preserves periodicity in the $x_1$ variable. Indeed, under this reflection procedure, we have that if the function $\varphi$ is periodic in $x_1$, then the extended function $P_\epsilon \varphi$ is also periodic in $x_1$. 
\item[(ii)] Lemma \ref{EOT} can also be applied to the case $G_\epsilon$ and $H_\epsilon$ independent of $\epsilon$. In particular, we still can apply this extension operator to the representative cell $Y^*$. 
\end{itemize}
\end{remark}

\begin{remark} \label{rem:eqnorm }
If for each $w \in W^{1,p}(\mathcal{O})$ we denote by  $||| \cdot |||$ the norm
$$
||| w |||_{W^{1,p}(\mathcal{O})} = 
\| w \|_{L^p(\mathcal{O})} + 
 \Big\| \frac{\partial w}{\partial x_1} \Big\|_{L^p(\mathcal{O})}
+ \left( 1+\eta(\epsilon) \right)\Big\| \frac{\partial w}{\partial x_2} \Big\|_{L^p(\mathcal{O})},
$$  
then we have that the extension operator $P_\epsilon$ must satisfy 
$
||| P_\epsilon w |||_{W^{1,p}(\widetilde \Omega^\epsilon)} \le K ||| w |||_{W^{1,p}(\Omega^\epsilon)}
$
where $K>0$ is independent of $\epsilon$. The norm $||| \cdot |||_{W^{1,p}}$ is equivalent to the usual one.
\end{remark}


Now let us to discuss how the solutions of \eqref{EP} depend on the domain $\Omega^\epsilon$ and more exactly on the functions $G_\epsilon$ and $H_\epsilon$.  As a matter of fact, we have a continuous dependence result in $L^\infty$ uniformly in $\epsilon$. 
Assume $G_\epsilon$, $\widehat G_\epsilon$, $H_\epsilon$ and $\widehat H_\epsilon$ are piecewise continuous functions satisfying hypothesis {\bf (H)}, 
and consider the associated oscillating domains $\Omega^\epsilon$ and $\widehat \Omega^\epsilon$ given by
$$
\begin{gathered} 
\Omega^\epsilon = \{ (x_1,x_2) \in \R^2 \; | \;  x_1 \in (0,1), \quad -G_\epsilon(x_1) < x_2 < H_\epsilon(x_1) \}, \\
\widehat \Omega^\epsilon = \{ (x_1,x_2) \in \R^2 \; | \;  x_1 \in (0,1), \quad -\widehat G_\epsilon(x_1) < x_2 < \widehat H_\epsilon(x_1) \}.
\end{gathered}
$$
Let $u^\epsilon$ and $\widehat u^\epsilon$ be the solutions of the problem (\ref{EP}) in the oscillating domains $\Omega^\epsilon$ and $\widehat \Omega^\epsilon$ respectively with $f^\epsilon \in L^2(\R^2)$. Then we have the following result:

\begin{proposition}  \label{BPT}
There exists a positive real function $\rho:[0,\infty) \mapsto [0,\infty)$ such that
$$
\|u^\epsilon-\widehat u^\epsilon\|^2_{H^1_\epsilon(\Omega^\epsilon \cap \widehat \Omega^\epsilon)} 
+ \|u^\epsilon\|^2_{H^1_\epsilon(\Omega^\epsilon \setminus \widehat \Omega^\epsilon)}
+ \|\widehat u^\epsilon\|^2_{H^1_\epsilon(\widehat\Omega^\epsilon \setminus \Omega^\epsilon)} \leq \rho(\delta)
$$
with $\rho(\delta)\to 0$ as $\delta\to 0$ uniformly for all 
\begin{enumerate}
\item[i)] $\epsilon > 0$;
\item[ii)] piecewise $C^1$ functions $G_\epsilon$, $\widehat G_\epsilon$, $H_\epsilon$ and $\widehat H_\epsilon$ with 
$$
\begin{gathered}
0\le G_0 \le G_\epsilon(x), \widehat G_\epsilon(x) \le G_1, \quad 
0< H_0 \le H_\epsilon(x), \widehat H_\epsilon(x) \le H_1, \\
\|G_\epsilon-\widehat G_\epsilon\|_{L^\infty(0,1)} \leq \delta \quad \textrm{ and } \quad 
\|H_\epsilon-\widehat H_\epsilon\|_{L^\infty(0,1)} \leq \delta;
\end{gathered}
$$
\item[iii)] $f^\epsilon\in L^2(\R^2)$, $\|f^\epsilon\|_{L^2(\R^2)}\leq 1$.
\end{enumerate}
\end{proposition}
\begin{proof}
The proof is quite analogous to that one performed in \cite[Theorem 4.1]{AP2} since we are taking functions $G$ and $H$ satisfying ${\bf (H)}$ with constant period $l_g$ and $l_h$ respectively. 
\end{proof}

\begin{remark} 
The important part of this result is that the positive function $\rho(\delta)$ does not depend on $\epsilon$. It only depends on the nonnegative constants $G_0$, $G_1$, $H_0$ and $H_1$.
\end{remark}


Finally, we mention some important estimates on the solutions of an elliptic problem posed in rectangles of the type  
$$
Q_\epsilon=\{ (x,y) \in \R^2 \; | \;  -\epsilon^\alpha<x<\epsilon^\alpha, \, 0<y<1\}
$$ 
with $\alpha>1$.  
For each $u_0 \in H^1(-\epsilon^\alpha,\epsilon^\alpha)$, let us define $u^\epsilon(x,y)$ as the unique solution of
\begin{equation} \label{P-basic}
\left\{
\begin{gathered}
 - \frac{\partial^2 u^\epsilon}{{\partial x}^2} - 
\frac{1}{\epsilon^2} \frac{\partial^2 u^\epsilon}{{\partial y}^2}= 0
\quad \textrm{ in } Q_\epsilon, \\
\qquad u(x,0)=u_0(x),\quad  \textrm{ on } \Gamma_\epsilon,\\
\frac{\partial u}{\partial \nu}=0,\quad  \textrm{ on } \partial Q_\epsilon \setminus \Gamma_\epsilon
\end{gathered}
\right.
\end{equation}
where $\nu$ is the outward unit normal to $\partial Q_\epsilon$ and 
$
\Gamma_\epsilon = \{ (x,0) \in \R^2 \, | \, -\epsilon^\alpha<x<\epsilon^\alpha \}.
$

\begin{lemma}
With the notations above, if we denote by $\bar u_0$ the average of $u_0$ in $\Gamma_\epsilon$, that is 
$$
\bar u_0=\frac{1}{2\epsilon^\alpha}\int_{-\epsilon^\alpha}^{\epsilon^\alpha}u_0(x) \, dx,
$$
then, there exists a constant $C$,  independent of $\epsilon$ and $u_0$, such that 
$$
\int_{-\epsilon^\alpha}^{\epsilon^\alpha}|u^\epsilon(x,y)-\bar u_0|^2 \, dx \leq C \exp\left\{-\frac{2y\pi}{\epsilon^{\alpha-1}}\right\} \|u_0\|_{L^2(-\epsilon^\alpha,\epsilon^\alpha)}^2
$$
$$
\int_0^1\int_{-\epsilon^\alpha}^{\epsilon^\alpha}|u(x,y)-\bar u_0|^2 \, dxdy \leq C\epsilon^{\alpha-1} \|u_0\|_{L^2(-\epsilon^\alpha,\epsilon^\alpha)}^2
$$
\end{lemma}
and 
\begin{equation}\label{basic-estimate}
\left\|\frac{\partial u}{\partial x}\right\|_{L^2(Q_\epsilon)}^2+\frac{1}{\epsilon^2}\left\|\frac{\partial u}{\partial y}\right\|^2_{L^2(Q_\epsilon)}\leq C \epsilon^{\alpha -1}
\left\|\frac{\partial u_0}{\partial x}\right\|_{L^2(-\epsilon^\alpha,\epsilon^\alpha)}^2.
\end{equation}
\begin{proof}
The proof follows from the known fact that the solution of the problem \eqref{P-basic} can be found explicitly and admits a Fourier decomposition of the form 
$$
u^\epsilon(x,y) = \frac{1}{2\epsilon^\alpha}\int_{-\epsilon^\alpha}^{\epsilon^\alpha} u_0(\tau) d\tau+ \sum_{  k=1  }^\infty (u_0,\varphi_k^\epsilon)\varphi_k^\epsilon(x) \frac{\cosh(\frac{k\pi(1-y)}{\epsilon^{\alpha-1}})}{\cosh(\frac{k\pi}{\epsilon^{\alpha-1}})}
$$
where $\varphi_k^\epsilon(x)=\epsilon^{-\alpha/2}\cos(\frac{k\pi x}{\epsilon^\alpha}),$ $k=1,2, \ldots,$ and  $(u_0,\varphi_k^\epsilon)=(u_0,\varphi_k^\epsilon)_{L^2(-\epsilon^\alpha, \epsilon^\alpha)}$.
\end{proof}

\section{The piecewise periodic case} \label{PWPC}

In this section we establish the limit of sequence $\{ u^\epsilon \}_{\epsilon > 0}$ given by
the elliptic problem \eqref{EP} as $\epsilon$ goes to zero for the case where the oscillating boundary of $\Omega^\epsilon$ is defined assuming that $G_\epsilon$ and $H_\epsilon$ are piecewise periodic functions. 

More precisely, we suppose the functions $G$ and $H$ satisfy hypothesis {\bf (H)}, assuming also they are independent functions of the first variable in each of the open sets $(\xi_{i-1},\xi_i)\times \R$. 
Thus, if $0=\xi_0<\xi_1<\ldots<\xi_{N-1}<\xi_{N}=1$ so that functions $G$ and $H$ satisfy  
\begin{equation} \label{GHH}
G(x, y)=G_i(y) \quad \textrm{ and } \quad H(x,y) = H_i(y), \quad \textrm{ for } x \in (\xi_{i-1},\xi_i),
\end{equation}
with $G_i(y+l_g)=G_i(y)$ and $H_i(y+l_h)=H_i(y)$ for all $y\in \R$. 
The functions $G_i$ and $H_i$ are $C^1$-functions satisfying
$0<G_0\leq G_i(\cdot)\leq G_1$ and $0<H_0\leq H_i(\cdot)\leq H_1$ for all $i=1,\ldots, N$, and then, the oscillating domain $\Omega^\epsilon$ is now 
$$
\begin{array}{l}
\Omega^\epsilon=\left\{ (x,y) \, | \,  \xi_{i-1}<x<\xi_i, -G_i(x/\epsilon)<y<H_i(x/\epsilon), i=1,\ldots, N\right\} \cup \\ 
\displaystyle \qquad \cup_{i=1}^{N-1}\left\{(\xi_i, y) \, | \, -\min\{G_{i}(\xi_i/\epsilon), G_{i+1}(\xi_i/\epsilon)\} < y < \min\{H_{i}(\xi_i/\epsilon), H_{i+1}(\xi_i/\epsilon)\}\right\}, 
\end{array}
$$
as illustrated by Figure \ref{figura2}. Also region $\widetilde \Omega^\epsilon$, previously introduced in \eqref{OmegaT}, is given by 
$$
\begin{array}{l}
\widetilde \Omega^\epsilon = \left\{ (x,y) \, | \,  \xi_{i-1}<x<\xi_i, -G_i(x/\epsilon)<y<H_1, i=1,\ldots, N \right\} \cup \\ 
\displaystyle \qquad \cup_{i=1}^{N-1} \left\{ (\xi_i, y) \, | \,  -\min\{G_{i}(\xi_i/\epsilon), G_{i+1}(\xi_i/\epsilon)\} < y <  \min \{ H_{0,i}, H_{0,i+1} \}  \right\},
\end{array}
$$
with $H_{0,i} = \min_{y \in \R} H_i(y)$, $i=1, \ldots, N$. 

We also denote by $\Omega_0$ the convenient open set without oscillating boundaries given by 
\begin{equation} \label{Omega0}
\begin{array}{l}
\Omega_0 =  \left\{ (x,y) \, | \,  \xi_{i-1}<x<\xi_i, -G_{0,i} <y<H_1, i=1,\ldots, N \right\} \cup \\ 
\displaystyle \qquad \cup_{i=1}^{N-1} \left\{(\xi_i, y) \, | \, -\min\{ G_{0,i}, G_{0,i+1} \} < y < \min \{ H_{0,i}, H_{0,i+1} \}  \right\},
\end{array}
\end{equation}
where the positive constant $G_{0,i}$, with $i=1,\ldots,N$, is set by  $G_{0,i} = \min_{y \in \R} G_i(y)$ whenever $x \in (\xi_{i-1},\xi_i)$. Here, we are establishing the following step function 
\begin{equation} \label{G0}
G_0(x) = G_{0,i} = \min_{y \in \R} G_i(y), \quad \textrm{ if } x \in (\xi_{i-1},\xi_i).
\end{equation}
Notice $\Omega_0 \subset \widetilde \Omega^\epsilon$ for all $\epsilon > 0$.

It is also important to observe that we still have the extension operator $P_\epsilon$ constructed in Lemma \ref{EOT} for the open regions $\Omega^\epsilon$ into $\widetilde \Omega^\epsilon$.

\begin{figure}[htp]
\centering \scalebox{1.10}{\includegraphics{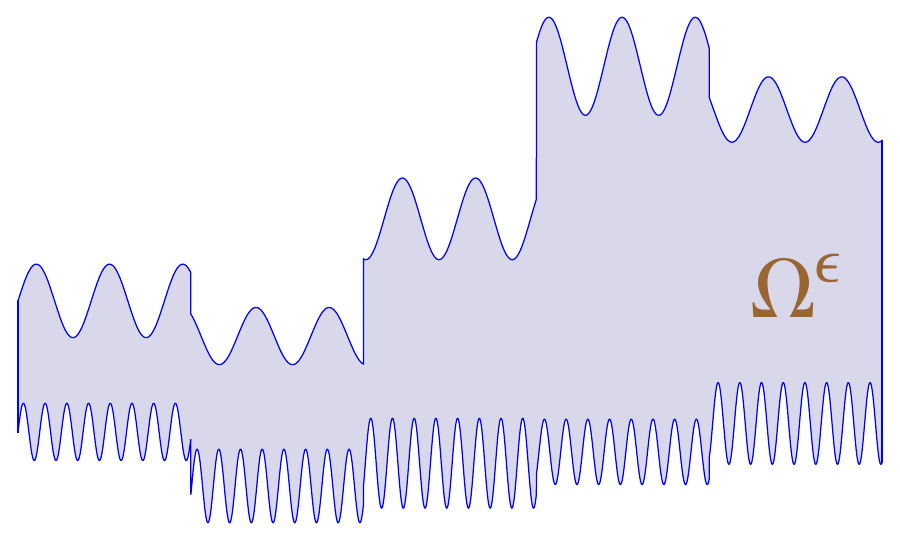}}
\caption{A piecewise periodic domain $\Omega^\epsilon$.}
\label{figura2}
\end{figure}

Now we can prove the following result 

\begin{lemma} \label{PPCT}
Assume that  $f^\epsilon \in L^2(\Omega^\epsilon)$ satisfies \eqref{FC} so that function 
\begin{equation} \label{FHD0}
\hat f^\epsilon(x) = \int_{-G_\epsilon(x)}^{H_\epsilon(x)} f(x,s) \, ds, \quad x \in (0,1),
\end{equation}
satisfies $\hat f^\eps\rightharpoonup \hat f$, w-$L^2(0,1)$. 

Then, there exists $\hat u \in H^1(0,1)$ such that, if $P_\epsilon$ is the extension operator given by Lemma \ref{EOT},  then 
$$
\| P_{\epsilon} u^\epsilon - \hat u \|_{L^2(\widetilde \Omega^\epsilon)} \to 0, \quad \textrm{ as } \epsilon \to 0,
$$ 
where $\hat u$  is the unique weak solution of the Neumann problem
\begin{equation} \label{VFPDL-piecewise0}
\int_0^1 \Big\{ q(x)  \; u_x(x) \, \varphi_x(x) 
+ p(x) \, u(x) \, \varphi(x) \Big\} dx = \int_0^1  \hat f(x) \, \varphi(x) \, dx
\end{equation}
for all $\varphi \in H^1(0,1)$, 
where $p(x)$ and $q(x)$ are piecewise constant functions defined a.e. $(0,1)$ as follows: if $0 = \xi_0 < \xi_1 < \ldots < \xi_N =1$, 
$p(x)=p_i$ for all $x\in (\xi_{i-1},\xi_i)$ where 
\begin{equation}\label{def-p}
\begin{gathered}
p_i=\frac{|Y_i^*|}{l_h} + \frac{1}{l_g} \int_0^{l_g} G_i(s) \, ds - G_{0,i},  \\
G_{0,i} = \min_{y \in \R} G_i(y),
\end{gathered}
\quad i=1,\ldots,N,
\end{equation}
$Y_i^*$ is the basic cell for $x\in(\xi_{i-1},\xi_i)$, that is,
$$
Y^*_i = \{ (y_1,y_2) \in \R^2 \; | \; 0< y_1 < l_h, \quad - G_{0,i} < y_2 < H_i(y_1) \},
$$
and $q(x)=q_i$ for all $x\in (\xi_{i-1},\xi_i)$ where
$$
q_i = \frac{1}{l_h} \int_{Y^*_i} \Big\{ 1 - \frac{\partial X_i}{\partial y_1}(y_1,y_2) \Big\} dy_1 dy_2
$$
and the function $X_i$ is the unique solution of
\begin{equation} \label{AUX}
\left\{
\begin{array}{l}
- \Delta X_i  =  0  \textrm{ in } Y_i^*  \\
\frac{\partial X_i}{\partial N}  =  0  \textrm{ on } B^i_2  \\
\frac{\partial X_i}{\partial N}  =  N_1 \textrm{ on } B^i_1  \\
X_i \quad \textrm{ $l_h$-periodic on } B^i_0 \\
\int_{Y_i^*} X_i \; dy_1 dy_2  =  0  
\end{array}
\right.
\end{equation}
where $B_0^i$, $B_1^i$ and $B_2^i$ are the lateral, upper and lower boundary of $\partial Y^*_i$ respectively.

\end{lemma}

\begin{remark} Note that if we call $f_0(x)=\hat f(x)/p(x)$, then problem \eqref{VFPDL-piecewise0} is equivalent to 
$$
- r_i u_{xx}(x) + u(x) = f_0(x) \quad  x \in (\xi_{i-1},\xi_i)\\
$$
for $i = 1, ..., N$, where $r_i=q_i/p_i$,  satisfying the following boundary conditions
$$
\left\{
\begin{gathered}
u_x(\xi_0) = u_x(\xi_N) = 0 \\
r_i \, u_x(\xi_i-) - r_{i+1} \, u_x(\xi_i+) = 0 \quad i = 1, ..., N -1.
\end{gathered}
\right. 
$$
Here, $u_x(\xi_i \pm)$ denote the right(left)-hand side limits of $u_x$ at $\xi_i$.

\end{remark}

\begin{proof}  
In order to prove Lemma \ref{PPCT}, we have to pass to the limit in the variational formulation of problem \eqref{EP} given by \eqref{VFP}.
For this, we first divide the domain $\widetilde \Omega^\epsilon$ in two open sets using an appropriated step function $G_0^\epsilon$, depending on $\epsilon$, that converges uniformly to the step function $G_0$ defined in \eqref{G0} and independent of parameter $\epsilon$.

Let us denote by $m_\epsilon$ the largest integer such that $m_\epsilon l_g \epsilon^{\alpha} \leq 1$. 
Now, for each $i=1,\ldots,N$ and $m=1,\ldots,m_\epsilon$, we take the following point
\begin{equation} \label{eqgamma}
\gamma_{\epsilon,m}^i \in [(m-1) l_g \epsilon^{\alpha}, m l_g \epsilon^{\alpha}] \cap (\xi_{i-1},\xi_i), 
\end{equation}
the minimum point of the piecewise periodic function $G_\epsilon$ restricted to $[(m-1) l_g \epsilon^{\alpha}, m l_g \epsilon^{\alpha}] \cap (\xi_{i-1},\xi_i)$, that can be empty depending on the values of $i$ and $m$. 
As a consequence of this construction, it is easy to see that
\begin{equation} \label{eqG00}
G_i(\gamma_{\epsilon,m}^i/\epsilon^\alpha) = \min_{y \in \R} G_i(y)  = G_{0,i}.
\end{equation}

Since the interval $(\xi_{i-1}, \xi_i)$ is finite, and $G_\epsilon|_{(\xi_{i-1}, \xi_i)}$ is continuous, then there exist just a finite number of points $\gamma_{\epsilon,m}^i \in (\xi_{i-1},\xi_i)$.
We can rename them such that  
\begin{equation} \label{PIE}
\{ \gamma_{\epsilon,0}^i, \gamma_{\epsilon,1}^i, ..., \gamma_{\epsilon,m_\epsilon^i + 1}^i \}
\end{equation}
defines a partition for the sub interval $[\xi_{i-1},\xi_i]$ for some $m_\epsilon^i \in \N$, $m_\epsilon^i \leq m_\epsilon$, where $\gamma_{\epsilon,0}^i=\xi_{i-1}$ and $\gamma_{\epsilon, m_\epsilon^i+1}^i=\xi_i$.
Note that $\gamma_{\epsilon,m}^i$ does not need to be uniquely defined.

Consequently, we can take the union of all partitions \eqref{PIE} setting a partition for the unit interval $[0,1]$ 
$$
\{ \gamma_{\epsilon,0}, \gamma_{\epsilon,1}, ..., \gamma_{\epsilon, \hat m_\epsilon+1} \},
$$
with $\gamma_{\epsilon,0}=0$ and $\gamma_{\epsilon,\hat m_\epsilon+1}=1$ for some $\hat m_\epsilon \in \N$ that we still denote by $m_\epsilon$.
Also, we have  
$$
\{ (\gamma_{m,\epsilon}, x_2) \; | \; - G_1< x_2 < - G_{0,i}  \}\cap \Omega^\epsilon=\emptyset,
$$ 
for all $m=1,2,\ldots, m_\epsilon$. 

Next we take $\epsilon$ small enough, and then we consider the convenient step function 
$$
G_{0}^\epsilon(x)=
\left\{
\begin{array} {ll}
 G_{0,1}, & x\in  [0, \gamma_{\epsilon,1}] \\ 
  \max\{ G(\gamma_{\epsilon,m}, \frac{\gamma_{\epsilon,m}}{\epsilon^\alpha}), G(\gamma_{\epsilon,m+1}, \frac{\gamma_{\epsilon,m+1}}{\epsilon^\alpha}) \}, & x\in  (\gamma_{\epsilon,m}, \gamma_{\epsilon,m+1}], m=1,2\ldots, m_\epsilon - 1   \\
G(1,1/\epsilon^\alpha) , & x\in  (\gamma_{\epsilon,m_\epsilon -1}, 1]
\end{array}
\right. .
$$

Due to \eqref{eqG00}, we have $G(\gamma_{\epsilon,m}, \frac{\gamma_{\epsilon,m}}{\epsilon^\alpha}) = G_i(\gamma_{\epsilon,m}/\epsilon^\alpha) = \min_{y \in \R} G_i(y)  = G_{0,i}$, whenever $\gamma_{\epsilon,m} \in (\xi_{i-1},\xi_i)$ for some $i=1,\ldots,N$, and so, $G_\epsilon(x) \geq G_0^\epsilon(x) \geq G_0(x)$ in $(0,1)$ where $G_0$ is the step function given by \eqref{G0}. 
Consequently, we have constructed a suitable step function $G^\epsilon_0$ that converges uniformly to $G_0$. More precisely, we have obtained
\begin{equation} \label{G_0C}
\| G_0 - G^\epsilon_0 \|_{L^\infty(0,1)} \to 0, \quad \textrm{ as } \epsilon \to 0.
\end{equation}

Using the step function $G_0^\epsilon$ we can introduce now the following open sets 
\begin{equation} \label{DOMAINS}
\begin{gathered}
\widetilde \Omega^\epsilon_+ = \{ (x_1, x_2) \in \R^2 \, | \, x_1 \in (0,1), \, - G_0^\epsilon (x_1) < x_2 < H_1 \} \textrm{ and } \\
\widetilde \Omega^\epsilon_- = \{ (x_1, x_2) \in \R^2 \, | \, x_1 \in (0,1), \, - G^\epsilon (x_1) < x_2 < - G_0^\epsilon(x_1) \}.
\end{gathered}
\end{equation}
Notice that
$$
\begin{gathered}
\widetilde \Omega^\epsilon = \mbox{Int} \left( \overline{\widetilde \Omega^\epsilon_+ \cup \widetilde \Omega^\epsilon_-} \right).
\end{gathered}
$$

Hence, if we denote by $\widetilde \cdot$ the standard extension by zero and by $\chi^\epsilon$ the characteristic function of $\Omega^\epsilon$, we can rewrite \eqref{VFP} as 
\begin{eqnarray} \label{VFP2}
& & \int_{\widetilde \Omega^\epsilon_-}  \left\{ \widetilde{\frac{\partial u^\epsilon}{\partial x_1}} \frac{\partial \varphi}{\partial x_1} 
+ \frac{1}{\epsilon^2} \widetilde{\frac{\partial u^\epsilon}{\partial x_2}} \frac{\partial \varphi}{\partial x_2} \right\} dx_1 dx_2 
+ \int_{\widetilde \Omega^\epsilon_+} \left\{ \widetilde{\frac{\partial u^\epsilon}{\partial x_1}} \frac{\partial \varphi}{\partial x_1} 
+ \frac{1}{\epsilon^2} \widetilde{\frac{\partial u^\epsilon}{\partial x_2}} \frac{\partial \varphi}{\partial x_2} \right\} dx_1 dx_2 \nonumber \\
& & \quad \quad + \int_{\widetilde \Omega^\epsilon} \chi^\epsilon \, P_\epsilon u^\epsilon \, \varphi \, dx_1 dx_2 = \int_{\widetilde \Omega^\epsilon} \chi^\epsilon \, f^\epsilon \varphi dx_1 dx_2, \quad  \forall \varphi \in H^1(\Omega^\epsilon),
\end{eqnarray}
where $P_\epsilon$ is the extension operator constructed in Lemma \ref{EOT}.

Now, let us to pass to the limit in the different functions that form the integrands of \eqref{VFP2} to get the homogenized problem. 
It is worth to observe that we will combine here techniques from \cite{AP2, AP3, AVP,Tt} establishing suitable oscillating test functions to accomplish our goal.

\par\bigskip\noindent {\bf  (a). Limit of $P_\epsilon u^\epsilon$ in $L^2(\Omega^\epsilon)$.}

First we observe that, due to \eqref{EST0} and Lemma \ref{EOT}, there exists $K>0$ independent of $\epsilon$ such that $P_\epsilon u^\epsilon$ satisfies
\begin{equation} \label{ESTPE}
\|P_\epsilon u^\epsilon \|_{L^2(\widetilde \Omega^\epsilon)}, \Big\| \frac{\partial P_\epsilon u^\epsilon}{\partial x_1} \Big\|_{L^2(\widetilde \Omega^\epsilon)}
\textrm{ and } \frac{1}{\epsilon} \Big\| \frac{\partial P_\epsilon u^\epsilon}{\partial x_2} \Big\|_{L^2(\widetilde \Omega^\epsilon)} 
\le K, \quad  \forall \epsilon > 0.
\end{equation}
Hence, if $\Omega_0$ is the open set given by \eqref{Omega0}, independent of $\epsilon$, $P_\epsilon u^\epsilon|_{\Omega_0} \in H^1(\Omega_0)$, and we can extract a subsequence, still denoted by $P_\epsilon u^\epsilon$, such that
\begin{equation} \label{WC}
\begin{array}{c}
P_\epsilon u^\epsilon \rightharpoonup u_0 \quad w-H^1(\Omega_0) \\
P_\epsilon u^\epsilon \rightarrow u_0 \quad s-H^s(\Omega_0) \textrm{ for all } s \in [0, 1) \textrm{ and }\\
\frac{\partial P_\epsilon u^\epsilon}{\partial x_2} \rightarrow 0 \quad s-L^2(\Omega_0) \\
\end{array}
\end{equation}
as $\epsilon \to 0$, for some  $u_0 \in H^1(\Omega_0)$.
Note that $u_0(x_1,x_2)$ does not depend on the variable $x_2$, that is,  
$
\frac{\partial  u_0}{\partial x_2}(x_1,x_2) = 0 \textrm{ a.e. } \Omega_0.
$
Indeed, for all $\varphi \in \mathcal{C}^\infty_0(\Omega_0)$, we have from \eqref{WC} that
\begin{equation} \label{u0x}
\int_{\Omega_0} u_0 \, \frac{\partial \varphi}{\partial x_2} \, dx_1 dx_2 = \lim_{\epsilon \to 0} \int_{\Omega_0} P_\epsilon u^\epsilon \, \frac{\partial \varphi}{\partial x_2} \, dx_1 dx_2 = - \lim_{\epsilon \to 0} \int_{\Omega_0} \frac{\partial P_\epsilon u^\epsilon}{\partial x_2} \, \varphi \, dx_1 dx_2 = 0,
\end{equation}
and then, $u_0(x_1,x_2) = u_0(x_1)$ for all $(x_1,x_2) \in \Omega_0$ implying $u_0 \in H^1(0,1)$.

From \eqref{WC}, we also have that the restriction of $P_\epsilon u^\epsilon$ to coordinate axis $x_1$ converges to $u_0$, in that,  
if $\Gamma = \{ (x_1,0) \in \R^2 \, | \, x_1 \in (0,1) \}$, then 
\begin{equation} \label{TRACE}
P_\epsilon u^\epsilon|_{\Gamma} \rightarrow u_0 \quad s-H^s(\Gamma), \quad \forall s \in [0,1/2).
\end{equation}
Thus, using \eqref{TRACE} with $s=0$, we can obtain the $L^2$-convergence of $P_\epsilon u^\epsilon$ to $u_0$ in $\widetilde \Omega^\epsilon$. 
In fact, due to \eqref{TRACE}, we have that
\begin{eqnarray*}
& &\|P_\epsilon u^\epsilon|_{\Gamma} - u_0\|^2_{L^2(\widetilde \Omega_\epsilon)}  = 
\int_0^1 \int_{-G_\epsilon(x_1)}^{H_1} | P_\epsilon u^\epsilon(x_1,0) - u_0(x_1) |^2 \, dx_2 dx_1 \\
& &\qquad\leq  C(G,H)\, \| P_\epsilon u^\epsilon|_{\Gamma} - u_0 \|_{L^2(\Gamma)}^2 \to  0, \textrm{ as } \eps \to 0,
\end{eqnarray*}
where $C(G,H)=G_1+H_1$.
Also,
$$
| P_\epsilon u^\epsilon(x_1,x_2) - P_\epsilon u^\epsilon(x_1,0) |^2 
= \left| \int_0^{x_2} \frac{\partial P_\epsilon u^\eps}{\partial x_2}(x_1,s) \, ds \right|^2 
\leq \left( \int_0^{x_2} \left| \frac{\partial P_\epsilon u^\epsilon}{\partial x_2}(x_1,s) \right|^2 ds \right) \, |x_2|.
$$
Consequently, integrating in $\widetilde \Omega^\epsilon$ and using \eqref{ESTPE}, we get
\begin{eqnarray*}
& & \|P_\epsilon u^\epsilon - P_\epsilon u^\epsilon|_{\Gamma} \|_{L^2(\widetilde \Omega^\epsilon)}^2 
\quad \leq  \int_0^1 \int_{-G_\epsilon(x_1)}^{H_1} \left(  \int_0^{x_2} \left| \frac{\partial P_\epsilon u^\epsilon}{\partial x_2}(x_1,s) \right|^2 ds \right) 
\, |x_2| \, dx_2 dx_1 \\ & & \qquad \leq  C(G,H) \, \left\| \frac{\partial P_\epsilon u^\epsilon}{\partial x_2} \right\|_{L^2(\widetilde \Omega^\epsilon)}^2  \to 0 \textrm{ as } \epsilon \to 0.
\end{eqnarray*}

Finally, since
\begin{eqnarray*}
\| P_\epsilon u^\epsilon - u_0 \|_{L^2(\widetilde \Omega^\epsilon)}  & \leq & \|P_\epsilon u^\epsilon - P_\epsilon u^\epsilon|_{\Gamma} \|_{L^2(\widetilde \Omega^\epsilon)}+ 
 \|P_\epsilon u^\epsilon|_{\Gamma} - u_0\|_{L^2(\widetilde \Omega^\epsilon)},
\end{eqnarray*}
we conclude that
\begin{equation} \label{L2CONV}
\| P_\epsilon u^\epsilon - u_0 \|_{L^2(\widetilde \Omega^\epsilon)} \to 0, \quad  \textrm{ as } \epsilon \to 0.
\end{equation}

\par\bigskip\noindent {\bf  (b). Limit of $\chi^\epsilon$.}

Let us consider the family of representative cell $Y^*_i$, $i=1,2\ldots,N$,  defined by 
$$
Y^*_i=\{ (y_1,y_2) \in \R^2 \; | \; 0<y_1<l_h \textrm{ and } -G_{0,i} < y_2 < H_i(y_1)\}
$$
and let $\chi_i$ be their characteristic function extended periodically on the variable $y_1 \in \R$ for each $i=1,\ldots, N$.  Eventually we will consider the family of representative cells $Y^*(x) = Y^*_i$ whenever $x \in (\xi_{i-1}, \xi_i)$.

If we denote by $\chi^\epsilon_i$ the characteristic function of the set 
$$
\Omega^\epsilon_{i,+} = \{(x_1,x_2) \, | \,  \xi_{i-1}<x_1<\xi_{i}, \, -G_{0,i}<x_2<H_i(x_1/\epsilon)\},
$$ 
we easily see that
\begin{equation} \label{chip}
\begin{gathered}
\chi^\epsilon(x_1,x_2) = \chi^\epsilon_i(x_1,x_2)  \quad \textrm{ and } \quad 
\chi^\epsilon_i(x_1,x_2) 
= \chi_i \left(\frac{x_1}{\epsilon},x_2\right)
\end{gathered} 
\end{equation}
whenever $(x_1,x_2) \in \Omega^\epsilon_{i,+}$.
Thus, due to \eqref{chip} and Average Theorem \cite[Theorem 2.6]{CD}, we have for each $i=1,\ldots,N$, and $x_2 \in (-G_{0,i}, H_1)$ that 
\begin{equation} \label{chi0}
\chi^\epsilon_i( \cdot, x_2) \stackrel{\eps\to 0}{\rightharpoonup} \theta_i(x_2) := \frac{1}{l_h} \int_0^{l_h} \chi_i(s,x_2) \, ds,  
\quad w^*-L^\infty(\xi_{i-1},\xi_i).
\end{equation}
Note that the limit function $\theta_i$ does not dependent on the variable $x_1\in (\xi_{i-1},\xi_i)$, although it depends on each $i=1,\ldots,N$, and it is related to the area of the open set $Y^*_i$ by formula
\begin{equation} \label{ITHETA}
l_h \int_{-G_{0,i}}^{H_1} \theta_i(x_2) dx_2 = |Y^*_i|.
\end{equation}

Moreover, using Lebesgue's Dominated Convergence Theorem and \eqref{chi0}, we can get that
\begin{equation} \label{chi}
\chi^\epsilon\stackrel{\eps\to 0}{ \rightharpoonup} \theta, \quad w^*-L^\infty(\Omega_0),
\end{equation}
where $\theta(x_1,x_2)=\theta_i(x_2)$ if $x_1\in (\xi_{i-1},\xi_i)$, $i=1,2,\ldots,N$. 
Indeed, from (\ref{chi0}) we have  
\begin{equation} \label{SL1}
\mathcal{F}^\epsilon_i(x_2) = \int_{\xi_{i-1}}^{\xi_i}  \varphi(x_1,x_2) \, \Big\{ \chi^\epsilon_i(x_1,x_2) - \theta_i(x_2) \Big\} \, dx_1 
\to 0, \textrm{ as } \epsilon \to 0, 
\end{equation}
a.e. $x_2 \in (-G_{0,i}, H_1)$ and for all $\varphi \in L^1(\Omega_0)$.
Thus, \eqref{chi} is a consequence of \eqref{SL1} and 
$$
\int_{\Omega_i} \varphi(x_1,x_2) \,\Big\{ \chi^\epsilon_i(x_1,x_2) - \theta_i(x_2) \Big\} \, dx_1 dx_2 
= \int_{-G_{0,i}}^{H_1} \mathcal{F}^\epsilon_i(x_2) dx_2, 
$$ 
since $|\mathcal{F}^\epsilon_i(x_2)| \le \int_{\xi_{I-1}}^{\xi_i}| \varphi(x_1,x_2)| dx_1$.

Notice that \eqref{ITHETA} implies the family of representative cells $Y^*(x)$ satisfies 
$$
Y^*(x) = l_h \int_{-G_0(x)}^{H_1} \theta(x_2) \, dx_2, \quad x \in (0,1).
$$

\par\bigskip\noindent {\bf (c). Limit in the tilde functions.}

Since $\|f^{\epsilon}\|_{L^{2}(\Omega^\epsilon)}$ is uniformly bounded, we get from \eqref{priori}  that there exists a constant $K>0$ independent of $\epsilon$ such that
$$
\begin{gathered}
\| \widetilde{u^\epsilon} \|_{L^2(\Omega_0)}, \Big\| \widetilde{\frac{\partial u^\epsilon}{\partial x_1}} \Big\|_{L^2(\Omega_0)} \textrm{ and }
\frac{1}{\epsilon} \Big\| \widetilde{\frac{\partial u^\epsilon}{\partial x_2}} \Big\|_{L^2(\Omega_0)} \le K
\textrm{ for all } \epsilon > 0.
\end{gathered}
$$
Then, we can extract a subsequence, still denoted by $\widetilde{u^\epsilon}$, $\widetilde{\frac{\partial u^\epsilon}{\partial x_1}}$
and $\widetilde{\frac{\partial u^\epsilon}{\partial x_2}}$, such that
\begin{equation} \label{WC0}
\begin{array}{c}
\widetilde{u^\epsilon} \rightharpoonup u^* \quad w-L^2(\Omega_0) \\
\widetilde{\frac{\partial u^\epsilon}{\partial x_1}} \rightharpoonup \xi^* \quad w-L^2(\Omega_0) \textrm{ and }\\
\widetilde{\frac{\partial u^\epsilon}{\partial x_2}} \rightarrow 0 \quad s-L^2(\Omega_0) \\
\end{array}
\end{equation}
as $\epsilon \to 0$, for some  $u^*$ and $\xi^* \in L^2(\Omega_0)$.


\par\bigskip\noindent {\bf  (d). Test functions.}

Here we introduce the first class of test functions needed to pass to the limit in the variational formulation \eqref{VFP2}. 
For each $\phi \in H^1(0,1)$ and $\eps > 0$, 
we define the following test function in $H^1(\widetilde \Omega^\epsilon)$
\begin{equation} \label{TESTF}
\begin{gathered}
\varphi^\epsilon(x_1,x_2) = \left\{
\begin{array}{ll}
\phi(x_1), & (x_1,x_2) \in \widetilde \Omega_+^\epsilon \\
Z^\epsilon_m(x_1,x_2), & (x_1,x_2) \in \widetilde \Omega^\epsilon_-\cap Q^\epsilon_m, \quad m=0,1,2,\ldots  
\end{array}
\right.
\end{gathered}
\end{equation}
where  $Q^\epsilon_m$ is the rectangle defined from the step function $G_0^\epsilon$,
\begin{equation} \label{RQ}
Q^\epsilon_m=\{ (x_1,x_2) \in \R^2 \, | \, \gamma_{m,\epsilon}<x_1<\gamma_{m+1,\epsilon}, \, - G_1 < x_2 < - G_0^\epsilon(x_1) \},
\end{equation}
and the function $Z^\epsilon_m$ is the solution of the problem
\begin{equation} \label{AUXSOL}
\left\{
\begin{array}{l}
- \frac{\partial^2 Z^\epsilon}{\partial x_1^2} - \frac{1}{\epsilon^2} \frac{\partial^2 Z^\epsilon}{\partial x_2^2} 
= 0, \quad \textrm{ in } Q^\epsilon_m \\
\frac{\partial Z^\epsilon}{\partial N^\epsilon}=0, \quad \textrm{ on }  \partial Q^\epsilon_m \backslash \Gamma_m^\epsilon  \\
Z^\epsilon = \phi,  \quad \textrm{ on } \Gamma_m^\epsilon
\end{array}
\right. 
\end{equation}
where $\Gamma_m^\epsilon$ is the top of the rectangle $Q^\epsilon_m$ given by 
$$
\Gamma_m^\epsilon = \{ (x_1, - G_0^\eps(x_1)) \, | \, \gamma_{m,\epsilon} < x_1 < \gamma_{m+1,\epsilon}\}.
$$

It is a direct consequence of \eqref{eqgamma} and estimate \eqref{basic-estimate} that functions $Z^\epsilon_m$ satisfies 
\begin{equation} \label{ESTX}
\left\| \frac{\partial Z^\epsilon_m}{\partial x_1} \right\|^2_{L^2(Q^\epsilon_m)}
+ \frac{1}{\epsilon^2} \left\| \frac{\partial Z^\epsilon_m}{\partial x_2} \right\|^2_{L^2(Q^\epsilon_m)}
\leq C \epsilon^{\alpha-1} \| \phi' \|^2_{L^2(\gamma_{m,\epsilon},\gamma_{m+1,\epsilon})}.
\end{equation}
Hence, if we denote by  $Q^\epsilon = \cup_{i=1}^{m_\epsilon} Q^\epsilon_i$, we have $\widetilde \Omega^\epsilon_-= Q^\epsilon \cap \widetilde \Omega^\epsilon$, and then, 
\begin{equation} \label{ESTX20}
\begin{array}{l}
\displaystyle \left\| \frac{\partial \varphi^\epsilon}{\partial x_1} \right\|^2_{L^2(\widetilde \Omega^\epsilon_-)}
+ \frac{1}{\epsilon^2} \left\| \frac{\partial \varphi^\epsilon}{\partial x_2} \right\|^2_{L^2(\widetilde \Omega^\epsilon_-)}
= \sum^{m_\epsilon}_{i=0} \left( 
\left\| \frac{\partial \varphi^\epsilon}{\partial x_1} \right\|^2_{L^2(Q^\epsilon_m)}
+ \frac{1}{\epsilon^2} \left\| \frac{\partial \varphi^\epsilon}{\partial x_2} \right\|^2_{L^2(Q^\epsilon_m)}
\right) \\
\displaystyle \qquad \quad\leq \sum^{m_\epsilon}_{i=0} C \, \epsilon^{\alpha -1} \, \left\| \phi' \right\|^2_{L^2(\gamma_{i,\epsilon},\gamma_{i+1,\epsilon})}  \leq  C \, \epsilon^{\alpha -1} \left\| \phi' \right\|^2_{L^2(0,1)}.
\end{array}
\end{equation}
Eventually we will use $Z^\epsilon$ to denote $Z^\epsilon(x_1,x_2) = Z^\epsilon_m(x_1,x_2)$ whenever $(x_1,x_2) \in \widetilde \Omega^\epsilon_- \cap Q^\epsilon_m$.

Consequently, we can argue as in \eqref{L2CONV} to show 
\begin{equation} \label{TFCONV}
\| \varphi^\epsilon - \phi \|_{L^2(\widetilde \Omega^\epsilon)} \to 0, \quad \textrm{ as } \epsilon \to 0.
\end{equation}
Indeed, since
$$
\varphi^\epsilon(x_1,x_2) - \phi(x_1) = \varphi^\epsilon(x_1,x_2) - \varphi^\epsilon(x_1,0) 
= \int_0^{x_2} \frac{\partial \varphi^\epsilon}{\partial x_2}(x_1,s) \, ds,
$$
we have from \eqref{TESTF} and \eqref{ESTX20} that 
\begin{eqnarray*}
\| \varphi^\epsilon - \phi \|_{L^2(\widetilde \Omega^\epsilon)}^2  \leq  C(G,H) \, \left\| \frac{\partial \varphi^\epsilon}{\partial x_2} \right\|_{L^2(\widetilde \Omega^\epsilon)}^2
 \le   C \, C(G,H) \, \epsilon^{1+\alpha} \, \left\| \phi' \right\|^2_{L^2(0,1)} 
 \to  0, \textrm{ as } \epsilon \to 0.
\end{eqnarray*}

\par\bigskip\noindent {\bf  (e). Passing to the limit in the weak formulation}.

Now let us to perform our first evaluation of the variational formulation \eqref{VFP2} of elliptic problem \eqref{EP} using the test functions $\varphi^\epsilon$ defined in \eqref{TESTF}.
For this, we analyze the different functions that form the integrands in \eqref{VFP2} using the computations previously established.

\begin{itemize}
\item First integrand: we obtain 
\begin{equation} \label{INT1}
\int_{\widetilde \Omega^\epsilon_-} \Big\{ \widetilde{\frac{\partial u^\epsilon}{\partial x_1}} \frac{\partial \varphi^\eps}{\partial x_1} 
+ \frac{1}{\epsilon^2} \widetilde{\frac{\partial u^\epsilon}{\partial x_2}} \frac{\partial \varphi^\eps}{\partial x_2} \Big\} dx_1 dx_2
\to 0, \quad \textrm{ as } \epsilon \to 0.
\end{equation}
 
Indeed, from \eqref{ESTX}, $\alpha > 1$ and \eqref{EST0}, we have that there exists $C>0$ independent of $\epsilon$ such that 
\begin{eqnarray*} \label{eqd1}
& & \qquad \qquad\qquad \left|\int_{\widetilde \Omega^\epsilon_-} \left\{ \widetilde{\frac{\partial u^\epsilon}{\partial x_1}} \frac{\partial \varphi^\epsilon}{\partial x_1} 
+ \frac{1}{\epsilon^2} \widetilde{\frac{\partial u^\epsilon}{\partial x_2}} \frac{\partial \varphi^\epsilon}{\partial x_2} \right\} dx_1 dx_2\right| \nonumber  
\\
& & \quad \leq  \left( \int_{\Omega^\epsilon} \left\{ \left( \frac{\partial u^\epsilon}{\partial x_1} \right)^2
+ \frac{1}{\epsilon^2} \left( \frac{\partial u^\epsilon}{\partial x_2} \right)^2 \right\} dx_1 dx_2 \right)^{1/2}
\left( \int_{\widetilde \Omega^\epsilon_-} \left\{ \left( \frac{\partial Z^\epsilon}{\partial x_1} \right)^2
+ \frac{1}{\epsilon^2} \left( \frac{\partial Z^\epsilon}{\partial x_2} \right)^2 \right\} dx_1 dx_2 \right)^{1/2} \nonumber \\
& & \quad \leq  C \, \eps^{(\alpha - 1)/2} \, \| \phi' \|_{L^2(0,1)} \to 0, \textrm{ as } \eps \to 0. 
\end{eqnarray*}

\par\medskip 
\item Second integrand: we have
\begin{equation} \label{INT2}
\int_{\widetilde \Omega^\epsilon_+} \left\{ \frac{\partial u^\epsilon}{\partial x_1} \frac{\partial \varphi^\epsilon}{\partial x_1} 
+ \frac{1}{\epsilon^2} \frac{\partial u^\epsilon}{\partial x_2} \frac{\partial \varphi^\epsilon}{\partial x_2} \right\} dx_1 dx_2
\to \int_{\Omega_0} \xi^* \, \phi'(x_1) \, dx_1 dx_2, \quad \textrm{ as } \epsilon \to 0.
\end{equation}
For see this, we first observe that \eqref{TESTF} implies
$$
\frac{\partial \varphi^\epsilon}{\partial x_1} \Big|_{\widetilde \Omega^\epsilon_+} = \frac{\partial \phi}{\partial x_1} = \phi' \quad
\textrm{ and } \quad 
\frac{\partial \varphi^\epsilon}{\partial x_2} \Big|_{\widetilde \Omega^\epsilon_+} = \frac{\partial \phi}{\partial x_2} = 0.
$$ 
Then, since $G^\epsilon_0 \geq G_0$ in $(0,1)$, we have $\Omega_0 \subset \widetilde \Omega^\epsilon_+$ and   
\begin{eqnarray} \label{eqSI}
& & \int_{\widetilde \Omega^\epsilon_+} \left\{ \widetilde{\frac{\partial u^\epsilon}{\partial x_1}} \frac{\partial \varphi^\epsilon}{\partial x_1} 
+ \frac{1}{\epsilon^2} \widetilde{\frac{\partial u^\epsilon}{\partial x_2}} \frac{\partial \varphi^\epsilon}{\partial x_2} \right\} dx_1 dx_2
 =  \int_{\widetilde \Omega^\epsilon_+} \widetilde{\frac{\partial u^\epsilon}{\partial x_1}}(x_1,x_2) \,  \phi'(x_1) \, dx_1 dx_2  \nonumber \\
& & \quad \quad = \int_{\Omega_0} \widetilde{\frac{\partial u^\epsilon}{\partial x_1}}(x_1,x_2) \,  \phi'(x_1) \, dx_1 dx_2 
 + \int_{\widetilde \Omega^\epsilon_+ \backslash \Omega_0} \frac{\partial u^\epsilon}{\partial x_1}(x_1,x_2) \,  \phi'(x_1) \, dx_1 dx_2.
\end{eqnarray}
Thus, from \eqref{WC0}, we pass to the limit as $\epsilon \to 0$ in the first integral of \eqref{eqSI} to get
\begin{eqnarray} \label{SL@}
\int_{\Omega_0} \frac{\partial u^\epsilon}{\partial x_1}(x_1,x_2) \,  \phi'(x_1) \, dx_1 dx_2
& \to & \int_{\Omega_0} \xi^* \,  \phi'(x_1) \, dx_1 dx_2.   
\end{eqnarray}
Hence, we will prove \eqref{INT2} if we show that the remaining integral of \eqref{eqSI} goes to zero as $\epsilon \to 0$. Let us evaluate it. 
From \eqref{EST0}), \eqref{Omega0}, \eqref{G_0C} and \eqref{DOMAINS}, we have 
\begin{eqnarray} \label{eqINT22}
\left| \int_{\widetilde \Omega^\epsilon_+ \backslash \Omega_0} \widetilde{\frac{\partial u^\epsilon}{\partial x_1}}(x_1,x_2) \,  \phi'(x_1) \, dx_1 dx_2 \right|
& \leq & \left\| \frac{\partial u^\epsilon}{\partial x_1} \right\|_{L^2(\Omega^\epsilon)}
\, \| \phi' \|_{L^2(\Omega^\epsilon_+ \backslash \Omega_0)} \nonumber \\
& \leq & C \, \| \phi' \|_{L^2(0,1)} \, \| G^\epsilon_0 - G_0 \|_{L^\infty(0,1)}^{1/2} 
\to 0, 
\end{eqnarray}
as $\epsilon \to 0$.
Therefore, \eqref{INT2} follows from \eqref{eqSI}, \eqref{SL@} and \eqref{eqINT22}.
\par\medskip

\item Third integrand: if $p(x)$ is that one in \eqref{def-p}, then
\begin{equation} \label{INT3}
\int_{\widetilde \Omega^\epsilon} \chi^\epsilon \, P_\epsilon u^\epsilon \, \varphi^\epsilon \, dx_1 dx_2 \to \int_0^1 p(x) \, u_0(x) \, \phi(x) \, dx, \quad \textrm{ as } \epsilon \to 0.
\end{equation}

We start observing that $P_\epsilon u^\epsilon |_{\Omega^\epsilon} = u^\epsilon$, and so
\begin{eqnarray*}
\int_{\widetilde \Omega^\epsilon} \chi^\epsilon \, P_\epsilon u^\epsilon \, \varphi^\epsilon \, dx_1 dx_2 & = & 
\int_{\Omega^\epsilon} \left( u^\epsilon - u_0 \right) \, \varphi^\epsilon \, dx_1 dx_2
+ \int_{\Omega^\epsilon} u_0 \, \left( \varphi^\epsilon - \phi \right) \, dx_1 dx_2 \\
& &  \qquad + \int_{\Omega^\epsilon} u_0 \, \phi \, dx_1 dx_2.
\end{eqnarray*}
Moreover, due to (\ref{L2CONV}) and (\ref{TFCONV}), we have 
$$
\begin{gathered}
\int_{\Omega^\epsilon} \left(u^\epsilon - u_0 \right) \, \varphi^\epsilon \, dx_1 dx_2 \to 0 \textrm{ and } \int_{\Omega^\epsilon} u_0 \, \left( \varphi^\epsilon - \phi \right) \, dx_1 dx_2 \to 0,
\end{gathered}
$$
as $\epsilon \to 0$, since $\Omega^\epsilon \subset \widetilde \Omega^\epsilon$, and so
$$
\| u^\epsilon - u_0 \|_{L^2(\Omega^\epsilon)} \leq \| P_\epsilon u^\epsilon - u_0 \|_{L^2(\widetilde \Omega^\epsilon)} 
\quad \textrm{ and } \quad 
\| \varphi^\epsilon - \phi \|_{L^2(\Omega^\epsilon)} \leq \| \varphi^\epsilon - \phi \|_{L^2(\widetilde \Omega^\epsilon)}.
$$

Thus, we need only to pass to the limit in  
\begin{equation} \label{SL@@}
\int_{\Omega^\epsilon} u_0(x_1) \, \phi(x_1) \, dx_1 dx_2 = \int_0^1 u_0(x) \, \phi(x) \, \left(H_\epsilon(x) + G_\epsilon(x) \right) \, dx,
\end{equation}
and then obtain \eqref{INT3}.
For this, we use the Average Theorem from \cite[Lemma 4.2]{AP3}, as well as, condition \eqref{GHH}. 
Indeed,  
\begin{eqnarray*}
H_\epsilon(x) + G_\epsilon(x) & = & H(x, x/\epsilon) + G(x,x/\epsilon^\alpha) \\
& \rightharpoonup & \frac{1}{l_h} \int_0^{l_h} H(x,y) \, dy + \frac{1}{l_g} \int_0^{l_g} G(x,y) \, dy, \quad  w^* - L^\infty(0,1),
\end{eqnarray*}
as $\epsilon \to 0$.
Hence, since $\frac{|Y^*(x)|}{l_h} - G_0(x) = \frac{1}{l_h} \int_0^{l_h} H(x,y) \, dy$, we have 
$$
H_\epsilon(x) + G_\epsilon(x) \rightharpoonup p(x), \quad  w^* - L^\infty(0,1).
$$


\item Fourth integrand: we claim that
\begin{equation} \label{INT4}
\int_{\widetilde \Omega^\epsilon} \chi^\epsilon \, f^\epsilon \, \varphi^\epsilon \, dx_1 dx_2 \to \int_0^1  \hat f(x) \, \phi(x) \, dx, \quad \textrm{ as } \epsilon \to 0.
\end{equation}
Since
$$
\int_{\widetilde \Omega^\epsilon} \chi^\epsilon \, f^\epsilon \, \varphi^\epsilon \, dx_1 dx_2 = 
\int_{\widetilde \Omega^\epsilon} \chi^\epsilon \, f^\epsilon \, \left( \varphi^\epsilon - \phi \right) \, dx_1 dx_2
+ \int_{\widetilde \Omega^\epsilon} \chi^\epsilon \, f^\epsilon \, \phi \, dx_1 dx_2
$$ 
and 
$$
\int_{\widetilde \Omega^\epsilon} \chi^\epsilon f^\epsilon \, \phi \, dx_1 dx_2 = 
\int_0^1 \left( \int_{-G_\epsilon(x_1)}^{H_\eps(x_1)} f^\epsilon(x_1,x_2) \, dx_2 \right) \phi(x_1) \, dx_1
= \int_0^1 \hat f^\epsilon(x) \, \phi(x) \, dx,
$$
we obtain \eqref{INT4} from \eqref{FHD0} and \eqref{TFCONV}. 
\end{itemize}

Consequently, we can use \eqref{INT1}, \eqref{INT2}, \eqref{INT3} and \eqref{INT4} to pass to the limit in \eqref{VFP2} to obtain the following limit variational formulation 
\begin{equation} \label{limitP}
\int_{\Omega_0} \xi^* \, \phi'(x_1) \, dx_1 dx_2 + \int_0^1 p(x) \, u_0(x) \, \phi(x) \, dx  
= \int_0^1 \hat f(x) \, \phi(x) \, dx,
\end{equation}
for all $\phi \in H^1(0,1)$. 

Next we need to evaluate the relationship between functions $\xi^*$ and $u_0$ to complete our proof obtaining the limit problem \eqref{VFPDL-piecewise0}.

\par\bigskip\noindent {\bf  (f). Relationship between $\xi^*$ and $u_0$.}

First let us to denote by $\Omega$ the rectangle $\Omega = (0,1) \times (-G_1,H_1)$, and recall the oscillating regions $\Omega^\epsilon_{i,+}$ given by
$$
\Omega^\epsilon_{i,+} = \left\{ (x_1, x_2) \, | \, \xi_{i-1} < x_1 < \xi_i, \, -G_{0,i} < x_2 < H_i(x_1/\epsilon) \right\}, \quad i=1,\ldots,N.
$$
Here we are taking the positive constants $G_1$ and $H_1$ from hypothesis ${\bf (H)}$, and $G_{0,i}$ is defined in \eqref{G0}. 
We also consider the families of isomorphisms 
$T^{\epsilon}_k : A^{\epsilon}_k \mapsto Y$ given by
\begin{equation} \label{ISO}
T^{\epsilon}_k(x_1,x_2) = \left(\frac{x_1 - \epsilon k l_h}{\epsilon},x_2\right)
\end{equation}
where 
$$
\begin{gathered}
A^{\epsilon}_k = \{ (x_1,x_2) \in \R^2 \; | \;  
\epsilon  k l_h \leq x_1 < \epsilon l_h (k+1)  \textrm{ and } -G_1 < x_2 < H_1 \} \\
Y = (0, l_h) \times (-G_1, H_1)
\end{gathered}
$$
with $k \in \N$.  
Let us recall the auxiliary problem in the representative cell $Y^*_i$ 
\begin{equation} \label{AUX0}
\left\{
\begin{array}{l}
- \Delta X_i  =  0  \textrm{ in } Y^*_i \\
\frac{\partial X_i}{\partial N}  =  0  \textrm{ on } B_2^i  \\
\frac{\partial X_i}{\partial N} 
= - \frac{H_i'(y_1)}{\sqrt{1+H_i'(y_1)^2}} \textrm{ on } B_1^i  \\
 X_i \textrm{ $l_h$-periodic on } B_0^i \\
  \int_{Y^*_i} X_i \; dy_1 dy_2 = 0  
 \end{array}
\right.
\end{equation}
where $B_0^i$, $B_1^i$ and $B_2^i$ are the lateral, upper and lower boundary of $\partial Y^*_i$ respectively.

Applying the same reflection procedure used in Lemma \ref{EOT}, we can define the extension operators 
\begin{equation} \label{PI}
P^i \in \mathcal{L}(H^1(Y^*_i),H^1(Y)) \cap \mathcal{L}(L^2(Y^*_i),L^2(Y)),
\end{equation}
which are obtained by reflection in the negative direction along the line $x_2=-G_{i,0}$, and in the positive direction along the graph of function $H_i$, as indicated in Remark \ref{rem:stoperator}. 

Thus, taking the isomorphism (\ref{ISO}) and extension operator \eqref{PI}, we can set the function
\begin{eqnarray*}
\omega^{\epsilon} (x_1,x_2) & = & x_1 - \epsilon \Big(  P^iX_i \circ T^\epsilon_k (x_1,x_2) \Big) \\
& = & x_1 - \epsilon \Big( P^i X_i \left( \frac{x_1 - \epsilon l_h k}{\epsilon},x_2 \right) \Big),  \hbox{ for } (x_1,x_2)\in \Omega_i\cap A^\epsilon_k, \quad i=1,\ldots, N,
\end{eqnarray*}
where 
$$
\Omega_i = (\xi_{i-1},\xi_i) \times (-G_1,H_1).
$$ 
Clearly function $\omega^{\epsilon}$ is well defined in  $\cup_{i=1}^N \Omega_i$. If $(x_1,x_2) \in \Omega_i$ for some $i=1,\ldots,N$, then there exists a unique $k\in \N$ such that $(x_1,x_2)\in A^\epsilon_k$. Furthermore, we have 
$$\omega^\epsilon \in H^1(\cup_{i=1}^N \Omega_i).$$

We introduce now the vector $\eta^\epsilon = (\eta_1^\epsilon,\eta_2^\epsilon)$ defined by 
\begin{equation} \label{VETA}
\eta_r^\epsilon(x_1,x_2) = \frac{\partial \omega^\epsilon}{\partial x_r}(x_1,x_2),  \quad (x_1,x_2)\in \cup_{i=1}^N \Omega_i, \quad r=1,2.
\end{equation}
Since $\frac{\partial}{\partial x_1} = \frac{1}{\epsilon} \frac{\partial}{\partial y_1}$ 
and $\frac{\partial}{\partial x_2} = \frac{\partial}{\partial y_2}$, we have that
\begin{equation} \label{ETA1EX}
\begin{gathered}
\eta_1^\epsilon(x_1,x_2)= 1 - \frac{\partial X_i}{\partial y_1}\left(\frac{x_1 - \epsilon k L}{\epsilon},x_2\right)
= 1 - \frac{\partial X_i }{\partial y_1}\left(\frac{x_1}{\epsilon},x_2\right) := \eta_1(y_1,y_2), \\ 
\eta_2^\epsilon(x_1,x_2)= - \epsilon \frac{\partial X_i }{\partial y_2}\left(\frac{x_1 - \epsilon k L}{\epsilon},x_2\right)
= - \epsilon \frac{\partial X_i }{\partial y_2}\left(\frac{x_1}{\epsilon},x_2\right)
:= \eta_2(y_1,y_2), 
\end{gathered}
\end{equation}
for $(y_1,y_2) = (\frac{x_1 - \epsilon k L}{\epsilon},x_2) \in Y^*_i$, $(x_1,x_2) \in \Omega_{i,+}^\epsilon$, $i=1,\ldots,N$.

Then, performing standard computations, we get from \eqref{AUX0} that $\eta_1^\epsilon$ and $\eta_2^\epsilon$ satisfy 
\begin{equation}\label{DIV}
\begin{gathered}
\frac{\partial \eta_1^\epsilon}{\partial {x_1}} 
+ \frac{1}{\epsilon^2} \frac{\partial \eta_2^\epsilon}{\partial{x_2}}   =  0 \textrm{ in }  \Omega^\epsilon_{i,+}, \\ 
\eta_1^\epsilon N^\epsilon_1 
+ \frac{1}{\epsilon^2} \eta_2^\epsilon N^\epsilon_2 = 0 \textrm{ on } \left(x_1, H_i \left(\frac{x_1}{\epsilon} \right) \right),\,  \\
\eta_1^\epsilon N^\epsilon_1 
+ \frac{1}{\epsilon^2} \eta_2^\epsilon N^\epsilon_2 = 0 \textrm{ on } (x_1, -G_{0,i}),
\end{gathered}
\end{equation}
for each $i=1,\ldots,N$, where 
$$
\begin{gathered}
N^\epsilon = (N^\epsilon_1, N^\epsilon_2) = \left(
 - \frac{H_i'(\frac{x_1}{\epsilon})}{(\epsilon^2+{H_i'(\frac{x_1}{\epsilon})}^2)^{\frac{1}{2}}}, 
 \frac{\epsilon}{ (\epsilon^2+{H_i'(\frac{x_1}{\epsilon})}^2)^{\frac{1}{2}}} \right) 
 \textrm{ on } (x_1, H_i \Big(\frac{x_1}{\epsilon} \Big) ),  \\
N^\epsilon = ( 0 , -1) \textrm{ on } (x_1, -G_{0,i}). 
\end{gathered}
$$


Therefore, multiplying first equation of \eqref{DIV} by a test function $\psi \in H^1(\Omega)$ with $\psi=0$ in a neighborhood of set
$\cup_{i=0}^N \{ (\xi_i,x_2) \, | \,  -G_1 \leq x_2\leq H_1 \}$ and integrating by parts, 
we obtain
\begin{eqnarray*}
0 & = & \int_{\Omega^\epsilon_+} \psi \left( \frac{\partial \eta_1^\epsilon}{\partial {x_1}} 
+ \frac{1}{\epsilon^2} \frac{\partial \eta_2^\epsilon}{\partial{x_2}} \right) dx_1 dx_2 \\
& = &  \int_{\partial \Omega^\epsilon_+}  \psi \left( \eta_1^\epsilon N^\epsilon_1 + \frac{1}{\epsilon^2} \eta_2^\epsilon N^\epsilon_2 \right) dS 
- \int_{\Omega^\epsilon_+} \left(\frac{\partial \psi}{\partial x_1} \eta_1^\epsilon 
+ \frac{1}{\epsilon^2} \frac{\partial \psi}{\partial x_2} \eta_2^\epsilon \right) dx_1 dx_2 \\ 
& = & 0 - \int_{\Omega^\epsilon_+}  \left(\frac{\partial \psi}{\partial x_1}  \eta_1^\epsilon  
+ \dfrac{1}{\epsilon^2} \frac{\partial \psi}{\partial x_2}  \eta_2^\epsilon \right) dx_1 dx_2,
\end{eqnarray*}
where 
$$
\Omega^\epsilon_+ = \mbox{Int} \left( \overline{ \cup_{i=1}^N \, \Omega^\epsilon_{i,+} }\right).
$$ 
Then, for all $\psi \in H^1(\Omega)$ with $\psi=0$ in a neighborhood of 
$\cup_{i=0}^N \{ (\xi_i,x_2) \, | \,  -G_1 \leq x_2\leq H_1 \}$, 
\begin{equation}\label{diveq}
\int_{\Omega^\epsilon_+}  \left(\eta_1^\epsilon   \frac{\partial \psi}{\partial x_1}  
+  \eta_2^\epsilon \dfrac{1}{\epsilon^2} \frac{\partial \psi}{\partial x_2} \right) dx_1 dx_2=0.
\end{equation}
Consequently, we can rewrite the variational formulation \eqref{VFP} using identity \eqref{diveq} in 
\begin{eqnarray} \label{VFP3}
& & \int_{\widetilde \Omega^\epsilon}  \left\{ \widetilde{\frac{\partial u^\epsilon}{\partial x_1}} \frac{\partial \varphi}{\partial x_1} 
+ \frac{1}{\epsilon^2} \widetilde{\frac{\partial u^\epsilon}{\partial x_2}} \frac{\partial \varphi}{\partial x_2}  + \chi^\epsilon \, P_\epsilon u^\epsilon \, \varphi \right\} dx_1 dx_2  
- \int_{\Omega^\epsilon_+}  \left(\eta_1^\epsilon   \frac{\partial \psi}{\partial x_1}  
+  \eta_2^\epsilon \dfrac{1}{\epsilon^2} \frac{\partial \psi}{\partial x_2} \right) dx_1 dx_2 \nonumber \\ 
& & \qquad \qquad \qquad = \int_{\widetilde \Omega^\epsilon} \chi^\epsilon \, f^\epsilon \varphi dx_1 dx_2, \quad  \forall \varphi \in H^1(\Omega^\epsilon).
\end{eqnarray}

Now, in order to accomplish our goal, we will pass to the limit in \eqref{VFP3}. For this, we introduce a second class of suitable test functions which will allow us to get our limit problem.

Let $\phi= \phi(x) \in \mathcal{C}^\infty_0(\cup_{i=1}^{N} (\xi_{i-1},\xi_{i}))$ and consider the following test function  
\begin{equation} \label{TF2}
\varphi^\epsilon(x_1,x_2) = \left\{ 
\begin{array}{ll}
\phi(x_1) \, \omega^\epsilon(x_1,x_2), & (x_1,x_2) \in \widetilde \Omega_+^\epsilon  \\
Z^\epsilon_m(x_1,x_2), & (x_1,x_2) \in \widetilde \Omega^\epsilon_- \cap Q^\epsilon_m, \quad m=0,1,2,\ldots  
\end{array}
\right.
\end{equation}
where  $Q^\epsilon_m$ is the rectangle defined by the step function $G_0^\epsilon$ previously introduced in \eqref{RQ}, with $\widetilde \Omega_+^\epsilon$ and $\widetilde \Omega_-^\epsilon$ given in \eqref{DOMAINS}.
The function $Z^\epsilon_m$ here is the solution of the problem
\begin{equation} \label{AUXSOL2}
\left\{
\begin{array}{c}
- \frac{\partial^2 Z^\epsilon}{\partial x_1^2} - \frac{1}{\epsilon^2} \frac{\partial^2 Z^\epsilon}{\partial x_2^2} 
= 0, \quad \textrm{ in } Q^\epsilon_m \\
\frac{\partial Z^\epsilon}{\partial N^\epsilon}=0, \quad \textrm{ on }  \partial Q^\epsilon_m \backslash \Gamma_m^\epsilon  \\
Z^\epsilon = \phi \, \omega^\epsilon,  \quad \textrm{ on } \Gamma_m^\epsilon
\end{array}
\right. 
\end{equation}
where $\Gamma_m^\epsilon$ is the top of rectangle $Q^\epsilon_m$.
Hereafter we may use notation $Z^\epsilon(x_1,x_2) = Z^\epsilon_m(x_1,x_2)$ whenever $(x_1,x_2) \in \widetilde \Omega^\epsilon_- \cap Q^\epsilon_m$.
Moreover, we observe that $\phi \, \omega^\epsilon|_{\Gamma_m^\epsilon} \in H^1(\Gamma_m^\epsilon)$, and auxiliary problems \eqref{AUXSOL} and \eqref{AUXSOL2} just differ by the condition on the top border $\Gamma_m^\epsilon$.

Now, let us to pass to the limit in functions $\omega^\epsilon$ and $\eta_1^\epsilon$.
Due to definition of $\omega_\epsilon$, we have for each $i=1,\ldots,N$, 
$$
\int_{A^\epsilon_k\cap \Omega_i} |\omega^\epsilon - x_1|^2 dx_1 dx_2 
= \int_{Y} \epsilon^3 |(P^i X_i)(y_1,y_2)|^2 dy_1 dy_2 \le \int_{Y^*_i} C \epsilon^3 |X_i(y_1,y_2)|^2 dy_1 dy_2
$$
and so, 
$$
\begin{gathered}
\int_{\Omega_i} |\omega^\epsilon - x_1|^2 dx_1 dx_2 
\approx  \sum_{k=1}^{\frac{C}{\epsilon l_h}}  \int_{Y^*_i} C \epsilon^3 |X_i(y_1,y_2)|^2 dy_1 dy_2 \\
\approx  \epsilon^2 \int_{Y^*_i} C |X_i(y_1,y_2)|^2 dy_1 dy_2 \rightarrow 0 \textrm{ as } \epsilon \to 0.
\end{gathered}
$$
Analogously,
\begin{eqnarray*}
\int_{A^\epsilon_k\cap \Omega_i} \Big|\frac{\partial }{\partial x_1} \left( \omega^\epsilon - x_1 \right) \Big|^2 dx_1 dx_2 
& = & \int_{Y}  \Big| \frac{\partial (P^i X_i)}{\partial y_1} (y_1,y_2) \Big|^2 \, \epsilon \, dy_1 dy_2 \\
& \le & \epsilon \int_{Y^*_i} C \Big|\frac{\partial X_i}{\partial y_1}(y_1,y_2)\Big|^2 dy_1 dy_2 
\end{eqnarray*}
and
\begin{eqnarray*}
\int_{A^\epsilon_k\cap \Omega_i} \Big|\frac{\partial }{\partial x_2} \left( \omega^\epsilon - x_1 \right) \Big|^2 dx_1 dx_2 
& = & \int_{Y}  \epsilon^3 \Big|\frac{\partial (P^i X_i)}{\partial y_2} (y_1,y_2)\Big|^2 \,  dy_1 dy_2 \\
& \le & \epsilon^3 \int_{Y^*_i}  C \Big|\frac{\partial X_i}{\partial y_2}(y_1,y_2)\Big|^2 dy_1 dy_2.
\end{eqnarray*}
Therefore
$$
\begin{gathered}
\int_{\Omega_i} \Big|\frac{\partial }{\partial x_1} \left( \omega^\epsilon - x_1 \right) \Big|^2 dx_1 dx_2 
\approx  \sum_{k=1}^{\frac{C}{\epsilon l_h}}  \epsilon \int_{Y^*_i} C\Big|\frac{\partial X_i}{\partial y_1}(y_1,y_2) \Big|^2 dy_1 dy_2 \\
\approx  \int_{Y^*_i} \tilde{C} \Big|\frac{\partial X_i}{\partial y_1}(y_1,y_2)\Big|^2 dy_1 dy_2 
\end{gathered}
$$
for all $\epsilon > 0$ and
$$
\begin{gathered}
\int_{\Omega_i} \Big|\frac{\partial }{\partial x_2} \left( \omega^\epsilon - x_1 \right) \Big|^2 dx_1 dx_2 
\le  \epsilon^2 \int_{Y^*_i} \tilde{C} \Big|\frac{\partial X_i}{\partial y_2}(y_1,y_2)\Big|^2 dy_1 dy_2 \to 0 \textrm{ as } \epsilon \to 0.
\end{gathered}
$$

Consequently, we can conclude for $\epsilon \to 0$
\begin{equation} \label{OMEGAL}
\omega^\epsilon \to x_1 \quad s-L^2(\Omega) \quad \textrm{ and } \quad 
w-H^1(\Omega_i), \quad i=1,\ldots,N,
\end{equation}
and
\begin{equation} \label{OMEGAE2}
\frac{\partial \omega^\epsilon}{\partial x_2} \to 0 \quad s-L^2(\Omega).
\end{equation}
In particular, $\omega^\epsilon$ is uniformly bounded in $H^1(\cup_{i=1}^N \Omega_i)$ for all $\epsilon>0$.

Next let $\widetilde{\eta}^\epsilon = \eta^\epsilon \chi_0$ be the extension by zero
of vector $\eta^\epsilon$ to the region $\Omega_0$ independent of $\epsilon$.
Since $X_i$ is $l_h$-periodic at variable $y_1$, we can apply the Average Theorem to \eqref{ETA1EX} obtaining
$$
\widetilde{\eta}_1^\epsilon(x_1,x_2) \rightharpoonup \frac{1}{l_h} \int_0^{l_h}  
\Big( 1 - \frac{\partial X_i}{\partial y_1}  (s,x_2) \Big) \chi_i(s,x_2)ds :=  \hat q_i(x_2),
\quad w^*-L^\infty(\xi_{i-1},\xi_i),
$$
where $\chi_i$ is the characteristic function of $Y^*_i$. 
Hence, we can argue as \eqref{chi} to get 
\begin{equation} \label{ETA}
\widetilde{\eta}_1^\epsilon \rightharpoonup \hat q, \quad w^*-L^\infty(\Omega_0),
\end{equation}
where $\hat q(x_1,x_2) \equiv \hat q_i(x_2)$, if $(x_1,x_2)\in \Omega_ i$, for $i=1,\ldots,N$.

Now we evaluate the test functions $\varphi^\epsilon$ as $\epsilon \to 0$.
It follows from estimate (\ref{basic-estimate}) that 
\begin{equation} \label{ESTX02}
\left\| \frac{\partial Z_m^\epsilon}{\partial x_1} \right\|^2_{L^2(Q^\epsilon_m)}
+ \frac{1}{\epsilon^2} \left\| \frac{\partial Z_m^\epsilon}{\partial x_2} \right\|^2_{L^2(Q^\epsilon_m)}
\leq C \eps^{\alpha-1} \left\|  \frac{\partial (\phi \, \omega^\epsilon)}{\partial x_1}  \right\|^2_{L^2(\Gamma_m^\epsilon)}.
\end{equation}
Denoting  $Q^\epsilon = \cup_{i=1}^{N_\epsilon} Q^\epsilon_m$, we have $\Omega^\epsilon_+= Q^\epsilon \cap \Omega^\epsilon$, and so, due to \eqref{TF2}, \eqref{OMEGAL} and \eqref{ESTX02}, 
\begin{equation} \label{ESTX2}
\begin{array}{l}
\displaystyle \left\| \frac{\partial \varphi^\epsilon}{\partial x_1} \right\|^2_{L^2(\Omega^\epsilon_-)}
+ \frac{1}{\epsilon^2} \left\| \frac{\partial \varphi^\eps}{\partial x_2} \right\|^2_{L^2(\Omega^\epsilon_-)}
= \sum^{m_\epsilon}_{m=0} \left( 
\left\| \frac{\partial \varphi^\epsilon}{\partial x_1} \right\|^2_{L^2(Q^\eps_m)}
+ \frac{1}{\epsilon^2} \left\| \frac{\partial \varphi^\epsilon}{\partial x_2} \right\|^2_{L^2(Q^\epsilon_m)}
\right) \\
\displaystyle \qquad \quad \leq C \, \epsilon^{\alpha -1} \, \max \left\{ \left\| \phi \right\|^2_\infty, \left\| \phi' \right\|^2_{\infty} \right\} \left\|  \omega^\epsilon  \right\|^2_{H^1(\cup_{i=1}^N \Omega_i)}  \\
\displaystyle \qquad \quad  \leq \widetilde C \, \epsilon^{\alpha -1},
\end{array}
\end{equation}
for some $\widetilde C>0$ independent of $\epsilon$.
Consequently, we can argue as in \eqref{L2CONV} to show 
\begin{equation} \label{TFCONV22}
\| \varphi^\epsilon - x_1 \, \phi \|_{L^2(\widetilde \Omega^\epsilon)} \to 0 \textrm{ as } \epsilon \to 0.
\end{equation}

Indeed, for $(x_1,x_2) \in \{ (x_1,x_2) \, | \, \gamma_{m,\epsilon}<x_1<\gamma_{m+1,\epsilon}, \, - G_\epsilon(x_1) < x_2 < H_1 \}$,
$$
\varphi^\epsilon(x_1,x_2) - \phi(x_1) \, \omega^\epsilon(x_1, -w^\epsilon_m) = \varphi^\epsilon(x_1,x_2) - \varphi^\epsilon(x_1,-w^\epsilon_m) 
= \int_{-w^\epsilon_m}^{x_2} \frac{\partial \varphi^\epsilon}{\partial x_2}(x_1,s) \, ds,
$$
where $w_m^\epsilon$ is the constant given by the step function $G^\epsilon_0$ in $(\gamma_{m,\epsilon}, \gamma_{m+1,\epsilon})$, that is, 
$$
w_m^\epsilon = G^\epsilon_0(x), \quad  \textrm{ for } x \in (\gamma_{m,\epsilon}, \gamma_{m+1,\epsilon}).
$$ 
Hence, if $\Gamma^\epsilon \subset \R^2$ is the graph of $-G_0^\epsilon$, we have $\varphi^\epsilon|_{\Gamma^\epsilon} = \varphi^\epsilon(x_1,-w^\epsilon_m) = \phi(x_1) \, \omega^\epsilon(x_1, -w^\epsilon_m)$ for $x_1 \in (\gamma_{m,\epsilon}, \gamma_{m+1,\epsilon})$, and so
\begin{eqnarray} \label{eq:SL0}
\int_{\widetilde \Omega^\epsilon} |\varphi^\epsilon -  \varphi^\epsilon|_{\Gamma^\epsilon} |^2 dx_1 dx_2 & \leq & \sum_{m=0}^{m_\epsilon} \int_{\gamma_{m,\epsilon}}^{\gamma_{m+1,\epsilon}}  \int_{-G_\epsilon(x_1)}^{H_1}  |x_2+w_m^\epsilon| \int_{-w_m^\epsilon}^{x_2} \left| \frac{\partial \varphi^\epsilon}{\partial x_2}(x_1,s) \right|^2 ds dx_2 dx_1 \nonumber \\
&  \leq & |H_1+G_1|^2 \int_0^1  \int_{-G_\epsilon(x_1)}^{H_1} \left| \frac{\partial \varphi^\epsilon}{\partial x_2}(x_1,s) \right|^2 ds dx_1 \nonumber
\\
& \leq & |H_1+G_1|^2 \left\| \frac{\partial \varphi^\epsilon}{\partial x_2} \right\|^2_{L^2(\widetilde \Omega^\epsilon)}.
\end{eqnarray}

On the other hand, 
\begin{eqnarray} \label{eq:SL00}
\int_{\widetilde \Omega^\epsilon} |\phi \, \omega^\epsilon -  \varphi^\epsilon|_{\Gamma^\epsilon} |^2 dx_1 dx_2 & \leq & \int_{\widetilde \Omega^\epsilon} |\phi \left( \omega^\epsilon -  \omega^\epsilon|_{\Gamma^\epsilon} \right)|^2 dx_1 dx_2 \nonumber \\ 
& \leq & |H_1+G_1|^2 \| \phi \|_\infty \left\| \frac{\partial \omega^\epsilon}{\partial x_2} \right\|^2_{L^2(\Omega)}.
\end{eqnarray}
Then, it follows from \eqref{eq:SL0} and \eqref{eq:SL00} that there exist $C>0$ independent of $\epsilon$ such that
\begin{eqnarray} \label{SLA}
\| \varphi^\epsilon - x_1 \, \phi \|_{L^2(\widetilde \Omega^\epsilon)}^2  & \leq & \| \varphi^\epsilon - \varphi^\epsilon|_{\Gamma^\epsilon} \|^2_{L^2(\widetilde \Omega^\epsilon)} + \| \varphi^\epsilon|_{\Gamma^\epsilon} - \phi \omega^\epsilon \|^2_{L^2(\widetilde \Omega^\epsilon)} + \| \phi \omega^\epsilon - x_1 \phi \|_{L^2(\widetilde \Omega^\epsilon)} \nonumber \\
& \leq & C \left\{ \left\| \frac{\partial \varphi^\epsilon}{\partial x_2} \right\|^2_{L^2(\widetilde \Omega^\epsilon)} + \left\| \frac{\partial \omega^\epsilon}{\partial x_2} \right\|^2_{L^2(\Omega)} +  \| \omega^\epsilon - x_1 \|_{L^2(\Omega)}\right\}.
\end{eqnarray}

Hence, we can conclude \eqref{TFCONV22} from \eqref{TF2}, \eqref{OMEGAL}, \eqref{OMEGAE2}, \eqref{ESTX2} and \eqref{SLA}.

Now, we are in condition to pass to the limit in \eqref{VFP3}. Taking as test functions $\varphi = \varphi^\epsilon$ and $\psi = \phi \, u^\epsilon$ in \eqref{VFP3}, we get
\begin{eqnarray}
& & \int_{\widetilde \Omega^\epsilon} \chi^\epsilon f^\epsilon \varphi^\epsilon \, dx_1 dx_2 \nonumber \\
& = & 
\int_{\widetilde \Omega^\epsilon} \Big\{ \widetilde{\frac{\partial u^\epsilon}{\partial x_1}} \frac{\partial \varphi^\epsilon}{\partial x_1}   + \frac{1}{\epsilon^2} \widetilde{\frac{\partial u^\epsilon}{\partial x_2}} \frac{\partial \varphi^\epsilon}{\partial x_2}  
+ \chi^\epsilon P_\epsilon u^\epsilon \varphi^\epsilon \Big\} dx_1 dx_2  \nonumber \\
& & \quad - \int_{\Omega_+^\epsilon} \Big\{  \eta^\epsilon_1 \frac{\partial (\phi u^\epsilon)}{\partial x_1} + \frac{1}{\epsilon^2} \eta^\epsilon_2 \frac{\partial (\phi u^\epsilon)}{\partial x_2} \Big\} dx_1 dx_2 \nonumber \\
& = & 
\int_{\widetilde \Omega^\epsilon_+} \Big\{ \widetilde{\frac{\partial u^\epsilon}{\partial x_1}} \phi'\omega^\epsilon + \phi \widetilde{\frac{\partial u^\epsilon}{\partial x_1}} \frac{\partial \omega^\epsilon}{\partial x_1}  + 
\frac{1}{\epsilon^2} \phi \widetilde{\frac{\partial u^\epsilon}{\partial x_2}} \frac{\partial \omega^\epsilon}{\partial x_2} \Big\} dx_1 dx_2 \nonumber \\
&  & + \int_{\widetilde \Omega^\epsilon_- } \Big\{ \widetilde{\frac{\partial u^\epsilon}{\partial x_1}} \frac{\partial \varphi^\epsilon}{\partial x_1}+ \frac{1}{\epsilon^2} \widetilde{\frac{\partial u^\epsilon}{\partial x_2}} \frac{\partial \varphi^\epsilon}{\partial x_2} \Big\} dx_1 dx_2
+ \int_{\widetilde \Omega^\epsilon} \chi^\epsilon {P_\epsilon u^\epsilon} \varphi^\epsilon \, dx_1 dx_2\nonumber \\
& & -  \int_{\Omega^\epsilon_+} \Big\{ {\eta_{1}^\epsilon} \phi'  u^\epsilon + {\eta_{1}^\epsilon} \phi  \frac{\partial u^\epsilon}{\partial x_1} 
+ \frac{1}{\epsilon^2}  {\eta_{2}^\epsilon} \phi \frac{\partial u^\epsilon}{\partial x_2}  \Big\} dx_1 dx_2. \label{Contas}
\end{eqnarray}
Consequently, due to \eqref{Contas}, \eqref{VETA} and $\Omega^\epsilon_+ \subset \widetilde \Omega^\epsilon_+$, we can rewrite \eqref{VFP3} as 
\begin{eqnarray} \label{VFP4}
& & \int_{\widetilde \Omega^\epsilon_+}  \widetilde{\frac{\partial u^\epsilon}{\partial x_1}} \, \omega^\epsilon \, \phi' \, dx_1 dx_2 + 
\int_{\widetilde \Omega^\epsilon_-}  \left\{ \widetilde{\frac{\partial u^\epsilon}{\partial x_1}} \frac{\partial \varphi^\epsilon}{\partial x_1} 
+ \frac{1}{\epsilon^2} \widetilde{\frac{\partial u^\epsilon}{\partial x_2}} \frac{\partial \varphi^\epsilon}{\partial x_2}  \right\} dx_1 dx_2 + \int_{\widetilde \Omega^\epsilon} \chi^\epsilon \, P_\epsilon u^\epsilon \, \varphi^\epsilon \, dx_1 dx_2 \nonumber \\
& &  \qquad  - \int_{\Omega^\epsilon_+}  \eta_1^\epsilon  \phi' \, u^\epsilon \,  dx_1 dx_2  = \int_{\widetilde \Omega^\epsilon} \chi^\epsilon \, f^\epsilon \varphi^\epsilon dx_1 dx_2, \quad  \forall \phi \in \mathcal{C}^\infty_0(\cup_{i=1}^{N} (\xi_{i-1},\xi_{i})). 
\end{eqnarray}

Let us now to evaluate \eqref{VFP4} when $\epsilon$ goes to zero.
\begin{itemize}
\item First integrand: we claim
\begin{equation} \label{INT12}
\int_{\widetilde \Omega^\epsilon_+}  \widetilde{\frac{\partial u^\epsilon}{\partial x_1}} \, \omega^\epsilon \, \phi' \, dx_1 dx_2 \to \int_{\Omega_0} \xi^* x_1 \phi' \, dx_1 dx_2, \quad \textrm{ as } \epsilon \to 0.
\end{equation}
Notice $\Omega_0 \subset \widetilde \Omega^\epsilon_+$, and so,
\begin{eqnarray*}
\int_{\widetilde \Omega^\epsilon_+}  \widetilde{\frac{\partial u^\epsilon}{\partial x_1}} \, \omega^\epsilon \, \phi' \, dx_1 dx_2 
& = & \int_{\Omega_0}  \widetilde{\frac{\partial u^\epsilon}{\partial x_1}} \, \omega^\epsilon \, \phi' \, dx_1 dx_2 
+ \int_{\widetilde \Omega^\epsilon_+ \setminus \Omega_0}  \widetilde{\frac{\partial u^\epsilon}{\partial x_1}} \, \omega^\epsilon \, \phi' \, dx_1 dx_2.
\end{eqnarray*}
Due to \eqref{WC0} and \eqref{OMEGAL}, it is easy to see 
$
\int_{\Omega_0}  \widetilde{\frac{\partial u^\epsilon}{\partial x_1}} \, \omega^\epsilon \, \phi' \, dx_1 dx_2 \to \int_{\Omega_0} \xi^* x_1 \phi' dx_1 dx_2.
$
On the other hand, it follows from \eqref{EST0}, \eqref{Omega0}, \eqref{G_0C}, \eqref{DOMAINS} and \eqref{OMEGAL} that 
\begin{eqnarray*}
\int_{\widetilde \Omega^\epsilon_+ \setminus \Omega_0}  \left| \widetilde{\frac{\partial u^\epsilon}{\partial x_1}} \, \omega^\epsilon \, \phi' \right| dx_1 dx_2
& \leq & \left\|\frac{\partial u^\epsilon}{\partial x_1} \right\|_{L^2(\Omega^\epsilon)} \| \phi' \omega^\epsilon \|_{L^2(\widetilde \Omega^\epsilon_+ \setminus \Omega_0)} \\
& \leq & \| u^\epsilon \|_{H^1(\Omega^\epsilon)}\| \omega^\epsilon \|_{H^1(\cup_i \widetilde \Omega^\epsilon_i)} \| \phi' \|_\infty^2 \left| \widetilde \Omega^\epsilon_+ \setminus \Omega_0 \right|^{1/2} \\
& \to & 0, \quad \textrm{ as } \epsilon \to 0,
\end{eqnarray*}
proving \eqref{INT12}.

\item Second integrand: we have   
\begin{equation} \label{INT22}
\int_{\widetilde \Omega^\epsilon_-} \Big\{ \widetilde{\frac{\partial u^\epsilon}{\partial x_1}} \frac{\partial \varphi^\eps}{\partial x_1} 
+ \frac{1}{\epsilon^2} \widetilde{\frac{\partial u^\epsilon}{\partial x_2}} \frac{\partial \varphi^\eps}{\partial x_2} \Big\} dx_1 dx_2
\to 0, \quad \textrm{ as } \epsilon \to 0.
\end{equation}
Indeed, it follows from estimates \eqref{ESTX2} and \eqref{EST0} that there exists $C>0$ such that 
\begin{eqnarray*} \label{eqd1}
& & \qquad \qquad\qquad \left|\int_{\widetilde \Omega^\epsilon_-} \left\{ \widetilde{\frac{\partial u^\epsilon}{\partial x_1}} \frac{\partial \varphi^\epsilon}{\partial x_1} 
+ \frac{1}{\epsilon^2} \widetilde{\frac{\partial u^\epsilon}{\partial x_2}} \frac{\partial \varphi^\epsilon}{\partial x_2} \right\} dx_1 dx_2\right| \nonumber  
\\
& & \quad \leq  \left( \int_{\Omega^\epsilon} \left\{ \left( \frac{\partial u^\epsilon}{\partial x_1} \right)^2
+ \frac{1}{\epsilon^2} \left( \frac{\partial u^\epsilon}{\partial x_2} \right)^2 \right\} dx_1 dx_2 \right)^{1/2}
\left( \int_{\widetilde \Omega^\epsilon_-} \left\{ \left( \frac{\partial \varphi^\epsilon}{\partial x_1} \right)^2
+ \frac{1}{\epsilon^2} \left( \frac{\partial \varphi^\epsilon}{\partial x_2} \right)^2 \right\} dx_1 dx_2 \right)^{1/2} \nonumber \\
& & \quad \leq  C \, \epsilon^{(\alpha - 1)/2} \to 0, \textrm{ as } \eps \to 0, 
\end{eqnarray*}
since $\alpha > 1$.

\par\medskip 
\item Third integrand: if $p(x)$ is that one defined in \eqref{def-p}, then
\begin{equation} \label{INT32}
\int_{\widetilde \Omega^\epsilon} \chi^\epsilon \, P_\epsilon u^\epsilon \, \varphi^\epsilon \, dx_1 dx_2 \to \int_0^1 p(x) \, u_0(x) \, x \phi(x) \, dx, \quad \textrm{ as } \epsilon \to 0.
\end{equation}
In fact, we can proceed as in \eqref{INT3}, since we have \eqref{L2CONV}, \eqref{TFCONV22}, $P_\epsilon u^\epsilon |_{\Omega^\epsilon} = u^\epsilon$, and 
\begin{eqnarray*}
\int_{\widetilde \Omega^\epsilon} \chi^\epsilon \, P_\epsilon u^\epsilon \, \varphi^\epsilon \, dx_1 dx_2 & = & 
\int_{\Omega^\epsilon} \left( u^\epsilon - u_0 \right) \, \varphi^\epsilon \, dx_1 dx_2
+ \int_{\Omega^\epsilon} u_0 \, \left( \varphi^\epsilon - x_1 \phi \right) \, dx_1 dx_2 \\
& &  \qquad + \int_{\Omega^\epsilon} u_0 \, x_1 \phi \, dx_1 dx_2.
\end{eqnarray*}

\item Fourth integrand:  Due to \eqref{L2CONV} and \eqref{ETA}, we can easily obtain  
\begin{equation} \label{INT42}
\int_{\Omega^\epsilon_+} \eta_1^\epsilon \, \phi' \, u^\epsilon \, dx_1 dx_2 \to \int_{\Omega_0} \hat q \, \phi' \, u_0 \, dx, \quad \textrm{ as } \epsilon \to 0,
\end{equation}
since $\Omega^\epsilon_+ \subset \Omega_0$, and 
$$
\int_{\Omega^\epsilon_+} \eta_1^\epsilon \, \phi' \, u^\epsilon \, dx_1 dx_2 = 
\int_{\Omega_0} \widetilde \eta_1^\epsilon \, \phi' \, P_\epsilon u^\epsilon \, dx_1 dx_2.
$$

\item Fifth integrand: we have 
\begin{equation} \label{INT52}
\int_{\widetilde \Omega^\epsilon} \chi^\epsilon \, f^\epsilon \, \varphi^\epsilon \, dx_1 dx_2 \to \int_0^1  \hat f(x) \, x \phi(x) \, dx, \quad \textrm{ as } \epsilon \to 0,
\end{equation}
which is derived from \eqref{FHD0} and \eqref{TFCONV22} in the same way that \eqref{INT4}.
\end{itemize}

Therefore, due to convergences obtained in \eqref{INT12}, \eqref{INT22}, \eqref{INT32}, \eqref{INT42} and \eqref{INT52}, we can pass to the limit in \eqref{VFP4} getting the following relation
\begin{equation} \label{LR}
\int_{\Omega_0} \xi^* \, x_1 \phi' \, dx_1 dx_2 + \int_0^1 p \, u_0 \, x \phi \, dx 
- \int_{\Omega_0} \hat q \, \phi' \, u_0 \, dx_1 dx_2 = \int_0^1 \hat f x \phi \, dx,
\end{equation}
for all $\phi \in \mathcal{C}^\infty_0(\cup_{i=1}^{N} (\xi_{i-1},\xi_{i}))$ where the step functions $p$ and $\hat q$ are given in \eqref{def-p} and \eqref{ETA} respectively by
\begin{equation} \label{SL123}
\begin{gathered}
p(x) = p_i =\frac{|Y_i^*|}{l_h} + \frac{1}{l_g} \int_0^{l_g} G_i(s) \, ds - G_{0,i},  \\
G_{0,i} = \min_{y \in \R} G_i(y), \\
\hat q(x, y) = \hat q_i(y) = \frac{1}{l_h} \int_0^{l_h}  
\Big( 1 - \frac{\partial X_i}{\partial y_1}  (s,y) \Big) \chi_i(s,y) \, ds,
\end{gathered}
\quad x \in (\xi_{i-1},\xi_i),
\end{equation}
for $i=1,\ldots,N$.
Thus, if we take $x_1 \phi(x_1)$ as a test function in \eqref{limitP}, we obtain 
\begin{equation} \label{SL01}
\int_{\Omega_0} \xi^* \frac{\partial }{\partial x_1} \left( x_1 \phi(x_1) \right) \, dx_1 dx_2 + \int_0^1 p \, u_0 \, x \phi \, dx = \int_0^1 \hat f \, x \phi \, dx.
\end{equation}

Combining \eqref{LR} and \eqref{SL01}, we get 
\begin{equation} \label{LRFF}
\int_{\Omega_0} \left\{ \hat q \, \phi' \, u_0 + \phi \, \xi^* \right\} dx_1 dx_2 = 0, \quad \forall \phi \in \mathcal{C}^\infty_0(\cup_{i=1}^{N} (\xi_{i-1},\xi_{i})).
\end{equation}
Hence, integrating by parts we have $\int_{\Omega_0} \hat q \, \phi' \, u_0 \, dx_1 dx_2 = - \int_{\Omega_0} \hat q \, \frac{\partial u_0}{\partial x_1} \, \phi \, dx_1 dx_2$, and so, we obtain via iterated integration and \eqref{LRFF} that
\begin{equation} \label{RLL}
\sum_{i=1}^N \int_{\xi_{i-1}}^{\xi_i} \int_{-G_{0,i}}^{H_1} \left\{ \hat q_i(x_2) \, \frac{\partial u_0}{\partial x_1}(x_1) - \xi^*(x_1,x_2) \right\} \phi(x_1) \, dx_1 dx_2 = 0, 
\end{equation}
for all $\phi \in \mathcal{C}^\infty_0(\cup_{i=1}^{N} (\xi_{i-1},\xi_{i}))$.

Then, if we consider the step function $q:(0,1) \mapsto \R$, $q(x)=q_i$ if $x \in (\xi_{i-1},\xi_i)$ with  
$$
q_i = \frac{1}{l_h} \int_{Y^*_i}  \Big( 1 - \frac{\partial X_i}{\partial y_1}  (y_1,y_2) \Big)  dy_1 dy_2,
$$
it follows from \eqref{RLL} and \eqref{SL123} that 
$$
\int_0^1 \left\{ q(x_1) \, \frac{\partial u_0}{\partial x_1}(x_1) - \left( \int_{-G_0(x_1)}^{H_1} \xi^*(x_1,x_2) \, dx_2 \right) \right\} \phi(x_1) \, dx_1 = 0, \quad \forall \phi \in \mathcal{C}^\infty_0(\cup_{i=1}^{N} (\xi_{i-1},\xi_{i})),
$$
where $G_0(x) = G_{0,i}$ if $x \in (\xi_{i-1},\xi_i)$.
Therefore, 
\begin{equation}\label{xi=q}
\int_{-G_0(x_1)}^{H_1}\xi^*(x_1,x_2) \, dx_2 = q(x_1) \frac{\partial u_0(x_1)}{\partial x_1},\quad \hbox{ a.e. } x_1 \in (0,1).
\end{equation}

Finally, since $\int_{\Omega_0} \xi^*(x_1,x_2) \, \phi'(x_1) \, dx_1 dx_2 = \int_0^1 \left( \int_{-G_0(x_1)}^{H_1} \xi^*(x_1,x_2) \, dx_2 \right) \phi'(x_1) \, dx_1$, we can plug this last equality \eqref{xi=q} in \eqref{limitP} getting our limit problem \eqref{VFPDL-piecewise0} write here as
$$
\sum_{i=1}^N\int_{\xi_{i-1}}^{\xi_i} \left\{ q_i \frac{\partial u_0}{\partial x_1}\frac{\partial \phi}{\partial x_1} 
+ p_i \, u_0 \, \phi \right\} dx_1 = \int_0^1  \hat f \, \phi \, dx_1,\quad \forall \phi \in H^1(0,1).
$$
\end{proof}


\section{The general homogenized limit} \label{GenPro}

Now we are in condition to get our main result concerned to the elliptic equation \eqref{EP} under hypothesis {\bf (H)}. Using approximation arguments on functions $G_\epsilon$ and $H_\epsilon$, the boundary perturbation result given by Proposition \ref{BPT}, and Lemma \ref{PPCT}, we are able to accomplish our goal using techniques previously discussed in \cite{AP2, AP3, AVP}.

\begin{theorem} \label{ET}
Let $u^\epsilon$ be the  solution of \eqref{EP} with $f^\epsilon \in L^2(\Omega^\epsilon)$ satisfying condition \eqref{FC}, and assume that the function 
\begin{equation} \label{FHD}
\hat f^\epsilon(x) = \int_{-G_\epsilon(x)}^{H_\epsilon(x)} f^\epsilon(x,s) \, ds, \quad x \in (0,1),
\end{equation}
satisfies that $\hat f^\epsilon \rightharpoonup \hat f$, w-$L^2(0,1)$, as $\epsilon \to 0$. 

Then, there exists $\hat u \in H^1(0,1)$, such that, if $P_\eps$ is the extension operator introduced in Lemma \ref{EOT}, then  
\begin{equation} \label{EC}
\| P_{\epsilon} u^\epsilon - \hat u \|_{L^2(\widetilde \Omega^\epsilon)} \to 0, \quad \textrm{ as } \epsilon \to 0, 
\end{equation}
where $\hat u$ is the unique solution of the Neumann problem
\begin{equation} \label{VFPDL}
\int_0^1 \Big\{ q(x)  \, u_x(x) \, \varphi_x(x) 
+ p(x) \, u(x) \, \varphi(x) \Big\} dx = \int_0^1  \, \hat f(x) \, \varphi(x) \, dx
\end{equation}
for all $\varphi \in H^1(0,1)$, where 
\begin{equation} \label{RPFL}
\begin{gathered}
q(x) =  \frac{1}{l_h} \int_{Y^*(x)} \left\{ 1 - \frac{\partial X(x)}{\partial y_1}(y_1,y_2) \right\} dy_1 dy_2, \\ 
p(x) = \frac{|Y^*(x)|}{l_h} + \frac{1}{l_g} \int_0^{l_g} G(x,y) \, dy - G_0(x), \\
G_0(x) = \min_{y \in \R} G(x,y),
\end{gathered}
\end{equation}
and $X(x)$ is the unique solution of the problem 
\begin{equation} \label{AUXG}
\left\{
\begin{array}{l}
- \Delta X(x)  =  0  \textrm{ in } Y^*(x)  \\
\frac{\partial X(x)}{\partial N}  =  0  \textrm{ on } B_2(x)  \\
\frac{\partial X(x)}{\partial N}  =  N_1 \textrm{ on } B_1(x)  \\
X(x) \textrm{ $l_h$-periodic on } B_0(x) \\
\int_{Y^*(x)} X(x) \; dy_1 dy_2  =  0  
\end{array}
\right.
\end{equation}
in the representative cell $Y^*(x)$ given by
$$
Y^*(x) = \{ (y_1,y_2) \in \R^2 \; | \; 0< y_1 < l_h, \quad -G_0(x) < y_2 < H(x,y_1) \}, 
$$
$B_0(x)$ is the lateral boundary, $B_1(x)$ is the upper boundary and $B_2(x)$ is the lower
boundary of $\partial Y^*(x)$ for each $x \in (0,1)$.
\end{theorem}

\begin{remark} \label{RLP}
\begin{itemize}
\item[i)] If the function $ q(x)$ is continuous, we have that the integral formulation (\ref{VFPDL}) 
is the weak formulation of problem 
$$
\left\{
\begin{array}{l}
\frac{1}{p(x)} \left( q(x) \, u_x(x) \right)_x + u(x) = f(x), \quad x \in (0,1), \\
u_x(0) = u_x(1) = 0, 
\end{array}
\right.
$$
with $f(x)=\hat f(x)/p(x)$.

\item[ii)] Also, if we initially assume that $f^\epsilon$ does not depend on the vertical variable $y$, that is, $f^\epsilon(x,y)=f_0(x)$, then it is not difficult to see that 
$$
\hat f^\epsilon(x) = \left( H_\epsilon(x) + G_\epsilon(x) \right) f_0(x)
$$ 
and so, due to the Average Theorem discussed for example in \cite[Lemma 4.2]{AP3},  
$$
H_\epsilon(x) + G_\epsilon(x) \rightharpoonup \frac{1}{l_h} \int_0^{l_h} H(x,y) \, dy + \frac{1}{l_g} \int_0^{l_g} G(x,y) \, dy, \quad  w^* - L^\infty(0,1),
$$ 
as $\epsilon \to 0$.
Thus, 
$H_\epsilon(x) + G_\epsilon(x) \rightharpoonup p(x)$, $ w^* - L^\infty(0,1)$, and $\hat f(x)=p(x)f_0(x)$ as discussed in \eqref{SL@@}. 

\item[iii)] Moreover, if we combine the uniform estimate \eqref{EST0} in $H^1(\Omega^\epsilon)$ and Lemma \ref{EOT}, we obtain $P_{\epsilon} u^\epsilon$ uniformly bounded in $H^1(\widetilde \Omega^\epsilon)$. Hence, from the convergence result \eqref{EC} in $L^2(\widetilde \Omega^\epsilon)$, we can obtain by interpolation \cite[Section 1.4]{Henry} that 
$$
\| P_{\epsilon} u^\epsilon - \hat u \|_{H^{\beta}(\widetilde \Omega^\epsilon)} \to 0, \quad \textrm{ as } \epsilon \to 0,
$$ 
for all $0 \leq \beta < 1$.
\end{itemize}
\end{remark}

\begin{remark} \label{SofQ}
As a matter of fact, we have that the problem \eqref{VFPDL} is well posed in the sense that the diffusion coefficient $q$ is uniformly positive and smooth in $(0,1)$.
For see this, we use the variational formulation of the auxiliary problem \eqref{AUXG} given by the bilinear form
$$
a_{Y^*}(\varphi,\phi) = \int_{Y^*(x)} \nabla \varphi \cdot \nabla \phi \, dy_1 dy_2, \quad \forall \varphi, \phi \in V,
$$
defined in the Hilbert space $V$ given by $V = V_{Y^*} / \R$, 
$$
V_{Y^*} = \{ \varphi \in H^1(Y^*) \; | \; \varphi \text{ is } l_h \text{- periodic} \textrm{ in variable $y_1$}  \},
$$
with norm 
$$
\| \varphi \|_V = \left( \int_{Y^*} \left| \nabla \varphi \right|^2 \, dy_1 dy_2 \right)^{1/2}.
$$ 
Due to hypothesis {\bf (H)}, we have that the representative cell $Y^* = Y^*(x)$ is defined for all $x \in [0,1]$.
Hence, for all $\phi \in V$ and $x \in [0,1]$, we have  
$$
a_{Y^*}(X,\phi) = \int_{{B_{1}}} \, N_1 \phi \, dS,
$$
where $B_1(x)$ is the upper boundary of the basic cell $Y^*$.
Consequently, $y_1 - X(x)$ satisfies
\begin{equation} \label{EQB0}
a_{Y^*}(y_1 - X, \phi) = \int_{{B_{1}}} N_1 \phi \, dS - \int_{Y^*} \phi \, dy_1 dy_2 
- \int_{{B_{1}}} N_1 \, \phi \, dS = 0, \quad \forall \phi \in V, 
\end{equation}
since $\phi$ is $l_h$-periodic in the $y_1$ variable.
Also, we have that
\begin{eqnarray} \label{EQB1}
q \, l_h
& = &  \int_{Y^*} \frac{\partial}{\partial y_1}(y_1 - X(y_1,y_2)) \, \frac{\partial y_1}{\partial y_1} \, dy_1 dy_2
= \int_{Y^*} \nabla(y_1 - X(y_1,y_2)) \cdot \nabla y_1 \, dy_1 dy_2 \nonumber \\
& = & a_{Y^*}(y_1 - X, y_1).
\end{eqnarray}
Hence, due to relation (\ref{EQB0}) with $\phi = - X$, and identity (\ref{EQB1}), we get for all $x \in [0,1]$
\begin{eqnarray*} \label{AP}
q \, l_h & = & a_{Y^*}(y_1 - X, y_1) + a_{Y^*}(y_1 - X, - X) \\
 & = & a_{Y^*}(y_1 - X, y_1 - X) = \|y_1-X\|_V > 0.
\end{eqnarray*}
Thus, since $\|y_1-X\|_V$ is a continuous function in $[0,1]$ (see \cite[Proposition A.1]{AP3}) and $|Y^*| > 0$, we have that the homogenization coefficient $q$ is uniformly positive and continuous in $[0,1]$ implying that, for example, the problem \eqref{VFPDL} is well posed being $\hat u$ its unique solution.
\end{remark}

We provide now a proof of the Theorem \ref{ET}.

\begin{proof}

From estimate (\ref{EST0}) and Lemma \ref{EOT}, we have $u^\epsilon|_{\widehat \Omega_0} \in H^1(\widehat \Omega_0)$ satisfying 
$$
\begin{gathered}
\| P_{\epsilon} u^\epsilon \|_{L^2(\widehat \Omega_0)}, \Big\| \frac{\partial P_{\epsilon} u^\epsilon}{\partial x_1} \Big\|_{L^2(\widehat \Omega_0)} \textrm{ and }
\frac{1}{\epsilon} \Big\| \frac{\partial P_{\epsilon} u^\epsilon}{\partial x_2} \Big\|_{L^2(\widehat \Omega_0)} \le M 
\textrm{ for all } \epsilon > 0,
\end{gathered}
$$
with $M>0$ independent of $\epsilon$, where $\widehat \Omega_0 \subset \widetilde \Omega^\epsilon$ is given here by 
$
\widehat  \Omega_0 = (0,1) \times (-G_0, H_1).
$
Then, there exists $u_0\in H^1(\widehat \Omega_0)$ and a subsequence, still denoted by $P_{\epsilon} u^\epsilon$, satisfying 
\begin{equation} \label{LEO-bis}
\begin{gathered}
P_\epsilon u^\epsilon \rightharpoonup u_0 \quad w-H^1(\widehat \Omega_0), 
\quad \textrm{ and } \quad 
\frac{\partial P_\epsilon u^\epsilon}{\partial x_2} \rightarrow 0 \quad s-L^2(\widehat \Omega_0).
\end{gathered}
\end{equation}
Thus, arguing as in \eqref{u0x}, we get $u_0(x_1,x_2) = u_0(x_1)$ on $\widehat \Omega_0$, and so, $u_0 \in H^1(0,1)$.

We will show that $u_0$ satisfies the Neumann problem \eqref{VFPDL} using a discretization argument on the oscillating boundary of the domain.

For this, let us fix a small $\delta>0$ and consider piecewise periodic functions $G^\delta (x,y)$ and $H^\delta (x,y)$ as described at the beginning of Section \ref{PWPC} satisfying hypothesis {\bf (H)}  and condition  
$$
\begin{gathered}
0\leq G^\delta(x,y)-G(x,y) \leq \delta, \\ 
0\leq H^\delta(x,y)-H(x,y) \leq \delta, 
\end{gathered}
\quad \forall (x,y) \in [0,1] \times \R.
$$ 

In order to construct these functions, we may proceed as follows. 
The functions $G$ and $H$ are uniformly $C^1$ in each interval $(\xi_{i-1},\xi_i) \times (0,1)$ being periodic in the second variable. In particular, for $\delta>0$ small enough and for a fixed $z\in (\xi_{i-1},\xi_i)$ we have that there exists a small interval $(z-\eta,z+\eta)$ with $\eta$ depending only on $\delta$ such that $|G(x,y)-G(z,y)|+|\partial_y G(x,y)-
\partial_y G(z,y)|<\delta/2$ and $|H(x,y)-H(z,y)|+|\partial_y H(x,y)-
\partial_y H(z,y)|<\delta/2$ for all $x\in (z-\eta,z+\eta)\cap (\xi_{i-1},\xi_i)$ and for all $y\in \R$. This allows us to select a finite number of points:  
$\xi_{i-1}=\xi_{i-1}^1< \xi_{i-1}^2<\ldots<\xi_{i-1}^r=\xi_i$ such that $\xi_{i-1}^r-\xi_{i-1}^{r-1}<\eta$ and therefore, defining
$G^\delta(x,y)=G(\xi_{i-1}^r,y)+\delta/2$ and $H^\delta(x,y)=H(\xi_{i-1}^r,y)+\delta/2$ for all $x\in (\xi_{i-1}^r,\xi_{i-1}^{r+1})$ we have that $0\leq G^\delta(x,y)-G(x,y)\leq \delta$, $|\partial_y G^\delta(x,y)-\partial_y G(x,y)|\leq \delta$, $0\leq H^\delta(x,y)-H(x,y)\leq \delta$ and $|\partial_y H^\delta(x,y)-\partial_y H(x,y)|\leq \delta$ for all
$(x,y)\in (\xi_{i-1},\xi_i)\times \R$. 

Note that this construction can be done for all $i=1,\ldots, N$. In particular, if we rename all the points $\xi_i^k$ constructed above by $0=z_0<z_1<\ldots<z_m=1$ observing that $m=m(\delta)$, then the functions $G^\delta$ and $H^\delta$ satisfy $G^\delta(x,y)=G^\delta_i(y)$ and $H^\delta(x,y)=H^\delta_i(y)$ in $(x,y) \in (z_{i-1},z_i) \times \R$, $i=1,\ldots, m$, where $G^\delta_i$ and $H^\delta_i$ are $C^1$-functions, $l_g$ and $l_h$-periodic respectively. At each point $z_i$, we can set $G^\delta$ and $H^\delta$ as the minimum value of the lateral limit in $z_i$.

Let us now to denote $G_\epsilon^\delta(x)=G^\delta(x,x/\epsilon^\alpha)$, $\alpha > 1$, and $H_\epsilon^\delta(x)=H^\delta(x,x/\epsilon)$, aiming to introduce the following oscillating domains
$$
\begin{gathered}
\Omega^{\epsilon,\delta} = \{ (x,y) \in \R^2 \; | \;  x \in (0,1),  \;
 - G_{\epsilon}^{\delta}(x) < y < H_{\epsilon}^{\delta}(x) \}, \\
\widetilde \Omega^{\epsilon,\delta} = \{ (x,y) \in \R^2 \; | \;  x \in (0,1),  \;
 - G_{\epsilon}^{\delta}(x) < y < H_1 \}.
\end{gathered}
$$

Since $H_{\epsilon}^{\delta}$ satisfies the hyphoteses of Lemma \ref{EOT}, there exists an extension operator
$$
P_{\epsilon,\delta} \in \mathcal{L}(L^p(\Omega^{\epsilon,\delta}),L^p(\widetilde\Omega^\delta)) 
  \cap \mathcal{L}(W^{1,p}(\Omega^{\epsilon,\delta}),W^{1,p}(\widetilde\Omega^\delta))
$$
satisfying the uniform estimate \eqref{EQOP} with $\eta(\epsilon) \sim 1/\epsilon$.

Taking $f^\epsilon \in L^2(\Omega^\epsilon)$ satisfying $\|f^\epsilon\|_{L^2(\Omega^\epsilon)}\leq C$, and extend it by 0 outside $\Omega^\epsilon$, and still denoting the extended function again by $f^\epsilon$, and using that $G_\delta \geq G$ and $H_\delta \geq H$, we have that 
$\hat f^\epsilon_\delta(x)=\int^{H^\delta_\epsilon(x)}_{-G_\epsilon^\delta(x)} f^\epsilon (x,y)dy=\int^{H_\epsilon(x)}_{-G_\epsilon(x)} f^\epsilon(x,y) dy = \hat f^\epsilon(x)$ and by hypothesis, we have that $\hat f^\eps_\delta\equiv \hat f^\eps\rightharpoonup \hat f$ w-$L^2(0,1)$. 

Therefore, it follows from Theorem \ref{PPCT} that for each $\delta > 0$ fixed, 
there exist $u^\delta \in H^1(0,1)$
such that the solutions $u^{\epsilon,\delta}$ 
of \eqref{EP} in $\Omega^{\epsilon,\delta}$ satisfy
\begin{equation} \label{UDC}
\| P_{\epsilon,\delta} u^{\epsilon,\delta} - u^\delta \|_{L^2(\widetilde \Omega^{\epsilon,\delta})} \to 0, \quad \textrm{ as } \epsilon \to 0,
\end{equation}
where $u^\delta \in H^1(0,1)$ is the unique solution of the Neumann problem
\begin{equation} \label{VFPD}
\int_0^1 \Big\{ q^\delta(x) \; u_x^\delta(x) \, \varphi_x(x) 
+ p^\delta(x) \, u^\delta(x) \, \varphi(x) \Big\} dx = \int_0^1  \, \hat f(x) \, \varphi(x) \, dx, \quad \forall \varphi \in H^1(0,1),
\end{equation}
where $q^\delta$ and, $p^\delta: (0,1) \mapsto \R$ are strictly positive functions, locally constant, given by
$$
\left\{ 
\begin{gathered}
q^\delta(x) = \frac{1}{l_h} \int_{Y^*_i} \Big\{ 1 - \frac{\partial X_i}{\partial y_1}(y_1,y_2) \Big\} dy_1 dy_2, \\
p^\delta(x) =\frac{|Y_i^*|}{l_h} + \frac{1}{l_g} \int_0^{l_g} G_i^\delta(s) \, ds - G^\delta_{0,i},  \\
G_{0,i}^\delta = \min_{y \in \R} G_i^\delta(y),
\end{gathered}
\right. 
\quad x \in (z_{i-1},z_i),
$$
where the function $X_i$ is the unique solution of (\ref{AUX}) in the representative cell $Y^*_i$ given by 
$$
Y^*_i = \{ (y_1,y_2) \in \R^2 \; | \; 0< y_1 < l_h, \quad - G_{0,i}^\delta < y_2 < H_i^\delta(y_1) \}, \quad i=1, \ldots, m.
$$

Now, let us pass to the limit in $\eqref{VFPD}$ as $\delta \to 0$.
To do this, we consider the functions $q^\delta$ and $p^\delta$ defined in $x \in (0,1)$
and the functions $q$ and $p$ defined in \eqref{RPFL}.
We have that $q^\delta$ and $p^\delta$ converge to $q$ and $p$ uniformly in $(0,1)$.
The uniform convergence of $q^\delta$ to $q$ in $(0,1)$ follows from \cite[Proposition A.1]{AP2}. The uniform convergence of $p^\delta$ to $p$ follows from the uniform convergence of $G^\delta$ and $H^\delta$ to $G$ and $H$ respectively as $\delta\to 0$. 

Therefore, we obtain from \cite[p. 8]{BLP} or \cite[p. 1]{CP} the following limit variational formulation:
to find $u \in H^1(0,1)$ such that 
\begin{equation} \label{VFPDL-F}
\int_0^1 \Big\{ q(x) \; u_x(x) \, \varphi_x(x) 
+ p(x) \, u(x) \, \varphi(x) \Big\} dx = \int_0^1  \, \hat f(x) \, \varphi \, dx
\end{equation}
for all $\varphi \in H^1(0,1)$.
Hence, there exists $u^* \in H^1(0,1)$ such that 
\begin{equation} \label{DCU}
u^\delta \to u^* \textrm{ in } H^1(0,1)
\end{equation}
where $u^*$ is the unique solution of the Neumann problem (\ref{VFPDL-F}).

We will complete the proof showing that $u^* = u_0$ in $(0,1)$, where $u_0$ is the function obtained in \eqref{LEO-bis}.
In order to do so, we observe that $\|u^*-u_0\|_{L^2(0,1)}^2 = \left\{ H_1 + G_0\right\}^{-1}\|u^*-u_0\|_{L^2(\widehat \Omega_0)}^2$ and therefore, to show that $u^*=u_0$ it is enough to show that $\|u^*-u_0\|_{L^2(\widehat \Omega_0)}^2=0$. Adding and
subtracting appropriate functions,  we have for all $\epsilon$ and $\delta > 0$ that
\begin{equation}\label{EQF001}
\begin{array}{l}
\| u^* - u_0 \|_{L^2(\widehat \Omega_0)} \leq 
\| u^* - u^{\delta}  \|_{L^2(\widehat \Omega_0)}+ 
\| u^\delta - u^{\epsilon,\delta}  \|_{L^2(\widehat \Omega_0)}\\ \\
\qquad \qquad+ \|  u^{\epsilon,\delta} - u^\epsilon   \|_{L^2(\widehat \Omega_0)} 
+ \|  u^\epsilon - u_0  \|_{L^2(\widehat \Omega_0)}.
\end{array}
\end{equation}

Let $\eta$ be now a positive small number. 
From \eqref {DCU} and Theorem \ref{BPT}, we can choose a $\delta>0$ fixed and small such that 
$\| u^* - u^{\delta}  \|_{L^2(\Omega_0)} \leq \eta$ and $\|  u^{\epsilon,\delta} - u^\epsilon   \|_{L^2(\Omega_0)} \leq \eta$ uniformly for all $\epsilon>0$.  
For this particular value of $\delta$, we can choose, by \eqref{UDC},  $\epsilon_1>0$ small
enough such that $\| u^\delta - u^{\epsilon,\delta}  \|_{L^2(\Omega_0)}\leq \eta$  for $0<\epsilon<\epsilon_1$. 
Moreover, from 
\eqref{LEO-bis}, we have that there exists $\epsilon_2>0$ such that $ \|  u^\epsilon - u_0  \|_{L^2(\Omega_0)} \leq \eta$ for all $0<\epsilon<\epsilon_2$.   
Hence with $\epsilon=\min\{\epsilon_1,\epsilon_2\}$ applied to \eqref{EQF001}, we get
$\| u^* - u_0 \|_{L^2(\Omega_0)} \leq 4\eta$. Since $\eta$ is arbitrarily small, we get $\|u^*-u_0\|_{L^2(\widehat \Omega_0)}^2=0$.

\end{proof}


\section{Convergence of Linear Semigroups} \label{SCLSg}

In order to accomplish our goal, we consider here the linear parabolic problems associated to the perturbed equation \eqref{RP} and its limit problem \eqref{LRP} in the abstract framework given by \cite{Hale, Henry} to show that, under an appropriated notion of convergence, the linear semigroup given by \eqref{RP} converges to the one established by \eqref{LRP} as $\epsilon \to 0$. 
The convergence concept that we adopt here was first introduced in the works \cite{St1972a, St1972b,St1972c, Va1,Va2} and then successfully  applied in \cite{AC, ACL1,ACL2, ACL3, ACP} to concrete perturbation problems given by parabolic equations. 

To do so, let us first consider a family of Hilbert spaces $\{ Z_\epsilon \}_{\epsilon > 0}$
defined by $Z_\epsilon = L^2(\Omega^\epsilon)$ under the canonical inner product 
$$
( u, v )_{\epsilon} = \int_{\Omega^\epsilon} u(x_1,x_2) \, v(x_1,x_2) \, dx_1 dx_2 
$$ 
and let $Z_0 = L^2(0,1)$ be the limiting Hilbert space with the inner product $( \cdot, \cdot )_0$ given by
$$
( u, v )_{{0}} = \int_0^1 p(x) \, u(x) \, v(x) \, dx 
$$
where
$$
p(x) = \frac{|Y^*|}{l_h} + \frac{1}{l_g} \int_0^{l_g} G(x, y) \, dy - G_0(x)
$$
is the positive function previously defined in \eqref{RPFL}.

We write the elliptic problem \eqref{VFP} as an abstract equation $L_\epsilon u = f^{\epsilon}$ where $L_\epsilon: \mathcal{D}(L_\epsilon) \subset L^2(\Omega^\epsilon) \mapsto L^2(\Omega^\epsilon)$ is the self adjoint, positive linear operator with compact resolvent 
\begin{equation} \label{AEPS}
\begin{gathered}
\mathcal{D}(L_\epsilon) = \left\{ u \in H^2(\Omega^\epsilon) \, | \, 
\frac{\partial u}{\partial x_1} N_1^\epsilon + \frac{1}{\epsilon^2} \frac{\partial u}{\partial x_2}N_2^\epsilon = 0
\textrm{ on } \partial \Omega^\epsilon \right\} \\
L_\epsilon u =  - \frac{\partial^2 u}{{\partial x_1}^2} - 
\frac{1}{\epsilon^2} \frac{\partial^2 u}{{\partial x_2}^2} + u, \quad u \in \mathcal{D}(L_\epsilon).
\end{gathered}
\end{equation}

Analogously, we associate the limit elliptic problem \eqref{VFPDL} to the \emph{limit linear operator} 
$L_0: \mathcal{D}(L_0) \subset Z_0 \mapsto Z_{0}$ defined by
\begin{equation} \label{AO}
\begin{gathered}
\mathcal{D}(L_0) = \left\{ u \in H^2(0,1) \, | \, u'(0) = u'(1) = 0 \right\} \\
L_0 u = - \frac{1}{p(x)} \left( q(x) u_x \right)_x  + u, \quad u \in \mathcal{D}(L_0)
\end{gathered}
\end{equation}
where $p$ and $q$ are the homogenized coefficients established in \eqref{RPFL}.
Due to Remark \ref{SofQ}, it is clear that $L_0$ is a positive self adjoint operator with compact resolvent. 

In order to simplify the notation, we denote by $Z_{\epsilon}^{\alpha}$ the fractional power scale 
associated to operators $L_{\epsilon}$ with $0\leqslant \alpha \leqslant 1$ and $0\leqslant \epsilon \leqslant 1$. 
We also write $Z_{\epsilon}:=Z_{\epsilon}^0$ for all $0\leqslant \epsilon \leqslant 1$. Notice that $ Z_{\epsilon}^{1/2}$ is the Sobolev Space $H^1(\Omega^\epsilon)$ with norm 
$$
\|u\|_{Z_{\epsilon}^{1/2}}^2=\left\| \frac{\partial u} {\partial x_{1}}\right\|_{Z_{\epsilon}}^2 + \frac{1}{\epsilon^2}  \left\|\frac{\partial u} {\partial x_{2}}\right\|_{Z_{\epsilon}}^2  + \left\| u \right\| _{Z_{\epsilon}}^2.
$$

\begin{remark}\label{remark:exbound}
It follows from Remark $\ref{rem:eqnorm }$ that the extension operators $P_{\epsilon} \in \mathcal{L}(Z_{\epsilon}^{{1/2}}, H^1(\widetilde \Omega^\epsilon))$ $\cap$ $\mathcal{L}(Z_{\epsilon}, L^2(\widetilde \Omega^\epsilon))$ given by Lemma \ref{EOT} are uniformly bounded in $\epsilon$. Therefore, we obtain by interpolation that 
$$ 
\sup_{0\leqslant \epsilon \leqslant 1}\|P_{\epsilon}\|_{\mathcal{L}(Z_{\epsilon}^\alpha, H^{2\alpha}(\widetilde \Omega^\epsilon))} < \infty, \quad 0 \leqslant \alpha \leqslant \frac{1}{2}.
$$
\end{remark}

So far, we have passed to limit in the variational problem \eqref{VFP} as $\epsilon \to 0$ getting the limit equation \eqref{VFPDL}.
Here, we apply the concept of \emph{compact convergence} to obtain convergence properties of the linear semigroups generated by the operators $L_\epsilon$ and $L_0$.

For this, let us consider the family of linear continuous operators $E_\epsilon: Z_0 \mapsto Z_\epsilon$ given by
$$
(E_\epsilon u)(x_1,x_2) = u(x_1) \textrm{ on } \Omega^\epsilon
$$
for each $u \in Z_0$.
Since
\begin{eqnarray*}
\| E_\epsilon u \|^2_{Z_\epsilon}  =  \int_{\Omega^\epsilon} u^2(x_1) \, dx_1 dx_2 
 =  \int_0^1 \left\{ H_\epsilon(x_1) + G_\epsilon(x_1) \right\} u^2(x_1) \, dx_1,
\end{eqnarray*}
we have that $\| E_\epsilon u \|_{Z_\epsilon} \to \| u \|_{Z_0} \textrm{ as } \epsilon \to 0$.
Observe that $E_\epsilon$ is a kind of inclusion operator from $Z_0$ into $Z_\epsilon$.
Similarly, we can consider $E_{\epsilon}: L^1_{0} \to L^1_{\epsilon}$, and so, taking in $ L^1_{0}$ the equivalent norm $\|u\|_{Z^1_0}= \| - u_{xx} + u \|_{ Z_0}$,  we obtain
$$
\|E_{\epsilon} u\|_{L^{1}_\epsilon} \to \| u\|_{L^{1}_{0}}.
$$
Consequently, since 
$$
\sup_{0 \leqslant \epsilon \leqslant 1} \{\| E_\epsilon \|_{\mathcal{L}(Z_0,Z_\epsilon)},  \|E_\epsilon \|_{\mathcal{L}(L^1_0,L^1_\epsilon)} \} < \infty,
$$
we get by interpolation that
$$
C  = \sup_{\epsilon>0}\| E_{\epsilon} \|_{\mathcal{L}(Z_{0}^\alpha,Z_{\epsilon}^\alpha)} < \infty 
\, \,  \textrm{ for } 0\leqslant \alpha \leqslant 1.
$$ 

Now we are in condition to set the following concepts of convergence, compactness and compact convergence of operators 
associated to the family of operators $\{ E_\epsilon \}_{\epsilon > 0}$.
\begin{definition}
We say that a sequence of elements $\{ u^\epsilon \}_{\epsilon > 0}$ with $u^\epsilon \in Z_\epsilon$ is \emph{E-convergent} to $u \in Z_0$, if $\| u^\epsilon - E_\epsilon u \|_{Z_\epsilon} \to 0$
as $\epsilon \to 0$. We write $u^\epsilon \stackrel{E}{\rightarrow} u$.
\end{definition}

\begin{definition} A sequence $\{ u_n \}_{n \in \N}$ with $u_n \in Z_{\epsilon_n}$ is said to be 
\emph{E-precompact} if for any subsequence $\{ u_{n'} \}$ there exist a subsequence $\{ u_{n''} \}$
and $u \in Z_0$ such that $u_{n''} \stackrel{E}{\rightarrow} u$ as $n'' \to \infty$.
A family $\{ u^\epsilon \}_{\epsilon > 0}$ is called \emph{pre-compact} if each sequence $\{ u_{\epsilon_n} \}$,
with $\epsilon_n \to 0$, is pre-compact.
\end{definition}

\begin{definition} We say that a family of operators $\{ B_\epsilon \in \mathcal{L}(Z_\epsilon) \, | \, \epsilon > 0 \}$ 
\emph{E-converges} to $B \in \mathcal{L}(Z_0)$ as $\epsilon \to 0$, if $B_\epsilon f^\epsilon \stackrel{E}{\rightarrow} B f$
whenever $f^\epsilon \stackrel{E}{\rightarrow} f \in Z_0$. We write $B_\epsilon \stackrel{EE}{\rightarrow} B$.
\end{definition}

\begin{definition} We say that a family of compact operators $\{ B_\epsilon \in \mathcal{L}(Z_\epsilon) \, | \, \epsilon > 0 \}$ 
converges compactly to a compact operator $B \in \mathcal{L}(Z_0)$, if for any family $\{ f^\epsilon \}_{\epsilon > 0}$ with 
$\| f^\epsilon \|_{Z_\epsilon} \le 1$, we have that the family $\{ B_\epsilon f^\epsilon \}$ is E-precompact and 
$B_\epsilon \stackrel{EE}{\rightarrow} B$. We write $B_\epsilon \stackrel{CC}{\rightarrow} B$.
\end{definition}

We finally note this notion of convergence can also be extended to sets following \cite{ACL3, ACP}.

\begin{definition} \label{USCC}
Let $\mathcal{O}_\epsilon \subset Z_\epsilon^\alpha$, $\epsilon \in [0,1]$, and $\mathcal{O}_0 \subset Z_0^\alpha$, $\alpha \in [0,1)$. We say that the family of sets $\set{\mathcal{O}_\epsilon}_{\epsilon \in [0,1]}$ is $E$-upper semicontinuous or just upper semicontinuous at $\epsilon=0$ if 
$$
\sup_{w^\epsilon \in \mathcal{O}_\epsilon} 
\Big[ \inf_{w \in \mathcal{O}_0} \left\{ \| w^\epsilon - E_\epsilon w \|_{Z_\epsilon^\alpha}  \right\} \Big]
\to 0, 
\textrm{ as } \epsilon \to 0.
$$
Let us also recall an useful characterization of upper semicontinuity of sets: 
If any sequence $\set{u^\epsilon} \subset \mathcal{O}_\epsilon$ has a $E$-convergent subsequence with limit belonging  to $\mathcal{O}$, then $\set{\mathcal{O}_\epsilon}$ is $E$-upper semicontinuous at zero.
\end{definition}

The following result is basically Theorem \ref{ET} written according to previous framework.

\begin{corollary} \label{CCAE}
The family of compact operators $\{ L_\epsilon^{-1} \in \mathcal{L}(Z_\epsilon) \}_{\epsilon > 0}$ converges compactly to the
compact operator $L^{-1}_0 \in \mathcal{L}(Z_0)$ as $\epsilon \to 0$.
\end{corollary}
\begin{proof}
Let us take $\{ f^{\epsilon} \}_{\epsilon > 0} \subset Z_\epsilon$ with $\| f^{\epsilon} \|_{Z_\epsilon} \le 1$ and define $u^\epsilon = L_\epsilon^{-1} f^{\epsilon}$.
Then, $L_\epsilon u^\epsilon = f^{\epsilon}$ and $u^\epsilon$ satisfies the problem \eqref{VFP}.
Consequently, we get from Theorem \ref{ET} and Remark \ref{RLP} that there exist $f_0 \in Z_{0}$ and $u_0 \in H^1(0,1)$ such that $L_0 u_0 = f_0$, $\| P_\epsilon u^\epsilon - u_0\|_{L^2(\widetilde \Omega^\epsilon)} \to 0$, as $\epsilon \to 0$, where $u_0(x_1,x_2) = u_0(x_1)$. Recall that $P_{\epsilon}$ is the extension operator given by Lemma \ref{EOT}.
Hence, we can conclude from the inequality
$$
\| u^\epsilon - E_\epsilon u_0 \|_{Z_\epsilon}  =  \| \left( P_\epsilon u^\epsilon - u_0 \right) |_{\Omega^\epsilon} \|_{Z_\epsilon} 
 \le  \| P_\epsilon u^\epsilon - u_0 \|_{L^2(\widetilde \Omega^\epsilon)} 
$$
that $u^\epsilon \stackrel{E}{\rightarrow} u_0$
proving that the family $\{ L_\epsilon^{-1} f^{\epsilon} \}_{\epsilon > 0}$ is E-precompact.

Finally, we have to show that $L^{-1}_\epsilon \stackrel{EE}{\rightarrow} L^{-1}_0$. For this, let us suppose 
\begin{equation} \label{FEC}
f^\epsilon \stackrel{E}{\rightarrow} f_0.
\end{equation}
Due to \eqref{FHD} and \eqref{FEC}, we have for any $\varphi \in L^2(0,1)$ that
$$
\int_{\Omega^\epsilon} \left\{ f^\epsilon(x_1,x_2) - f_0(x_1) \right\} \varphi(x_1) \, dx_1 dx_2 
= \int_0^1 \left\{ \hat f^\epsilon(x) - \left( H_\epsilon(x) + G_\epsilon(x) \right) f_0(x) \right\} \varphi(x) \, dx \to 0,
$$
as $\epsilon \to 0$.
Hence, since $\left( H_\epsilon(x) + G_\epsilon(x) \right) f_0(x) \rightharpoonup p(x) f_0(x)$, $w^* - L^\infty(0,1)$, see Remark \ref{RLP}, we can conclude $\hat f^\epsilon(x) \rightharpoonup p(x) f_0(x)$, $w^* - L^\infty(0,1)$.
Thus, it follows from Theorem \ref{ET} and Remark \ref{RLP} that $L^{-1}_\epsilon f^\epsilon \to L^{-1}_0 f_0$, and then $L^{-1}_\epsilon \stackrel{EE}{\rightarrow} L^{-1}_0$ as $\epsilon \to 0$.
\end{proof}

Now, let us take the positive coefficient $p(x)$ from \eqref{RPFL} and consider the operator $M_\epsilon: L^r(\Omega^\epsilon) \mapsto L^r(0,1)$, $1 \le r \le \infty$, given by
$$
\begin{gathered}
(M_\epsilon f^\epsilon)(x) = \frac{1}{p(x)} \int_{-G_\epsilon(x)}^{H_\epsilon(x)} f^\epsilon(x, s) \, ds  \quad x \in (0,1).
\end{gathered}
$$
It is easy to see that $M_\epsilon$ is a well-defined bounded linear operator with 
\begin{equation} \label{MEE}
\| M_\epsilon f^\epsilon \|_{L^p(0,1)} \le C \| f^\epsilon \|_{L^p(\Omega^\epsilon)}
\end{equation}
for some $C>0$ depending only on $r$, $G_0$, $H_0$, $G_1$ and $H_1$.
A similar operator was considered in \cite{ACL1,ACL2}. We still note that $M_{\epsilon}$ is a multiple of operator $\hat f$ defined by expression \eqref{FHD}.

Under this setting we still can point out to Theorem \ref{ET} showing the following result:
\begin{lemma} \label{ANTERIOR}
Let $\{ f^\epsilon \} \subset Z_\epsilon$  be a sequence and suppose that $\| f^\epsilon \|_{Z_{\epsilon}} \leqslant C$, for some $C$ independent of $\epsilon$.
Then, there exists a subsequence such that
$$
\| L^{-1}_\epsilon f^\epsilon - E_\epsilon L^{-1}_0 M_\epsilon f^\epsilon \|_{Z_\epsilon} 
\to 0 \textrm{ as } \epsilon \to 0. 
$$
\end{lemma}
\begin{proof}
Since $f^\epsilon$ is uniformly bounded in $L^2(\Omega^\epsilon)$, and $M_\epsilon$ is a bounded operator, we can extract a subsequence such that $M_\epsilon f^\epsilon \rightharpoonup f_0$, w-$L^{2}(0,1)$, for some $f_0 \in L^2(0,1)$. 
Then, from Theorem \ref{ET} and Remark \ref{RLP}, we have $\| L^{-1}_\epsilon f^\epsilon -  L^{-1}_0 f_{0}\|_{L^{2}(\Omega^\epsilon)} \to 0$, as $\epsilon \to 0$. Finally, the continuity of operator $L^{-1}_0$ implies the desired result.
\end{proof}

As a consequence of Lemma \ref{ANTERIOR}, we get the main result of this section, namely, the convergence of the resolvent operators of $L_\epsilon$ and $L_0$.
\begin{corollary} \label{eq:resolconver}
There exist $\epsilon_0 > 0$, and a function $\vartheta: (0,\epsilon_0) \mapsto (0,\infty)$, 
with $\vartheta(\epsilon) \to 0$ as $\epsilon \to 0$, such that
$$
\| L^{-1}_\epsilon - E_\epsilon L^{-1}_0 M_\epsilon \|_{\mathcal{L}(Z_\epsilon)} \le \vartheta(\epsilon), 
\quad \forall \epsilon \in (0,\epsilon_0).
$$
\end{corollary}
\begin{proof}
Let us show it by contradiction.
To do so, suppose there exist a $\delta >0$ and sequences $\{ \epsilon_n \}_{n \in \N} \subset (0,\infty)$, $\epsilon_n \to 0$ as $n \to \infty$, and $\{ f^{n} \}_{n \in \N} \subset Z_{\epsilon_n}$ 
with $\| f^{n} \|_{Z_{\epsilon_n}} = 1$, such that  
$$
\| L^{-1}_{\epsilon_n} f^{n}- E_{\epsilon_n} L^{-1}_0 M_{\epsilon_n} f^{n} \|_{Z_{\epsilon_n}} \geqslant \delta, \quad \textrm{ for all } n \in \N.
$$
On the other hand, from Lemma \ref{ANTERIOR} we can extract a subsequence satisfying  
$$\| L^{-1}_{\epsilon_{n_i}} f^{n_i} - E_{\epsilon_{n_i}} L^{-1}_0 M_{\epsilon_{n_i}} f^{n_i} \|_{Z_{\epsilon_{n_i}}} \stackrel{i\to\infty}{\longrightarrow} 0$$  which give us a contradiction completing the proof.
\end{proof}

\begin{remark} \label{ECONT}
Note that Corollary \ref{CCAE} implies that $L_\epsilon$ satisfies the following condition 
\begin{center}
\emph{\bf (C)} $L_\epsilon$ is a closed operator, has compact resolvent, the number zero belongs to its resolvent set $\rho(L_\epsilon)$ for all $\epsilon \in [0,1]$, and $L^{-1}_\epsilon \stackrel{CC}{\rightarrow} L^{-1}_0$.  
\end{center}
It is known that the spectrum of $L_\epsilon$ or $L_0$, denoted by $\sigma(L_\epsilon)$ or $\sigma(L_0)$, consists only of isolated eigenvalues.
Hence, if we consider an isolated point $\lambda_{0} \in \sigma(L_0)$ and its generalized eigenspace $W(\lambda_{0},L_0) = Q(\lambda_0,L_0) Z_0$, where
$$
Q(\lambda_{0},L_0) = \frac{1}{2 \pi i} \int_{S_\delta} (\xi \, I - L_0)^{-1} d\xi,
$$
$S_\delta = \{ \xi \in \C \, | \, |\xi - \lambda_{0}| = \delta \}$
and $\delta$ is chosen small enough such that there is no other point of $\sigma(L_0)$ in the disc 
$\{ \xi \in \C \, | \, |\xi - \lambda_{0}| \le \delta \}$, then, by condition {\bf (C)} and \cite[Lemma 4.9]{ACL1}, we have that there exists $\epsilon_0 > 0$ such that 
$\rho(L_\epsilon) \supset S_\delta$ for all $\epsilon \in (0, \epsilon_0)$.
Thus, we can denote by $W(\lambda_{0},L_\epsilon) = Q(\lambda_{0},L_\epsilon) Z_\epsilon$ where
$$
Q(\lambda_{0},L_\epsilon) = \frac{1}{2 \pi i} \int_{S_\delta} (\xi \, I - L_\epsilon)^{-1} d\xi.
$$
\end{remark}

\begin{remark} \label{Tspecconv} 
Moreover, it follows from condition {\bf (C)} and \cite[Lemma 4.10]{ACL1} the following statements about spectrum convergence of operators $L_\epsilon$:
\begin{enumerate}
\item[(i)] For any $\lambda_0 \in \sigma(L_0 )$, there is a sequence $\lambda_{\epsilon} \in \sigma(L_\epsilon)$, such that $\lambda_\epsilon \to \lambda_0$ as $\epsilon \to 0$.
\item[(ii)] If  $\lambda_\epsilon \to
\lambda_0$, with $\lambda_\epsilon \in \sigma(L_\epsilon)$, then
$\lambda_0 \in \sigma(L_0 )$.
\item[(iii)] There is $\epsilon_0>0$ such that $\dim{W(\lambda_{0},L_\epsilon)}=
\dim{W(\lambda_{0},L_0 )}$ for all $0 < \epsilon \leqslant \epsilon_0$.
\item[(iv)] For any $u \in W(\lambda_0,L_0 )$, there is a sequence ${u^{\epsilon}} \in W(\lambda_0,L_\epsilon)$, such
that ${u^{\epsilon}}
\overset{E}{\longrightarrow} u$.
\item[(v)] If $u^\epsilon \in W(\lambda_{0},L_\epsilon)$
satisfies $\norma{u^\epsilon}_{Z_{\epsilon}}=1$, then $\set{u^\epsilon}$ has an $E$-convergent subsequence and any limit point of this sequence belongs to $W(\lambda_0,L_0 )$.
\end{enumerate}
Finally, we note that the first eigenvalue of $L_\epsilon$ and $L_0$ is $1$ and its associated normalized eigenfunction is the constant $|\Omega^\epsilon|^{-1/2} \to ( \int_0^1 p(x) \, dx )^{-1/2}$ as $\epsilon \to 0$ by Remark \ref{RLP}. 

\end{remark}


Now we are in condition to discuss the convergence properties of the linear semigroups generated by the operators $L_\epsilon$ and $L_0$ considered in \eqref{AEPS} and \eqref{AO} respectively. 
We proceed here as the authors in \cite{ACL3, ACPS}.
Using standard arguments discussed for example in \cite{Pazzy:92}, it is easy to see that there exists $\epsilon_0 > 0$ such that the numerical range of the operators $-L_\epsilon$ 
are contained in $(-\infty, -1] \subset \C$ for all $\epsilon \in (0,\epsilon_0)$. 
Thus, we get from \cite[Theorem 3.9]{Pazzy:92} that there exists $M>0$ and $\frac{\pi}{2}<\phi<\pi$,
independent of $\epsilon$, such that
\begin{equation}\label{Eunif-lp-est}
\|\left(\mu + L_\epsilon \right)^{-1}\|_{\mathcal{L}(Z_{\epsilon})} \leqslant \frac{M}{|\mu + 1|},
\quad \forall \mu \in \Sigma_{-1 , \phi},
\end{equation}
where $\Sigma_{-1 , \phi}=\{\mu \in \C \, | \,  0<|\text{\rm{arg}}(\mu + 1)|\leqslant \phi \}$. 
We are setting here $Z_{\epsilon}$ by $Z_0$ as $\epsilon=0$. 
Hence, the operators $L_\epsilon$ are sectorial operators for all $\epsilon \in [0,\epsilon_{0}]$, with uniform estimates in $\epsilon$ for the resolvent operators $\left(\mu - L_\epsilon \right)^{-1}$ on the sector $\C\backslash\Sigma_{1 , \pi-\phi}$.  

We also get from Remark \ref{ECONT} that, if $\lambda \in \rho(L_0)$, there exists $\epsilon_0 > 0$ such that $\lambda \in \rho(L_\epsilon)$ for all $0 \leqslant \epsilon < \epsilon_0$, and so, we can use the resolvent identity given by \cite[Lemma 3.5]{ACL3} to obtain 
$$
(\lambda -  L_\epsilon)^{-1} - E_\epsilon (\lambda - L_0)^{-1} M_\epsilon = [I - \lambda (\lambda - L_\epsilon)^{-1}] [E_{\epsilon} L_0^{-1} M_\epsilon - L_\epsilon^{-1}]
[I - \lambda E_\epsilon (\lambda - L_0)^{-1} M_\epsilon].
$$
Consequently, since \eqref{Eunif-lp-est} implies 
$$
\begin{gathered}
\| I - \lambda (\lambda - L_\epsilon)^{-1} \|_{\mathcal{L}(Z_\epsilon)} \le 1 + M, \\
\| I - \lambda E_\epsilon (\lambda - L_0)^{-1} M_\epsilon \|_{\mathcal{L}(Z_\epsilon)} \le 1 + \| E_\epsilon \| \, \| M_\epsilon \| \, M, 
\end{gathered}
$$
we have by Corollary \ref{eq:resolconver} that there exists $\vartheta:(0,\epsilon_0)\to R^+$, 
$\vartheta(\epsilon)\to 0$ as $\epsilon \to 0$, such that
\begin{equation}\label{eq:remark}
\|(\lambda - L_\epsilon)^{-1} - E_\epsilon (\lambda - L_0)^{-1} M_\epsilon \|_{\mathcal{L}(Z_{\epsilon})} \leqslant \vartheta(\epsilon). 
\end{equation}

Moreover, if $\{ \mathrm{e}^{-L_\epsilon t} \, | \,  t\geqslant 0 \}$ denote the exponentially decaying analytic semigroup  in $Z_\epsilon$ generated by the sectorial operator $L_\epsilon$, then we obtain from \cite[Theorem 1.4.3]{Henry} that for any $0< \omega < 1$, there exists a constant $C=C(\omega)$, independent of $\epsilon$, such that
\begin{equation}\label{unif-exp-decay}
\| \mathrm{e}^{-L_\epsilon t} \|_{\mathcal{L}(Z_{\epsilon},Z_{\epsilon}^\alpha)} \leqslant  C \, t^{-\alpha}\, \mathrm{e}^{-\omega t} 
\textrm{ for all } t>0, \; 0\leqslant \alpha \leqslant 1 \textrm{ and } 0 \leqslant \epsilon \leqslant \epsilon_{0}.
\end{equation}

Finally, the continuity of resolvent operators allow us to obtain the continuity of linear semigroups associated to the family of sectorial operators $\{ L_\epsilon \}_{\epsilon \geq 0}$ in appropriated spaces.

\begin{theorem}\label{Tconvsgl} 
Suppose $0 \leqslant \alpha < \frac{1}{2}$. Then there exists a function $\vartheta_\alpha:(0, \epsilon_0] \mapsto (0,\infty)$, 
$\vartheta_\alpha(\epsilon) \to 0$, as $\epsilon \to 0$, such that
$$
\| e^{-{L_\epsilon} t} - E_{\epsilon} e^{-L_0 t} M_\epsilon \|_{\mathcal{L}(Z_{\epsilon},Z_{\epsilon}^{\alpha})} 
\leqslant  \vartheta_ \alpha(\epsilon) e^{-\omega t} t^{\alpha -1}, \quad   \textrm{ for all } t > 0.
$$

Consequently, there exists a constant $K>0$, independent of $\epsilon$, such that
$$
\|P_{\epsilon} e^{-{L_\epsilon} t} - e^{-L_0 t} M_\epsilon \|_{\mathcal{L}(L^2(\Omega^\epsilon),H^{2\alpha}(\widetilde \Omega^\epsilon))} 
\leq K  \vartheta_\alpha(\epsilon) e^{-\omega t} t^{{\alpha-1}}  
\textrm{ for all } t > 0.
$$
\end{theorem}
\begin{proof} 
For any sectorial operators as ${L_\epsilon}$, it is known that for any $0<\bar\omega< 1$    
$$
e^{(-{L_\epsilon} + \bar\omega I)t} = \frac{1}{2\pi i}
\underset{\tilde\Gamma}\int e^{(\mu + \bar\omega) t} (\mu+\bar\omega
+ {L_\epsilon} - \bar\omega )^{-1} \mathrm{d}\mu, 
$$
where $\tilde\Gamma$ is the oriented border of the sector $\Sigma_{-1,\phi}=\{\mu \in \C: |\mathrm{arg}(\mu+1)|\leq \phi \}$, $\frac{\pi}{2} <\phi< \pi$, such that the imaginary part of $\mu$ increases when $\mu$ describes the curve $\tilde\Gamma$.
We perform a changing of variable $\mu+\bar\omega \mapsto \mu$ and call  
${B_{\epsilon}}  := {L_\epsilon} - \bar\omega$ in order to evaluate 
\begin{equation}\label{intsgl1} 
2 \pi  \| e^{-B_{\epsilon} t} u^\epsilon -  E_{\epsilon} e^{-B_0 t} M_\epsilon u^\epsilon \|_{Z_{\epsilon}^{\alpha}} 
= \left\| \underset{\Gamma_0} \int e^{\mu t} [(\mu+{B_\epsilon})^{-1}  u^\epsilon -
 E_{\epsilon} (\mu+B_0)^{-1} M_\epsilon u^\epsilon] \mathrm{d} \mu \right\|_{Z_{\epsilon}^{\alpha}}
\end{equation}
where $\Gamma_0$ is the border of $\Sigma_{0,\phi}$.
For this, let us first collect some estimates involving $B_\epsilon$.

Due to \eqref{Eunif-lp-est}, we get for all $\mu \in \Gamma_0$ and $\epsilon \in [0,\epsilon_0]$ that $\| (\mu+{B_\epsilon})^{-1}\|_{\mathcal{L}(Z_{\epsilon})} \leq {\frac{C}{|\mu|}}$, and then,    
\begin{eqnarray}\label{eq:mu1}
\| (\mu+{B_\epsilon})^{-1} u^\epsilon - E_{\epsilon} (\mu+B_0)^{-1} M_\epsilon u^\epsilon \|_{Z_{\epsilon}}  
& \leq & \frac{C + \|E_\epsilon\| \, \| M_\epsilon \|}{|\mu|} \|u\|_{Z_\epsilon} \nonumber \\
& \leq & \frac{C_1}{|\mu|} \|u\|_{Z_\epsilon}.
\end{eqnarray}
We also have that 
\begin{eqnarray*}
\| B_{\epsilon} (\mu+{B_\epsilon})^{-1} u^\epsilon \|_{Z_\epsilon} 
& = & \| (I-\mu (\mu+{B_\epsilon})^{-1}) u^\epsilon \|_{Z_\epsilon} \\
& \leq & \| u^\epsilon \|_{Z_\epsilon} + |\mu| \| (\mu+{B_\epsilon})^{-1} u^\epsilon \|_{Z_\epsilon}  \\
& \leq & (1 + C) \| u^\epsilon \|_{Z_\epsilon}.
\end{eqnarray*}

Now, using Moment's Inequality from \cite[Section 1.4]{Henry}, we get 
\begin{eqnarray*}
\|B_{\epsilon}^{1/2} (\mu+{B_\epsilon})^{-1} u^\epsilon \|_{Z_{\epsilon}} 
& \leq & \| (\mu+{B_\epsilon})^{-1} u^\epsilon \|_{Z_{\epsilon}}^{{1/2}} \,
\| (\mu+{B_\epsilon})^{-1} u^\epsilon \|_{Z_{\epsilon}^1}^{{1/2}}  \\
& \leq & \frac{C^{{1/2}}}{|\mu|^{{1/2}}} (1 + C)^{1/2} \|u^\epsilon\|_{Z_\epsilon} . 
\end{eqnarray*}
Consequently, since for each $u^\epsilon \in Z_{\epsilon}$, $(\mu+B_{0})^{-1} M_\epsilon u^\epsilon \in \mathcal{D}(L_0) \subset H^2(0,1)$, we also obtain, 
\begin{eqnarray*}
\|B_{\epsilon}^{1/2} E_{\epsilon}(\mu+B_{0})^{-1} M_\epsilon u^\epsilon \|_{Z_{\epsilon}} &
\leq & (H_1+G_1)^{1/2} \|B_{0}^{1/2} (\mu+B_{0})^{-1} M_\epsilon u^\epsilon \|_{Z_{0}}  \\
& \leq & (H_1+G_1)^{1/2} \frac{C^{1/2}}{|\mu|^{1/2}} (1 + C)^{1/2} \|M_\epsilon\| \, \|u^\epsilon\|_{Z_\epsilon}.
\end{eqnarray*}

Thus, we can conclude that  
\begin{equation}\label{eq:interest}
\| (\mu+{B_\epsilon})^{-1} u^\epsilon - E_{\epsilon} (\mu+B_0)^{-1} M_\epsilon u^\epsilon \|_{Z_{\epsilon}^{1/2}} \leq
\frac{C_2}{|\mu|^{1/2}} \|u^\epsilon \|_{Z_\epsilon}.
\end{equation} 

Next let us denote $x=(\mu+{B_\epsilon})^{-1} u^\epsilon -   E_{\epsilon} (\mu+B_0)^{-1} M_\epsilon u^\epsilon$. Again using Moment's Inequality
\begin{align*}
 \| x\|_{Z_{\epsilon}^\alpha} & \leq C_{3} \|x\|_{Z_{\epsilon}^{1/2}}^{2\alpha} \|x\|_{Z_{\epsilon}}^{1-2\alpha} .
\end{align*}

Therefore, due to estimates \eqref{eq:remark}, \eqref{eq:mu1} and \eqref{eq:interest}, we get for $0 \leq \alpha \leq 1/2$ that 
\begin{equation}\label{eq:stimaresol}
\| (\mu+{B_\epsilon})^{-1} - E_{\epsilon} (\mu+B_0)^{-1} M_\epsilon \|_{\mathcal{L}(Z_{\epsilon},Z_{\epsilon}^\alpha)}\leqslant \frac{C_{3} \,  \vartheta(\epsilon)^{(1 - 2 \alpha)}}{| \mu|^{\alpha}}.
\end{equation}

Now performing the change of variable $\beta = \mu t$ in the integral given by \eqref{intsgl1} we get
$$
\left\| \underset{\Gamma_0}\int e^\beta
\left[\left({\beta t^{-1}}+{B_{\epsilon}}\right)^{-1} E_{\epsilon} u -
E_{\epsilon} \left({\beta t^{-1}}+B_0
\right)^{-1} u \right]\frac{d\beta}{t} \right\|_{Z_{\epsilon}^{\alpha}}.
$$
Hence, it follows from \eqref{eq:stimaresol} that
\begin{align*}
& \;\Big\| t^{-1}
\underset{\Gamma_0} \int e^\beta \big[\left({\beta
t^{-1}}+{B_{\epsilon}}\right)^{-1} 
- E_{\epsilon} \left({\beta t^{-1}}+B_0\right)^{-1} M_\epsilon  \big]\mathrm{d} \beta \Big \|_{\mathcal{L}(Z_\epsilon,Z_{\epsilon}^{\alpha})}  \leq  C_{3} \, t^{{\alpha -1}} \vartheta(\epsilon)^{(1 - 2 \alpha)} 
\underset{\Gamma_0} \int \frac{|e^\beta|}{|\beta|^{\alpha}}
\mathrm{d} |\beta|,
\end{align*}
and then,
$$
\|e^{-{B_{\epsilon}} t}  - E_{\epsilon} e^{-{B_0} t} M_\epsilon \|_{\mathcal{L}(Z_\epsilon,Z_{\epsilon}^{\alpha})} 
\leq  C_{4} t^{{\alpha -1}} \vartheta(\epsilon)^{(1 - 2 \alpha)}, \quad t > 0.
$$

Consequently, for all $\alpha \in [0,{1/2})$ and $\omega \in (0,1)$, there exists a function $\vartheta_\alpha: (0,\epsilon_0] \to \R^+$ with $\vartheta_\alpha(\epsilon) \overset{\epsilon \to 0} \longrightarrow 0$ such that
$$
\| e^{-{L_\epsilon} t} - E_{\epsilon} e^{-L_0 t} M_\epsilon \|_{\mathcal{L}(Z_{\epsilon},Z_{\epsilon}^{\alpha})} 
\leq \vartheta_\alpha(\epsilon) e^{-\omega t} t^{{\alpha -1}} \textrm{ for all } t > 0.
$$

Finally, we conclude the proof noting Remark \ref{remark:exbound} implies the existence of $K$ such that
\begin{eqnarray} \label{eq:extsegl}
\|P_{\epsilon} e^{-{L_\epsilon} t} -  e^{-L_0 t} M_\epsilon  \|_{\mathcal{L} (Z_\epsilon,H^{2\alpha}(\widetilde \Omega^\epsilon))}  & = & \|P_{\epsilon} e^{-{L_\epsilon} t} - P_{\epsilon} E_{\epsilon} e^{-L_0 t} M_\epsilon \|_{\mathcal{L}(Z_\epsilon,H^{2\alpha}(\widetilde \Omega^\epsilon))} \nonumber \\
& \leqslant & \|P_{\epsilon}\|_{\mathcal{L}(Z_{\epsilon}^\alpha,H^{2\alpha}(\widetilde \Omega^\epsilon))} \| e^{-{L_\epsilon} t} - E_{\epsilon} e^{-L_0 t} M_\epsilon \|_{\mathcal{L}(Z_\epsilon,Z_\epsilon^\alpha)}  \nonumber \\
& \leqslant & K \| e^{-{L_\epsilon} t} - E_{\epsilon} e^{-L_0 t} M_\epsilon \|_{\mathcal{L}(Z_\epsilon,Z_\epsilon^\alpha)}.
\end{eqnarray}

\end{proof}

\begin{corollary}\label{Cconvsgl}
Suppose $0\leqslant \alpha < {1/2}$ and $u^{\epsilon} \stackrel{E}{\longrightarrow} u$. Then there is a
function $\vartheta:{(0, \epsilon_0]} \mapsto (0,\infty),\; \vartheta(\epsilon ) \to 0$, as $\epsilon \to 0$, such that
\begin{equation}\label{SGunif-continuity-compact}
\norma{e^{-{L_\epsilon} t} u^{\epsilon} - E_{\epsilon} e^{-L_0 t} u}_{Z_{\epsilon}^{\alpha}} 
\leq \vartheta(\epsilon) e^{-\omega t} t^{\alpha-1}, \quad \textrm{ for all } t > 0.
\end{equation}
\end{corollary}
\begin{proof}
It is a direct consequence of Theorem \ref{Tconvsgl}, and estimatives \eqref{unif-exp-decay} and \eqref{MEE}, since 
$$
\| e^{-L_\epsilon t} u^{\epsilon} - E_{\epsilon} e^{-L_0 t} u \|_{Z_{\epsilon}^{\alpha}} 
\leqslant \| e^{-L_\epsilon t} u^{\epsilon} - E_\epsilon e^{-L_0 t} M_{\epsilon } u^\epsilon \|_{Z_{\epsilon}^{\alpha}}  
+ \| E_\epsilon e^{-L_0 t} \left( M_{\epsilon } u^\epsilon - u\right) \|_{Z_{\epsilon}^{\alpha}}, 
$$
and $M_\epsilon u^\epsilon - u = M_\epsilon \left( u^\epsilon - E_\epsilon u \right)$.
\end{proof}


\section{Upper semicontinuity of attractors and the set of equilibria} \label{S-USC}

Let $f:\R \mapsto \R$ be a bounded $\mathcal{C}^2$-function with bounded derivatives up to second order also satisfying the dissipative condition \eqref{HF}. Let us also consider the perturbed thin domain $\Omega^\epsilon$ defined in \eqref{domain} by the functions $G_\epsilon$ and $H_\epsilon$ introduced in Section \ref{SBF}.

In the previous sections, we have studied the behavior of the linear parts of problem \eqref{RP} 
as $\epsilon$ tends to zero and we have proved results on the continuity of the linear semigroups associated to \eqref{RP} and \eqref{LRP}. 
It is known that under these growth and dissipative conditions the solutions of problems \eqref{RP} and \eqref{LRP} are globally defined, and so, we can associate to them the nonlinear semigroups $\{T_\epsilon(t) \, | \, t\geq 0\}$ and $\{T_0(t) \, | \, t\geq 0\}$, well defined in $H^{2 \alpha}(\Omega^\epsilon)$ and $H^{2 \alpha}(0,1)$ respectively, for all $0 \leq \alpha \leq 1/2$ and $t>0$.
These dynamical systems are gradient and possess a family of compact global attractors $\{ \mathscr{A}_{\epsilon} \, | \, \epsilon \in [0,\epsilon_0] \} $, $\mathscr{A}_{\epsilon} \subset Z_\epsilon$ and $\mathscr{A}_0 \subset Z_0$ which lie in more regular spaces, namely $L^\infty(\Omega^\epsilon)$ and $L^\infty(0,1)$. Also, we can rewrite \eqref{RP} and \eqref{LRP} in the abstract form 
$$
\begin{gathered}
\left\{
\begin{array}{l}
\dot{u}^\epsilon + L_\epsilon u^\epsilon = \hat f_\epsilon(u^\epsilon) \\
u^\epsilon(0) =  u_0^\epsilon \in Z_\epsilon^\alpha
\end{array}
\right.
\quad \textrm{ and } \quad 
\left\{
\begin{array}{l}
\dot{u} + L_0 u = \hat f_0(u) \\
u(0) =  u_0 \in Z_0^\alpha
\end{array}
\right.
\end{gathered}
$$ 
where $\hat f_\epsilon: Z_\epsilon^\alpha \mapsto Z_\epsilon: u^\epsilon \to f(u^\epsilon)$ is the Nemitsk\u{\i}i operator defined by $f$ (see \cite{ACB,HR}).

In this section, we are in condition to relate the continuity of the linear semigroups with the
continuity of the nonlinear semigroups using the variation of constants formula establishing at the end the upper semicontinuity of the family of attractors, as well as, the upper semicontinuity of the set of stationary states at $\epsilon=0$.

\begin{theorem} \label{USCont}
Suppose $0 \leqslant \alpha < {1/2}$, and let $u^\epsilon \in Z_\epsilon$ satisfying
\begin{equation} \label{LUE}
\| u^\epsilon \|_{Z_\epsilon} \leq C
\end{equation}
for some positive constant $C$ independent of $\epsilon$.

Then, for each $\tau > 0$, there exists a function 
$\bar \vartheta_\alpha:(0, \epsilon_0] \to (0,\infty),\; \bar \vartheta_\alpha(\epsilon) \to 0$, as $\epsilon \to 0$, such that
\begin{equation}\label{eq:nlsg}
\| T_\epsilon(t)u^\epsilon - E_\epsilon T_0(t) M_\epsilon u^\epsilon \|_{Z_\epsilon^\alpha} 
\leq \bar \vartheta_\alpha(\epsilon)  t^{{\alpha -1}} 
\end{equation}
for all $t \in (0, \tau)$.

Moreover, we have the family of attractors $\{ \mathscr{A}_\epsilon \, | \, \epsilon \in [0, \epsilon_0] \}$ of problems \eqref{RP} and \eqref{LRP} is upper semicontinuous at $\epsilon = 0$ in $Z_\epsilon^\alpha$, in the sense that
\begin{equation} \label{USCT}
\sup_{\varphi^\epsilon \in \mathscr{A}_\epsilon} 
\Big[ \inf_{\varphi \in \mathscr{A}_0} \left\{ \| \varphi^\epsilon - E_\epsilon \varphi \|_{Z_\epsilon^\alpha}  \right\} \Big]
\to 0, 
\textrm{ as } \epsilon \to 0.
\end{equation}
Also, if we call $\mathcal{E}_\epsilon$ the set of stationary states of problems \eqref{RP}, for $\epsilon \in (0,\epsilon_0]$, and \eqref{LRP}, for $\epsilon=0$, then the family of sets $\{ \mathcal{E}_\epsilon \, | \, \epsilon \in [0, \epsilon_0] \}$ is upper semicontinuous at $\epsilon =0$, that is, 
\begin{equation} \label{USCE}
\sup_{\varphi^\epsilon \in \mathcal{E}_\epsilon} 
\Big[ \inf_{\varphi \in \mathcal{E}_0} \left\{ \| \varphi^\epsilon - E_\epsilon \varphi \|_{Z_\epsilon^\alpha}  \right\} \Big]
\to 0, 
\textrm{ as } \epsilon \to 0.
\end{equation}

Consequently there exists a constant $K$ independent of $\epsilon$ such that
\begin{equation}\label{eq:extsenl}
\|P_{\epsilon} T_\epsilon(t)u^\epsilon - T_0(t) M_\epsilon u^\epsilon \|_{H^{2 \alpha}(\widetilde \Omega^\epsilon)}  \leq K \bar \vartheta_\alpha(\epsilon)  t^{{2\alpha -1}}
\end{equation}
for all $t \in (0, \tau)$ and all $0\leqslant \alpha < {1/2}$. 
Furthermore,
\begin{equation} \label{EUSCT1}
\sup_{\varphi^\epsilon \in \mathscr{A}_\epsilon} 
\Big[ \inf_{\varphi \in \mathscr{A}_0} \left\{ \| P_{\epsilon} \varphi^\epsilon -  \varphi \|_{H^{2\alpha}(\widetilde \Omega^\epsilon)}  \right\} \Big]
\to 0, 
\textrm{ as } \epsilon \to 0,
\end{equation}
and 
\begin{equation} \label{EUSCE}
\sup_{\varphi^\epsilon \in \mathcal{E}_\epsilon} 
\Big[ \inf_{\varphi \in \mathcal{E}_0} \left\{ \| P_{\epsilon} \varphi^\epsilon -  \varphi \|_{H^{2\alpha}(\widetilde \Omega^\epsilon)}  \right\} \Big]
\to 0, 
\textrm{ as } \epsilon \to 0.
\end{equation}
\end{theorem}
\begin{proof}
First we observe that \eqref{eq:extsenl}, \eqref{EUSCT1} and \eqref{EUSCE} follow from \eqref{eq:nlsg}, \eqref{USCT} and \eqref{USCE} arguing as in \eqref{eq:extsegl}. 
Next let us show \eqref{eq:nlsg}. 
Using the variation of constants formula  
$$
T_\epsilon(t)u^\epsilon = e^{-{L_\epsilon} t} u^\epsilon 
+ \int_0^t e^{-{L_\epsilon}(t-s)} \hat f_\epsilon ({T_\epsilon(s) u^\epsilon})\;ds, \quad  \textrm{ for } \epsilon \in [0, 1],
$$
we obtain 
$$
\begin{gathered}
\| {T_\epsilon(t) u_\epsilon} - E_\epsilon {T_0(t)M_\epsilon u^\epsilon} \|_{Z_\epsilon^\alpha}
 \leqslant \| e^{-{L_\epsilon} t}  u^\epsilon - E_\epsilon  e^{-L_0 t} M_\epsilon u^\epsilon \|_{Z_\epsilon^\alpha} \\
\quad \quad \quad \quad  
+ \int_0^t \| e^{-{L_\epsilon}(t-s)} \hat f_\epsilon({T_\epsilon(s) u^\epsilon}) - E_\epsilon e^{-L_0(t-s)} \hat f_0({T_0(s)M_\epsilon u^\epsilon}) \|_{Z_\epsilon^\alpha} ds.
\end{gathered}
$$
It follows from \eqref{SGunif-continuity-compact} that there exist $\epsilon_0 > 0$ and $\vartheta:(0,\epsilon_0] \mapsto (0,\infty)$,
$\vartheta \stackrel{\epsilon \to 0}{\rightarrow} 0$, such that
$$
\| e^{-{L_\epsilon} t} - E_{\epsilon} e^{-L_0 t} M_\epsilon \|_{\mathcal{L}(Z_\epsilon,Z_{\epsilon}^{\alpha})} 
\leq \vartheta(\epsilon) e^{-\omega t} t^{{\alpha -1}}, \textrm{ for } t > 0.
$$
Furthermore, we have
\begin{align*}
\int_0^t & \| e^{-{L_\epsilon}(t-s)} \hat f_\epsilon({T_\epsilon(s)u^\epsilon}) - E_\epsilon {e^{-L_0(t-s)}}
 \hat f_0({T_0(s)M_\epsilon u^\epsilon}) \|_{Z_\epsilon^\alpha} ds \\
& \leqslant \int_0^t 
\| \left( {e^{-{L_\epsilon}(t-s)}} - E_\epsilon {e^{-{L_0}(t-s)}} M_\epsilon \right) 
\hat f_\epsilon({T_{\epsilon}(s) u^\epsilon}) \|_{Z^\alpha_\epsilon} ds \\
& + \int_0^t
 \| E_\epsilon {e^{-{L_0}(t-s)}} \Big( M_\epsilon \hat f_\epsilon({T_{\epsilon}(s) u^\epsilon}) 
 - \hat f_0({T_0(s)M_\epsilon u^\epsilon}) \Big) \|_{Z^\alpha_\epsilon}ds.
\end{align*}

Since $u^\epsilon$ satisfies (\ref{LUE}) for all $\epsilon > 0$, $T_\epsilon$ is global defined, and $f$ is bounded function, we have that
$\{\hat f_\epsilon({T_{\epsilon}(s) u^\epsilon}) \in Z_\epsilon \, | \, s\in [0,t]\}$ is uniformly bounded.
Hence, we obtain by Theorem \ref{Tconvsgl} that there exists a constant $\hat C_1 = \hat C_1(\tau, C)$ such that
\begin{align*}
\int_0^t \| &\Big( {e^{-{L_\epsilon}(t-s)}}   - E_\epsilon {e^{-L_0(t-s)}} M_\epsilon \Big) 
\hat f_\epsilon({T_\epsilon(s) u^\epsilon})\|_{Z^\alpha_\epsilon}  ds \\
& \leqslant \int_0^t   \vartheta_\alpha(\epsilon) e^{-\omega (t-s)} (t-s)^{{\alpha -1}} \|\hat f_\epsilon({T_\epsilon(s)u^\epsilon})\|_{Z_{\epsilon}} ds \leqslant \hat C_1  \vartheta_\alpha(\epsilon) t^{\alpha - 1} \ \textrm{ for all } t \in (0,\tau).
\end{align*}

If $K$ is the uniform Lipschitz constant of the Nemitsk\u{\i}i operator $\hat f_\epsilon$, independent of $\epsilon$, we can use $E_{\epsilon} \hat f_0 = \hat f_\epsilon E_{\epsilon}$ and $M_\epsilon E_\epsilon = I$ to get
\begin{align*}
\int_0^t & \| E_\epsilon {e^{-{L_\epsilon}(t-s)}} 
\Big( M_\epsilon \hat f_\epsilon({T_\epsilon(s) u^\epsilon}) - \hat f_0({T_0(s) M_\epsilon u^\epsilon}) \Big) \|_{Z^\alpha_\epsilon}ds\\
& = \int_0^t \| E_\epsilon {e^{-{L_\epsilon}(t-s)}} M_\epsilon \Big( \hat f_\epsilon({T_\epsilon(s) u^\epsilon}) 
- \hat f_\epsilon (E_\epsilon {T_0(s)M_\epsilon u^\epsilon}) \Big) \|_{Z^\alpha_\epsilon}ds \\
& \leq \int_0^t \hat C_2 \, \| E_\epsilon \| \, \|M_\epsilon\| \, K e^{-w(t-s)} (t-s)^{-\alpha} 
\| T_\epsilon(s) u^\epsilon - E_\epsilon {T_0(s)M_\epsilon u^\epsilon} \|_{Z^\alpha_\epsilon},
\end{align*}
for some constant $\hat C_2=\hat C_2(w)$.
Hence,
\begin{equation*}
\varphi(t) \leqslant (1 + \hat C_1) \vartheta_\alpha(\epsilon)   t^{{\alpha -1}} 
+ \hat C_2 \, \| E_\epsilon \| \, \|M_\epsilon\| \, K \int_0^t\;(t- s)^{-\alpha} \varphi(s) \, ds
 \textrm{ on } (0,\tau),
\end{equation*}
where $\varphi(t):= e^{\omega t}\norma{{T_\epsilon(t) u^\epsilon} - E_\epsilon {T_0(t)M_\epsilon u^\epsilon}}_{Z^\alpha_\epsilon}$.
Thus, due to Gronwall's Inequality from \cite[Section 7.1]{Henry}, we get 
$$
\varphi(t) \leqslant  \hat C_3 \vartheta_\theta(\epsilon)   t^{{\alpha -1}}
$$
where $\hat C_3 = \hat C_3(\hat C_1, \hat C_2,  K, \tau, \|E_\epsilon \|, \|M_\epsilon\|)$ is a constant, and so,  \eqref{eq:nlsg} follows.

In order to show the upper semicontinuity of the attractors $\mathscr{A}_\epsilon$, we first note that by uniform $L^\infty(\Omega^\epsilon)$ bounds of the attractors given by \cite[Theorem 2.6]{ACB} and Remark \ref{Tspecconv}, we also obtain due to \eqref{MEE} that $\bigcup_{0 \le \epsilon \le \epsilon_0} M_\epsilon \mathscr{A}_\epsilon$ is a bounded set in $L^\infty(0,1)$.
Then, using the attractivity property of $\mathscr{A}_0$ in $Z_0$, we have that for any $\eta > 0$ there exists $\tau > 0$ such that 
$$
\inf_{\varphi \in \mathscr{A}_0} \| T_0(\tau) M_\epsilon \varphi^\epsilon - \varphi \|_{Z_0^\alpha} \leq (H_1+G_1)^{-1/2} \eta/2, \quad \forall \varphi^\epsilon \in  \mathscr{A}_\epsilon \textrm{ and } 0 \leq \epsilon \leq \epsilon_0.
$$ 
Thus 
$$
\inf_{\varphi \in \mathscr{A}_0} \| E_\epsilon T_0(\tau) M_\epsilon \varphi^\epsilon - E_\epsilon \varphi \|_{Z_\epsilon^\alpha} \leq \eta/2, \quad \forall \varphi^\epsilon \in  \mathscr{A}_\epsilon \textrm{ and } 0 \leq \epsilon \leq \epsilon_0.
$$ 

Now, due to the convergence of the nonlinear semigroups \eqref{eq:nlsg} with $t = \tau$, we have that there exists $\epsilon_1 > 0$ such that for all $0 \leq \epsilon \leq \epsilon_1$
$$
\| T_\epsilon(\tau)\varphi^\epsilon - E_\epsilon T_0(\tau) M_\epsilon \varphi^\epsilon \|_{Z_\epsilon^\alpha} \leq \eta/2, \quad \forall \varphi^\epsilon \in \mathscr{A}_\epsilon.
$$
Consequently, since $\mathscr{A}_\epsilon$ is an invariant set by the flow, $T_\epsilon(\tau) \varphi^\epsilon = \varphi^\epsilon$, and so, we get 
$$
\inf_{\varphi \in \mathscr{A}_0} \| \varphi^\epsilon - E_\epsilon \varphi \|_{Z_\epsilon^\alpha} \leq \eta, \quad \forall \varphi^\epsilon \in  \mathscr{A}_\epsilon \textrm{ and } 0 \leq \epsilon \leq \epsilon_1.
$$


Finally, we show the upper semicontinuity of the set of stationary states $\mathcal{E}_\epsilon$. Let us use here the characterization discussed in \eqref{USCC}.
First, note $u^\epsilon \in \mathcal{E}_\epsilon$ if only if satisfies 
\begin{equation} \label{FFNL}
\int_{\Omega^\epsilon} \Big\{ \frac{\partial u^\epsilon}{\partial x_1} \frac{\partial \varphi}{\partial x_1} 
+ \frac{1}{\epsilon^2} \frac{\partial u^\epsilon}{\partial x_2} \frac{\partial \varphi}{\partial x_2}
+ u^\epsilon \varphi \Big\} dx_1 dx_2 = \int_{\Omega^\epsilon} f(u^\epsilon) \varphi \, dx_1 dx_2,  \quad \forall \varphi \in H^1(\Omega^\epsilon).
\end{equation}
Hence, substituting $\varphi = u^\epsilon$ in \eqref{FFNL}, we get  
$$
\begin{gathered}
\Big\| \frac{\partial u^\epsilon}{\partial x_1} \Big\|_{L^2(\Omega^\epsilon)}^2
+ \frac{1}{\epsilon^2}\Big\| \frac{\partial u^\epsilon}{\partial x_2} \Big\|_{L^2(\Omega^\epsilon)}^2
+ \| u^\epsilon \|_{L^2(\Omega^\epsilon)}^2
\le \| f(u^\epsilon) \|_{L^2(\Omega^\epsilon)} \| u^\epsilon \|_{L^2(\Omega^\epsilon)},
\end{gathered}
$$
Thus, since $f \in \mathcal{C}^2(\R,\R)$, there exists $C=C(f) > 0$, independent of $\epsilon > 0$, such that
$$
\begin{gathered}
\| u^\epsilon\|_{Z_\epsilon^{1/2}} \le C.
\end{gathered}
$$

Therefore, we obtain from \ref{eq:nlsg} that there exists $u_0 \in \mathcal{E}_0$, as well as a subsequence $u^\epsilon \in \mathcal{E}_\epsilon$ with $\|  u^\epsilon - E_\epsilon u_0 \|_{Z_\epsilon^{\alpha}} \to 0$, as $\epsilon \to 0$, for all $0 \leq \alpha < 1/2$. 
Indeed, since $T_\epsilon(t) u^\epsilon = u^\epsilon$ for each $t>0$, we have 
\begin{equation} \label{CUE}
\| u^\epsilon - E_\epsilon T_0(t) M_\epsilon u^\epsilon \|_{Z^\alpha_\epsilon} \to 0, 
\quad \textrm{ as } \epsilon \to 0,
\end{equation}
and then, $T_0(t) M_\epsilon u^\epsilon = M_\epsilon u^\epsilon$ for each $t>0$ implying that the uniformly bounded sequence $\{ M_\epsilon u^\epsilon \}_{ \epsilon>0 } \subset Z_0$ is $E$-convergent satisfying \eqref{CUE}. 
Notice that we can take $u_0 \in Z_0$ as a limit from $\{ M_\epsilon u^\epsilon\}_{\epsilon>0} \subset Z_0$.
Let us show now that $u_0 \in \mathcal{E}_0$. Using once more $T_\epsilon(t) u^\epsilon = u^\epsilon$ for any $t>0$, we have 
$$
\| u^\epsilon - E_\epsilon T_0(t)  u_0 \|_{Z^\alpha_\epsilon} 
= \| T_\epsilon(t)  u^\epsilon - E_\epsilon T_0(t)  u_0 \|_{Z^\alpha_\epsilon} \to 0, 
\quad \textrm{ as } \epsilon \to 0,
$$ 
for any $t>0$. Thus $T_0(t)  u_0 = u_0$ for all $t>0$ and $u_0 \in \mathcal{E}_0$ completing the proof.

\end{proof}


\end{document}